\documentclass[12pt]{article}
\usepackage[titletoc,toc,title]{appendix}
\usepackage{amsfonts,amssymb, bm}
\usepackage[a4paper, margin=2.7cm]{geometry}
\usepackage{graphicx,color,epstopdf}
\usepackage{amsmath, morefloats, bm}
\newtheorem{prop}{Proposition}[section]
\newtheorem{mylemma}{Lemma}[section]
\newtheorem{mytheorem}{Theorem}[section]
\newtheorem{mycorollary}{Corollary}[section]
\newtheorem{myremark}{Remark}[section]
\newtheorem{proof}{Proof}[section]

\def\XXint#1#2#3{{\setbox0=\hbox{$#1{#2#3}{\int}$}
		\vcenter{\hbox{$#2#3$}}\kern-.5\wd0}}

\def\bi{{\bf i}}
\def\ol#1{\overline{#1}}
\def\Log{{\rm Log}}
\def\eps{{\epsilon}}

\def\mCpp{{\mathbb{C}^{++}}}
\def\mCpm{{\mathbb{C}^{+-}}}
\def\mCmp{{\mathbb{C}^{-+}}}

\def\tx{{\tilde{x}}}
\def\ty{{\tilde{y}}}
\def\tM{{\tilde{M}}}

\makeatletter
\newcommand{\undersim}[1]{\mathrel{\mathpalette\@undersimG{#1}}}
\newcommand{\@undersim}[2]{%
  \vcenter{%
    \ialign{%
      ##\cr
      $\m@th#1#2$\cr
      \noalign{\nointerlineskip\kern.2ex}
      $\m@th#1\sim$\cr
      \noalign{\kern-.4ex}
    }%
  }%
}

\newcommand{\tu}{\tilde{u}}
\newcommand{\cb}[1]{{\color{black}#1}}

\newcommand{\dd}[1]{\left\langle #1\right\rangle}
\makeatother

\begin{document}
\title{On Wellposedness and Convergence of UPML method for analyzing wave scattering in layered media} \author{Wangtao Lu \thanks{School of
    Mathematical Sciences, Zhejiang University, Hangzhou 310027, China. Email:
    wangtaolu@zju.edu.cn.} \and Jun Lai\thanks{School of Mathematical Sciences,
    Zhejiang University, Hangzhou 310027, China. Email: laijun6@zju.edu.cn.}
  \and {Haijun Wu} \thanks{Department of Mathematics, Nanjing University,
    Nanjing, 210093, China. Email: hjw@nju.edu.cn. {This author was partially supported by the NSF of China under grants 11525103 and 91630309.}} } {

}
\maketitle

\begin{abstract}
  This paper \cb{proposes a novel method to} establish the wellposedness and
  convergence theory of the uniaxial-perfectly-matched-layer (UPML) method in
  solving a two-dimensional acoustic scattering problem due to a compactly
  supported source, where the medium consists of two layers separated by the
  horizontal axis. When perfectly matched layer (PML) is used to truncate the
  vertical variable only, the medium structure becomes a closed waveguide. The
  Green function due to a primary source point in this waveguide can be
  constructed explicitly based on variable separations and Fourier
  transformations. In the horizontal direction, by properly placing periodical
  PMLs and locating periodic source points imaged by the primary source point,
  the exciting waveguide Green functions by those imaging points can be
  assembled to construct the Green function due to the primary source point for
  the two-layer medium truncated by a UPML. Incorporated with Green's
  identities, this UPML Green function directly leads to the wellposedness of
  the acoustic scattering problem in a UPML truncation \cb{with no constraints
    about wavenumbers or UPML absorbing strength. Consequently, we firstly prove
    that such a UPML truncating problem is unconditionally resonance free.
    Moreover, we show, under quite general conditions, that this UPML Green
    function converges to the exact layered Green function exponentially fast as
    absorbing strength of the UPML increases, which in turn gives rise to the
    exponential convergence of the solution of the UPML problem towards the
    original solution.}


\end{abstract}
	
\section{Introduction}
	
Large amount of applications in optics (electromagnetics) and acoustics require
the accurate analysis of wave scattering in a layered medium. Examples include
optical waveguides, near field imaging, communication with submarine, detection
of buried objects and so on. As a result, the analysis and numerical computation
of layered medium scattering problems have been constantly attracting attentions
from researchers both in engineering and mathematical communities
\cite{Gang2011Imaging, Chew90, Jiang2016Quantitative, NovoHecht06,
  Bao2018HuYin, LuHu2019, LuLu2012, SongLu2015}.

 In this paper, we are concerned with two-dimensional time-harmonic acoustic scattering in a two-layered medium
\begin{align}\label{eq:org:helm}
\Delta u  + k(x)^2 u &= f,\quad{\rm in}\quad\mathbb{R}^2\backslash\Gamma,
\end{align}
where $f$ is the source term with a compact support $D\in\mathbb{R}^2$, and $u$
is the scattered field; see Figure~\ref{fig:model}(a).
\begin{figure}[!ht]
  \centering
  (a)\includegraphics[width=0.3\textwidth]{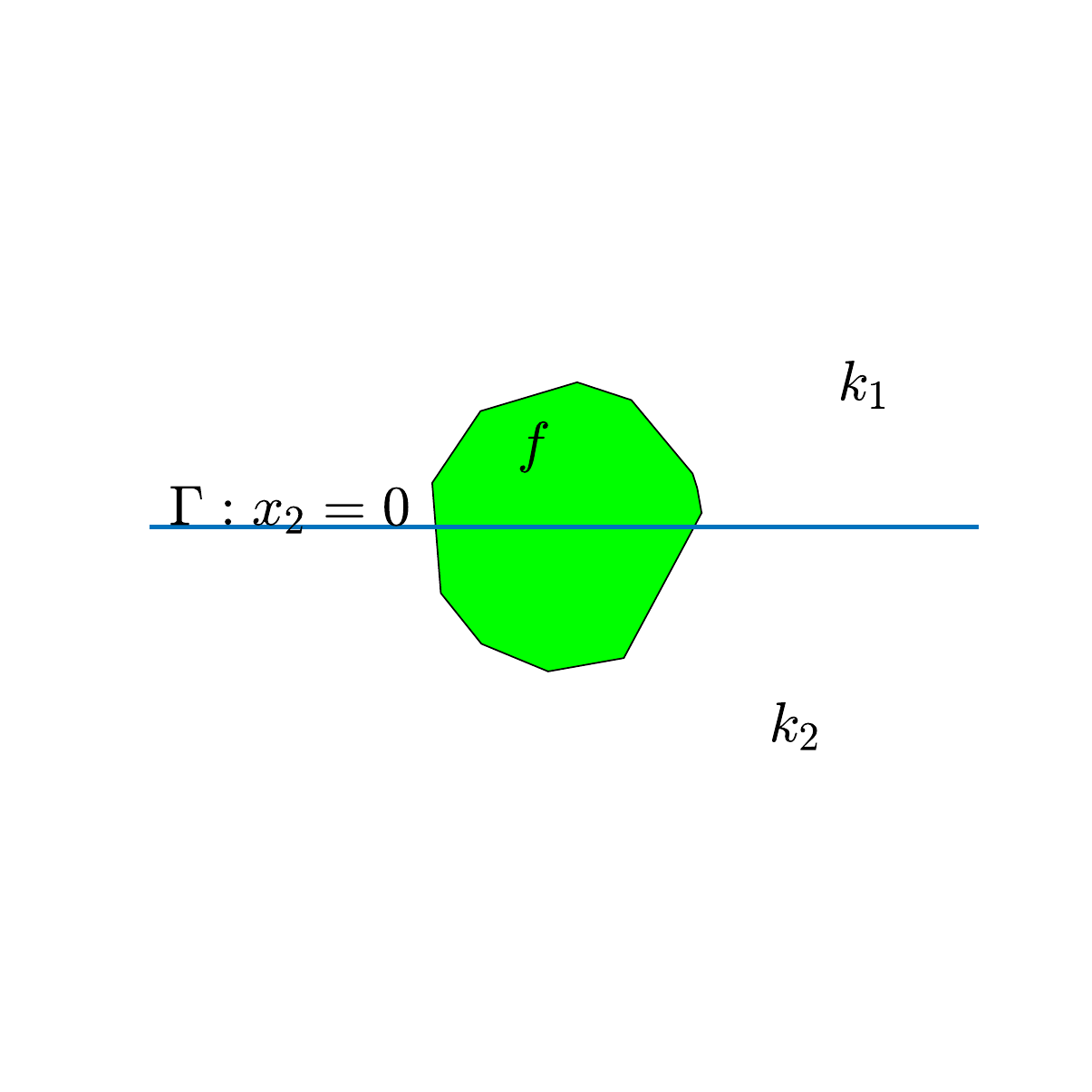}
  (b)\includegraphics[width=0.3\textwidth]{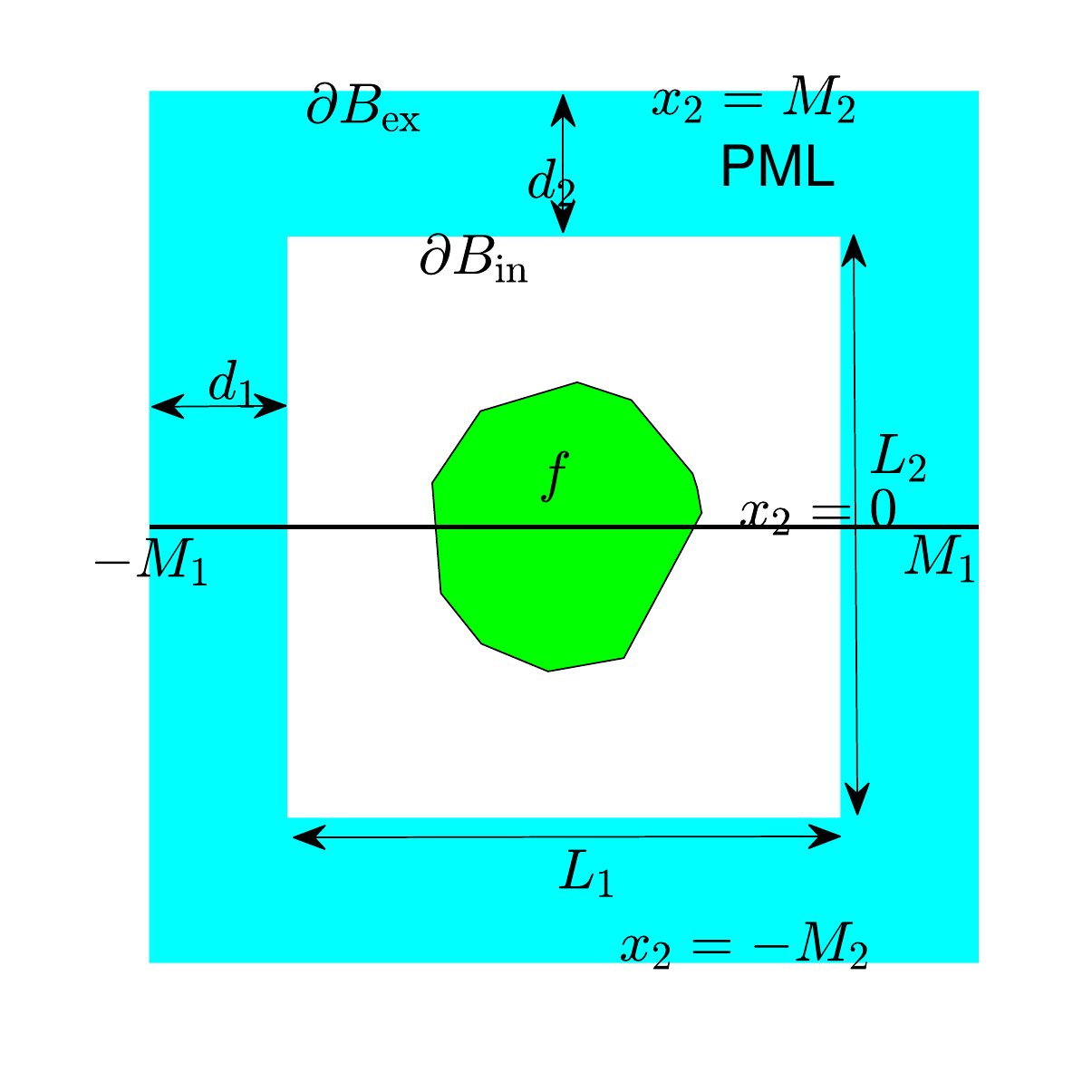}
  \vspace{-0.4cm}
  \caption{\cb{(a): The two-layer medium structure with compactly
    supported source $f$. (b): The UPML truncation.}} 
  \label{fig:model}
\end{figure}
Denote by $x=(x_1,x_2)$ the two dimensional {coordinates}. The interface $\Gamma$ is simply assumed to be the axis $x_2=0$, by which the domain $\mathbb{R}^2$ is divided into the upper half space $\mathbb{R}_+^2$ and lower half space $\mathbb{R}_-^2$. The wavenumber $k(x)$ takes the form
\begin{align}\label{eq:def:k}
k(x) = \left\{
\begin{array}{lc}
k_1,\quad x\in\mathbb{R}_+^2,\\
k_2,\quad x\in\mathbb{R}_-^2.
\end{array}
\right.
\end{align} 
where $k_1$ and $k_2$ are two arbitrary positive real number. We assume the field and flux are continuous across the interface $\Gamma$
\begin{align}
\label{eq:continuity}
[u]_{\Gamma} &= 0,\quad [\partial_n u]_{\Gamma} = 0,
\end{align}
where $[\cdot]$ denotes the jump across the interface $\Gamma$. The field $u$ also satisfies the Sommerfeld radiation condition at infinity
\begin{align} \label{eq:rad:cond}
 \lim_{|x|\rightarrow\infty} |x|^{1/2}\left( \partial_{|x|} u - \bi k(x) u \right) &= 0.
 \end{align}
  
Due to the important roles they play in applications, the layered medium scattering problems have been studied extensively in the literature. For the well-posedness, we refer readers to \cite{Bao2018HuYin, KG1980, Odeh63} for the acoustic scattering problems in a two-layered medium with locally perturbed interfaces and to \cite{CPHC98} for the layered  electromagnetic scattering problems.  Discussions on the inverse scattering problems in a layered medium can be found in \cite{Gang2011Imaging}. For numerical computation, given the infinite domain of Equation \eqref{eq:org:helm}, integral equation method is a natural candidate  since they discretize the support $D$ alone and
impose the Sommerfeld radiation condition by construction. In particular, if we give the layered medium Green's function $G_{k_1,k_2}(x,y)$, the scattered field $u$ can be found simply by an integration
\begin{align}\label{intrep}
u(x) =\int_{D}G_{k_1,k_2}(x,y)f(y)dy. 
\end{align}
However, in order to make effective use of this approach, one must generally
evaluate the governing Green's function $G_{k_1,k_2}(x,y)$ that satisfies the
continuity interface conditions \eqref{eq:continuity}. \cb{Using Fourier
  analysis, the closed form of $G_{k_1,k_2}(x,y)$ can be derived in terms of
  Sommerfeld integrals~\cite{sommerfeld}, which, however, is quite costly to
  evaluate \cite{cai}. Over the past decades, a number of methods have been
  developed to remedy this issue: literatures in \cite{caiyu,
    okhmatovski-2014,paulus_2000,Perez-Arancibia2014, koh, oneil-imped} aim at
  developing fast algorithms to efficiently evaluate such layered media Green
  functions; literatures in \cite{Bruno2016Windowed, LAI2018359} use windowed
  function methods to avoid the evaluation of layered media Green functions,
  etc.}

\cb{A more attactive approach is introducing a perfectly matched layer (PML) to
  truncate the domain so that some standard methods like finite-difference
  method or finite element method could apply.} The basic idea of PML, which was
first proposed by Berenger \cite{BERENGER1994185} in the 90s, is to truncate the
infinite computation domain by an artificial layer with zero boundary condition
in the exterior; see Figure~\ref{fig:model}(b). The layer has been specifically
designed to absorb all \cb{outgoing} waves propagating from the interior of the
computational domain. Due to the effectiveness of this method in computation,
considerable attentions have been paid to the convergence study. These include
the acoustic scattering problems by Lassas et al \cite{Lassas1998On,
  lassas_somersalo_2001}, Hohage et al \cite{Hohage2003Solving}, and Collino et
al \cite{CollMonk}, the grating problems with adaptive FEM by Chen et
al\cite{ChenWu03}, Bao et al \cite{baochenwu05}, the electromagnetic scattering
problems by \cite{baowu05, brampas12, Chen2013, LiWu2019, ZhouWu2018}, and the
elastic scattering problems by \cite{brampas10, chenxi15}. As they all focused
on the scattering within homogeneous background, analysis for the layered medium
scattering problem becomes much more complicated due to the lack of closed form
of the layered Green's function. Recently, great progress has been made for
two-layer media by Chen and Zheng \cite{chenzheng} for acoustic scattering
problems and \cite{Chen2017PML} for electromagnetic scattering problems based on
the Cagniard–de Hoop transform for the Green’s function; \cb{the first author
  and his collaborators in \cite{Lu2018Perfectly} developed a hybrid method of
  PML and boundary integral equation for numerically solving the two-layer
  scattering problem}. An important conclusion of these studies is that the PML
solution converges to the solution of the original scattering problem
exponentially fast as the PML absorbing strength increases.

  However, despite all these contributions, a fundamental question remains open:
  is the truncated PML problem always \cb{unconditionally resonance free}? In
  other words, for a scattering problem \cb{after a PML truncation}, is it
  always uniquely solvable for \cb{arbitrary positive wavenumber and arbitrary
    absorbing strength}. For acoustic scattering in homogeneous background,
  Collino and Monk showed that the truncated problem has a unique solution
  except at a discrete set of exceptional frequencies for PML in curvilinear
  coordinates \cite{CollMonk}. They conjectured that the exceptional set might
  be empty without a proof. \cb{On the other hand, for acoustic scattering in
    layered media, the authors in \cite{chenzheng} proved that the truncated
    problem has a unique solution when the PML absorbing strength is
    sufficiently large. } In this paper, we \cb{propose a novel} approach to show
  that the acoustic scattering problem \cb{in two-layer media} due to a compactly
    supported source is always resonance free with uniaxial PML (UPML)
  truncation.

When PML is used to truncate the vertical variable only, the medium structure
becomes a closed waveguide. The Green function due to a primary source point in
this waveguide can be constructed explicitly based on variable separations and
Fourier transformations. In the horizontal direction, if we properly place
periodical PMLs and carefully choose periodic source points imaged by the
primary source point, the exciting waveguide Green functions can be assembled to
construct the Green function due to the primary source point for the two-layer
medium truncated by a UPML. \cb{Incorporated with Green’s identities, this UPML
  Green function directly leads to the wellposedness of the acoustic scattering
  problem in a UPML truncation with no constraints about wavenumbers or UPML
  absorbing strength. Consequently, we firstly prove that such a UPML truncating
  problem is unconditionally resonance free. Moreover, we show, under quite
  general conditions, that the layered Green’s function with UPML truncation
  converges to the exact layered Green’s function exponentially fast as
  absorbing strength of the UPML increases, which in turn gives rise to the
  exponential convergence of the solution of the UPML problem to the original
  solution.}

The outline of paper is given as follows. Section 2 introduces the UPML
formulation and presents our two main results. Section 3 gives an explicit
construction of the layered Green's function in a UPML truncated domain, and
proves the wellposedness of the solution for the truncated layered scattering
problem with a compact source in this section. Section 4 is devoted to the
exponential convergence study for the solution of the layered medium problem
when absorbing strength of the UPML increases. We conclude the paper with a
brief discussion on the future work in Section 5.

\section{UPML formulation and main results}

We restrict our discussion to the layered medium with ratio of the wavenumber
$\kappa:=k_2/k_1>1$, as the analysis for $\kappa<1$ is the same by symmetry.
\cb{For the simplicity of notation, we shall frequently use $C$ for a generic
  positive constant, which, unless otherwise specified, is independent of all
  model parameters $k_1$ and $ k_2$, and the PML parameters $L_j, d_j,
  \sigma_j$. We will also often write $A\lesssim B$ and $B\gtrsim A$ for the
  inequalities $A \le CB$ and $B \ge CA$, respectively. $A\eqsim B$ is used for
  an equivalent statement when both $A\lesssim B$ and $B\lesssim A$ hold.} By
rescaling, we will also assume the compact support $D$ of the source term $f$ in
equation \eqref{eq:org:helm} is {enclosed in a disk centered at the origin of
  radius $R\lesssim1$}.

As shown in Figure~\ref{fig:model}(b), to truncate the scattering problem by UPML,
we introduce two rectangular boxes. One is the inner box $B_{\rm
  in}=[-L_1/2,L_1/2]\times [-L_2/2,L_2/2]$ {of sizes $L_j>0$, $j= 1, 2$, which
  we call the physical domain in the following.} Denote by $B^1_{\rm in}=B_{\rm
  in}\cap \mathbb{R}_+^2$ and $B^2_{\rm in}=B_{\rm in}\cap \mathbb{R}_-^2$ the
intersection of the box $B_{\rm in}$ with the upper and lower half space,
respectively. The other is the outer box $B_{\rm ex}=[-M_1,M_1]\times
[-M_2,M_2]$, with $M_j=L_j/2+d_j$ and $d_j>0$, $j = 1, 2$, which is the
computational domain we are concerned with. The parameter $d_j$ represents the
thickness of the uniaxial PML along the $x_j$ direction. Denote by $B^1_{\rm
  ex}=B_{\rm ex}\cap \mathbb{R}_+^2$ and $B^2_{\rm ex}=B_{\rm ex}\cap
\mathbb{R}_-^2$ the intersection of the box $B_{\rm ex}$ with the upper and
lower half space, respectively. For further discussion, we divide the complex
plane $\mathbb{C}$ into four regions, namely, $\mathbb{C}^{-+} =
\{z\in\mathbb{C}:{\rm Re}(z)<0,{\rm Im}(z)>0\}$,
$\mathbb{C}^{++}=\{z\in\mathbb{C}:{\rm Re}(z)>0,{\rm Im}(z)>0\}$,
$\mathbb{C}^{--}=\{z\in\mathbb{C}:{\rm Re}(z)<0, {\rm Im}(z)<0\}$,
$\mathbb{C}^{+-}=\{z\in\mathbb{C}: {\rm Re}(z)>0, {\rm Im}(z)<0\}$. We will
denote $a^+=\sqrt{a^2}$ where $\sqrt{\cdot}$ is defined to be in the branch with
nonnegative imaginary part.

\cb{Mathematically, UPML can be described by the following complex coordinate
  transformations \cite{cjm97} 
\begin{eqnarray}\label{eq:cct}
  \tilde{x}_j = x_j+\bi\int_0^{x_j}\sigma_j(t)dt = \int_0^{x_j}\alpha_j(t)dt, \quad x_j\in[-M_j,M_j],\quad
  j = 1,2,
\end{eqnarray} 
where $\sigma_j(t)$ is the absorbing function on $[-M_j,M_j]$, and the medium
function $\alpha_j = 1+\bi \sigma_j$. To simplify the presentation, we assume
that $\sigma_j, j=1,2$ are Lipschitz continuous and satisfy the following
conditions
\begin{equation}\label{condforabs}
\begin{cases}
\sigma_j(t)=0, \mbox{ for } t\in [-L_j/2,L_j/2],  \\
\sigma_j(t)\ge 0 \mbox{ and } \sigma_j(t)=\sigma_j(-t), \mbox{ for } t \in [-M_j,M_j]\backslash [-L_j/2,L_j/2], \\
\bar{\sigma}_j=\int_{-M_j}^{-L_j/2}\sigma_j(t)dt=\int_{L_j/2}^{M_j}\sigma_j(t)dt>0.
\end{cases}
\end{equation} 
Here, $\bar{\sigma}_j, j=1,2$ reflect absorbing strength of the UPML and are
called absorbing constants.
\begin{myremark}
  The whole results in this paper can be extended easily to when $\sigma_j$ is
  nonnegative and piecewise Lipschitz continuous and $\sigma_j(-t)\neq
  \sigma_j(t)$ for $j=1,2$. The condition that $\bar{\sigma}_j>0$ for $j=1,2$ in
  \eqref{condforabs} cannot be removed as otherwise transformations
  (\ref{eq:cct}) cannot be regarded as a UPML.
\end{myremark}
}

With the definition above, the original scattering problem with UPML truncation is formulated as {(see e.g. \cite{chenzheng})}
\begin{align}\label{eq:upml}
\begin{cases}
&\frac{\partial}{\partial x_1}\left( \frac{\alpha_2}{\alpha_1}\frac{\partial \tilde{u}}{\partial x_1} \right)+ \frac{\partial}{\partial x_2}\left( \frac{\alpha_1}{\alpha_2}\frac{\partial \tilde{u}}{\partial x_2} \right) + \alpha_1\alpha_2 k^2\tilde{u} = f \mbox{ in } B_{\rm ex},\\
&[\tilde{u}] = 0,\quad [\partial_{x_2} \tilde{u}] = 0 \mbox{ on } \Gamma_{\rm ex},\\
&\tilde{u} = 0 \mbox{ on } \partial B_{\rm ex}.
\end{cases}
\end{align}
where $\Gamma_{\rm ex} =\Gamma\cap B_{\rm ex}$, and $\tu$ is the scattered solution
with UPML truncation. The variational formulation of the above problem reads:
Find $\tu\in H^{1}_0(B_{\rm ex})$ such that
\begin{equation}
\label{eq:varform:a}
a(\tu,v):= ({\bf A}\nabla \tu,\nabla v)_{B_{\rm ex}} - (k^2\alpha_1 \alpha_2\tu,v)_{B_{\rm ex}} = -{\dd{f,v}_{B_{\rm ex}}\quad\forall v\in H^{1}_0(B_{\rm ex}), }
\end{equation}
 where ${\bf A}={\rm
  diag}(\alpha_2/\alpha_1,\alpha_1/\alpha_2)$. Here, {we denote by
  ${\dd{f,v}}_{B_{\rm ex}}$ the duality paring between $f\in H^{-1}(B_{\rm ex})$ and
  $v\in H_0^{1}(B_{\rm ex})$, as $ H^{-1}(B_{\rm ex})$ is the dual space of
  $H_0^{1}(B_{\rm ex})$, and by $(\cdot,\cdot)_{B_{\rm ex}}$ the usual
  $L^2-$inner product on $B_{\rm ex}$.}

Our first main result is concerned with the well-posedness of the variational equation \eqref{eq:varform:a}. 
\begin{mytheorem}\label{wellposethm}
	For any positive wavenumbers $k_2>k_1$, and any positive constants $d_j, L_j,
  and \bar{\sigma}_j$ for $j=1,2$, there exists, for any $f\in H^{-1}(B_{\rm ex})$,
  a unique solution $\tu\in H_0^1(B_{\rm ex})$ to the variational equation
  \eqref{eq:varform:a}.
\end{mytheorem}
\begin{myremark}
  In contrast with the previous well-posedness results in \cite{Lassas1998On,
    chenzheng, Chen2017PML} where absorbing constants $\bar{\sigma}_j$ are
  required to be sufficiently large to ensure no resonances could occur, \cb{our
    results allow $\bar{\sigma}_j$ to be any positive constant and affirmatively
    show that our problem is unconditionally resonance free. }
\end{myremark}

  Our second main result studies \cb{the error between the UPML solution $\tu$ and
  the original solution $u$ as PML parameters vary}. {In doing so, we need
    the following assumptions about the PML parameters:
\begin{align}
  \label{eq:assump1}
  \min(L_1,L_2)> 2R,\\
  \label{eq:assump2}
  \frac{\bar{\sigma}_1}{\bar{\sigma}_2}\eqsim\frac{L_1-2R}{L_2+2R}\eqsim\frac{L_1}{L_2}\eqsim\frac{d_1}{d_2}\eqsim 1,\\
  \label{eq:assump3}
  \min(L_j,d_j,\bar{\sigma}_j)\gtrsim \frac{1}{k_1}.
\end{align}
Here, (\ref{eq:assump1}) indicates that the PMLs should enclose the support of
source $f$, (\ref{eq:assump2}) indicates that domains $B_{\rm in}$ and $B_{\rm
  ex}$ should {not be too thin} 
and that absorbing constants $\bar{\sigma}_1$ and $\bar{\sigma}_2$ should be
comparable, and (\ref{eq:assump3}) indicates that any of size of physical domain
$B_{\rm in}$, thickness of the PML, and PML absorbing constants should at least
be comparable with the wavelength so that the wavefield $u$ can be absorbed to
some extent.}
\begin{mytheorem}\label{convergthm}
	For $f\in L^2(\mathbb{R}^2)$ with compact support $D$, under the assumptions
  (\ref{eq:assump1}-\ref{eq:assump3}), the PML solution $\tu$ converges to the
  original scattering solution $u$ exponentially fast as the absorbing constants
  $\bar{\sigma}_j$, $j=1,2$, increase. In particular, it holds,
	\begin{align}
    \label{eq:L2error}
	||\tu - u||_{L^2(B_{\rm in})}&\lesssim e^{-\gamma k_1\bar{\sigma}_1}L_1||f||_{L^2(D)},\\
    \label{eq:H1error}
    ||\tu - u||_{H^1(B_{\rm in})}&\lesssim k_1L_1e^{-\gamma k_1\bar{\sigma}_1}||f||_{L^2(D)}.
	\end{align}
	where $\gamma\eqsim 1$ will be defined later in Theorem~\ref{Greenestimate}.
\end{mytheorem}  

\begin{myremark}
  \cb{The method in \cite{chenzheng} can also be directly extended to our
    problem. However, some disadvantages may arise. First, \cite{chenzheng}
    makes some restrictions on the PML parameters including: $\sigma_j$ are
    piecewise constant, $d_1=d_2$, $\bar{\sigma}_1=\bar{\sigma}_2$, and
    $\bar{\sigma}_j$ are sufficiently large, etc.; it seems that those
    restrictions are not easy to release. Moreover, our error upper bound
    \eqref{eq:H1error} is sharper than Eq.~(7.8) in \cite{chenzheng} in the
    sense that in our upper bound (\ref{eq:H1error}), the constant in front of
    the exponent is proportional to $k_1L_1$ only, and that the decaying rate in
    the exponent is proportional to $k_1\bar{\sigma}_1$ only and is independent
    of $L_j$.}
\end{myremark}

The following sections are devoted to the proofs of Theorems \ref{wellposethm}
and \ref{convergthm}.

\section{Well-posedness of the scattering problem with UPML truncation} In order to study the well-posedness, we first show the the existence of Green's function for the UPML truncated layered medium scattering problem \begin{align}\label{eq:upml:green}
\begin{cases}
&\frac{\partial}{\partial x_1}\left( \frac{\alpha_2}{\alpha_1}\frac{\partial G_{\rm PML}}{\partial x_1} \right)+ \frac{\partial}{\partial x_2}\left( \frac{\alpha_1}{\alpha_2}\frac{\partial{G_{\rm PML}}}{\partial x_2} \right) + \alpha_1\alpha_2 k^2G_{\rm PML} = -\delta(x-y), \quad {x,y\in}   B_{\rm ex},\\
&[G_{\rm PML}] = 0,\quad [\partial_{x_2} G_{\rm PML}] = 0 \quad\mbox{ on }  {x_2=0},\\
&G_{\rm PML} = 0 \quad\mbox{ on } \partial B_{\rm ex}.
\end{cases}
\end{align}
The proof is by explicit construction of $G_{\rm PML}(x,y)$. Our idea is first
constructing the Green's function $G(x,y)$ for the waveguide problem
(\ref{eq:green:pmly}) where PMLs are placed above and below the interface
$\Gamma$ to terminate $x_2$, and then by placing periodic PMLs leftwards and
rightwards along $x_1$-direction and by introducing periodically distributed
source points we could construct $G_{\rm PML}$ by use of $G$ excited by those
source points. Subsections 3.1-3.4 are devoted to prove the existence of
$G(x,y)$; in Subsection 3.5, we explicitly construct $G_{\rm PML}$ by use of
$G$; in Subsection 3.6, we prove Theorem \ref{wellposethm} to address the
wellposedness of the UPML problem (\ref{eq:upml}). \cb{Throughout this section,
  we will assume that $k_j$, $\bar{\sigma}_j, d_j$, and $L_j$ are fixed positive
  constants as indicated in Theorem~\ref{wellposethm}; we emphasize that the
  generic constant $C$ used in notations $\lesssim$, $\gtrsim$, and $\eqsim$
  become dependent of those paramters within this section only.}

\subsection{Green's function for a waveguide problem}
The waveguide problem we consider is
\begin{align}\label{eq:green:pmly}
\begin{cases}
&\frac{\partial}{\partial x_1}\left( \alpha_2\frac{\partial G}{\partial x_1}
\right)+ \frac{\partial}{\partial x_2}\left( \frac{1}{\alpha_2}\frac{\partial
	G}{\partial x_2} \right) + \alpha_2 k^2G = -\delta(x-y),\mbox{ for }
x\in\mathbb{R}\times (-M_2,M_2),\\
&[G] = 0,\quad [\partial_{x_2} G] = 0,\quad{\rm on}\quad \Gamma,\\
&G = 0,\quad{\rm on}\quad x_2=\pm M_2.
\end{cases}
\end{align}
where $y=(y_1,y_2)$ denotes source point located in
{$\mathbb{R}\times (-M_2,M_2)$}. We require that the Green's function
for the waveguide problem satisfy the {Sommerfeld radiation
  condition}
\begin{equation}\label{eq:som}
\lim_{x_1\to\infty}\sqrt{|x_1-y_1|}\left(  \frac{\partial G}{\partial |x_1-y_1|} - ik(x) G\right) = 0,
\end{equation}
{We will first formally derive an explicit representation of $G$ in this subsection and then verify it does solve \eqref{eq:green:pmly}--\eqref{eq:som} in Subsections~\ref{subsec:A}--\ref{subsec:Gw}. L}et us formally take the Fourier transform of $G(x;y)$ along the variable $x_1$,
\begin{equation}\label{eq:def:hG}
\hat{G}(x_2;y_2,\xi) = \frac{1}{\sqrt{2\pi}}\int_{-\infty}^{\infty}G(x;y)e^{-\bi(x_1-y_1)\xi}dx_1.
\end{equation} 
For fixed $y_2\neq 0$ and $\xi$, $\hat{G}$ satisfies the equation
\begin{align}\label{eq:hG:pmly}
\begin{cases}
&\frac{d}{dx_2}\left(\frac{1}{\alpha_2}\frac{d \hat{G}}{d x_2} \right) + \alpha_2(k^2-\xi^2) \hat{G} =
-\frac{1}{\sqrt{2\pi}}\delta(x_2-y_2),\mbox{ for }
x_2\in (-M_2,M_2),\\
&[\hat{G}] = 0,\quad [\hat{G}'(x_2)] = 0,\quad{\rm on}\quad x_2=0,\\
&\hat{G} = 0,\quad{\rm on}\quad x_2 = \pm M_2.
\end{cases}
\end{align}
Let {$\Omega^1 = \mathbb{R}\times(0,M_2)$ and $\Omega^2 = \mathbb{R}\times (-M_2,0)$.}  Then,
by direct but tedious calculation, one gets the solution to problem
(\ref{eq:hG:pmly}) as follows: for $i=1,2$, if $x,y\in \Omega^{i}$, then
\begin{align}
\label{eq:hG:sol:++}
\hat{G}(x_2;y_2,\xi) =&\frac{B_1^{i}\bi}{2A\sqrt{2\pi}\mu_i}\left[  
e^{\bi\mu_i(4\tilde{M}_2 -\tilde{y}_2^+ - \tilde{x}_2^+)} -e^{\bi\mu_i(2\tilde{M}_2 -\tilde{y}_2^+ + \tilde{x}_2^+)}- e^{\bi\mu_i(2\tilde{M}_2 + \tilde{y}_2^+-\tilde{x}_2^+)}\right]\nonumber\\
&+\frac{\bi}{2\sqrt{2\pi}\mu_i}\left[\frac{B_2^{i}}{A(\mu_1+\mu_2)}+ \frac{\mu_i-\mu_{3-i}}{\mu_1+\mu_2}\right]e^{\bi \mu_i(\tilde{x}_2^++\tilde{y}_2^+)}\nonumber\\
&+ \frac{\bi}{2\sqrt{2\pi}\mu_i}\left[e^{\bi\mu_i(\tilde{x}_2-\tilde{y}_2)^+}-e^{\bi\mu_i(2\tilde{M}_2 - \tilde{y}_2^+-\tilde{x}_2^+)} \right], 
\end{align}
and if $x\in \Omega^{3-i}$ and $y\in \Omega^{i}$, then
\begin{align}
\hat{G}(x_2;y_2,\xi) =& \frac{\bi}{A\sqrt{2\pi}}\left[e^{\bi\left(\mu_i(2\tilde{M}_2 - \tilde{y}_2^+)+\mu_{3-i}(2\tilde{M}_2-\tilde{x}_2^+)\right)}- e^{\bi\left(\mu_i(2\tilde{M}_2-\tilde{y}_2^+)+\mu_{3-i}\tilde{x}_2^+\right)} -  e^{\bi\left(\mu_i\tilde{y}_2^++\mu_{3-i}(2\tilde{M}_2-\tilde{x}_2^-)\right)}\right]\nonumber\\
&+\frac{\bi}{\sqrt{2\pi}(\mu_1+\mu_2)}\left(1+ \frac{B}{A}\right)e^{\bi( \mu_i\tilde{y}_2^++\mu_{3-i}\tilde{x}_2^+ )},
\end{align}
where $\mu_i=\sqrt{k_i^2-\xi^2}$, $\epsilon_i=e^{2\bi \mu_i\tilde{M}_2}$, $\tilde{M}_2=\int_0^{M_2}\alpha_{2}(t)dt$ and 
\begin{align}
\label{eq:def:B+1}
B^i_1 &= (\mu_{i}-\mu_{3-i})- (\mu_1+\mu_2) \epsilon_{3-i},\\
\label{eq:def:B+2}
B^i_2 &=(\mu_{i}^2-\mu_{3-i}^2)\epsilon_1\epsilon_2-(\mu_1-\mu_2)^2\epsilon_{i} - 4\mu_1\mu_2\epsilon_{3-i},\\
\label{eq:def:B+3}
B &=(\mu_1+\mu_2)\epsilon_1\epsilon_2 - (\epsilon_1-\epsilon_2)(\mu_1-\mu_2),\\
\label{eq:def:A}
A &= (1-\epsilon_1\epsilon_2)(\mu_1+\mu_2)+(\epsilon_1-\epsilon_2)(\mu_1-\mu_2) =(1-\epsilon_2)(1+\epsilon_1)\mu_1 + (1-\epsilon_1)(1+\epsilon_2)\mu_2.
\end{align}
Recall that the free space Green's function for Helmholtz equation with wavenumber $k_1$ \cb{\cite{Chew90}} is
\[
\Phi(k_1,x,y):=\frac{\bi}{4} H_0^{(1)}(k_1|x-y|) = \frac{\bi}{4\pi}\int_{-\infty}^{+\infty} \frac{1}{\mu_1} e^{\bi(x_1-y_1)\xi + \bi \mu_1|x_2-y_2|}d\xi.
\]
Thus, taking the inverse Fourier transform of $\hat{G}$ with respect to $\xi$,
we obtain $G(x;y)$ that takes the following form: for $i=1,2$, 
\begin{align}
\label{eq:Green:sol:+}
G(x,y) =& \left\{
\begin{array}{lc}
G^{i,i}_{\rm layer}\left((x_1,\tilde{x}_2),(y_1,\tilde{y}_2)\right)+G_{\rm res}^{1,1}\left((x_1,\tilde{x}_2),(y_1,\tilde{y}_2)\right), & {\rm for }\quad x,y\in \Omega^i,\\
G^{3-i,i}_{\rm layer}\left((x_1,\tilde{x}_2),(y_1,\tilde{y}_2)\right)+G_{\rm res}^{3-i,i}\left((x_1,\tilde{x}_2),(y_1,\tilde{y}_2)\right), & {\rm for }\quad x\in \Omega^{3-i}, y \in \Omega^i,\\
\end{array}
\right.
\end{align} 
where
\begin{align}
\label{eq:Green:res:++}
G_{\rm res}^{i,i}=&- \Phi(k_i,(x_1,2\tilde{M}_2-\tilde{x}_2^+),(y_1,\tilde{y}_2)) +\frac{\bi}{4\pi}\int_{-\infty}^{+\infty}\frac{e^{\bi(x_1-y_1)\xi}}{A} f_{x_2,y_2}^{i,i}(\xi)d\xi,\\
\label{eq:Green:layer:++}
G_{\rm layer}^{i,i} =&\Phi(k_i, (x_1,\tilde{x}_2), (y_1,\tilde{y}_2)) + \Phi (k_i, (x_1, \tilde{x}_2),(y_1,-\tilde{y}_2))+\frac{\bi}{4\pi}\int_{-\infty}^{+\infty}e^{\bi(x_1-y_1)\xi}g_{x_2,y_2}^{3-i,i}(\xi)d\xi,\\
\label{eq:Green:res:-+}
G_{\rm res}^{3-i,i}=&\frac{\bi}{2\pi}\int_{-\infty}^{+\infty}\frac{e^{\bi(x_1-y_1)\xi}}{A}f_{x_2,y_2}^{3-i,i}(\xi) d\xi,\\
\label{eq:Green:layer:-+}
G_{\rm layer}^{3-i,i}=&\frac{\bi}{2\pi}\int_{-\infty}^{+\infty}e^{\bi(x_1-y_1)\xi}g_{x_2,y_2}^{3-i,i}(\xi)d\xi,
\end{align}
and
\begin{align}
\label{eq:f:++}
f_{x_2,y_2}^{i,i}(\xi)=&\frac{B_2^ie^{\bi \mu_i(\tilde{x}_2^++\tilde{y}_2^+)} }{\mu_i(\mu_1+\mu_2)}
+\frac{B_1^i}{\mu_i} \left(  e^{\bi\mu_i(4\tilde{M}_2 -\tilde{y}_2^+ - \tilde{x}_2^+)} -e^{\bi\mu_i(2\tilde{M}_2 -\tilde{y}_2^+ + \tilde{x}_2^+)}- e^{\bi\mu_i(2\tilde{M}_2 + \tilde{y}_2^+-\tilde{x}_2^+)}\right),\\ 
\label{eq:g:++}
g_{x_2,y_2}^{i,i}(\xi) =& \frac{2e^{\bi \mu_i(\tilde{x}_2^+ + \tilde{y}_2^+)}}{\mu_1+\mu_2},\\
\label{eq:f:-+}
f_{x_2,y_2}^{3-i,i}(\xi)=&\frac{Be^{\bi(\mu_i\tilde{y}_2^++\mu_{3-i}\tilde{x}_2^+)}}{\mu_1+\mu_2} + e^{\bi\left(\mu_{i}(2\tilde{M}_2 - \tilde{y}_2^+)+\mu_{3-i}(2\tilde{M}_2-\tilde{x}_2^+)\right)}
- e^{\bi\left(\mu_i(2\tilde{M}_2-\tilde{y}_2^+)+\mu_{3-i}\tilde{x}_2^+\right)} \nonumber\\&- e^{\bi\left(\mu_i\tilde{y}_2^++\mu_{3-i}(2\tilde{M}_2-\tilde{x}_2^+)\right)},\\
\label{eq:g:-+}
g_{x_2,y_2}^{3-i,i}(\xi) =& \frac{e^{\bi (\mu_i\tilde{y}_2^+ + \mu_{3-i}\tilde{x}_2^+)}}{\mu_1+\mu_2}.
\end{align}

Notice that $G^{i,j}_{\rm layer}$ represents the exact layered medium Green's
function $G_{k_1,k_2}(x,y)$ for $x\in \Omega^i$
and $y\in \Omega^j$ when $\bar{\sigma}_2=0$. In this sense, $G^{i,j}_{\rm res}$,
$i,j=1,2$, can be taken as the residual term for the layered medium Green's
function due to the horizontal PML truncation. As $G(x,y)$ is only formally
defined in \eqref{eq:Green:sol:+}, we need to verify that the integrals above,
including both $G^{i,j}_{\rm layer}$ and $G^{i,j}_{\rm res}$ are well-defined
when $\bar{\sigma}_2>0$. They all depend on the properties of {$A$ in
  \eqref{eq:def:A}  and $f_{x_2,y_2}^{i,j}$ and $g_{x_2,y_2}^{i,j}$ for $i,j=1,2$,}
which are studied in the following.
\subsection{Properties of $A$}\label{subsec:A}
We start by two technical lemmas.
\begin{mylemma}\label{techlemma1}
	For any $a>0$, the function
	\[
	F(x_1,x_2) = (1-e^{-2ax_1})(1-e^{-2x_2/a}) - 4e^{-ax_1-x_2/a}|\sin x_1\sin x_2|,
	\]
	defined on the domain $\{(x_1,x_2): x_1\geq 0, x_2\geq 0, x_1^2 + x_2^2\neq 0\}$ is always nonnegative, and $F(x_1,x_2)=0$ if and only if $x_1x_2=0$.
	\begin{proof}
		Clearly, by monotonicity and  periodicity, if $F(x_1,x_2)>0$, then $F(x_1+m\pi,x_2+n\pi)>0$ for any two integers $m,n\geq 0$. Therefore, we only need to study $F(x_1,x_2)$ in the domain:
		\[
		D_1=\{(x_1,x_2): 0\leq x_1,x_2\leq \pi, x_1^2 + x_2^2\neq 0\}.
		\]
		In fact, we can further reduce the domain $D_1$ to 
		\[
		D_2=\{(x_1,x_2):0\leq x_1,x_2\leq \pi/2, x_1^2 + x_2^2 \neq 0\},
		\]
		since $\sin(\pi-x_1)=\sin x_1$ and for any $(x_1,x_2)\in D_{1}\backslash D_2$,
		either $x_1\geq \pi - x_1$ or $x_2\geq \pi - x_2$.
		
		Now, we prove that if $x_1x_2\neq 0$, then	$F(x_1,x_2) >0\mbox{ in } D_{2}$.
		Since $\sin x_1 (1-e^{-2x_2/a})e^{-ax_2}>0$, it holds
		\begin{align}
		\label{eq:Bandf}
		\frac{ F(x_1,x_2) }{2\sin x_1(1-e^{-2x_2/a})e^{-ax_1}} = f(x_1;a) - \frac{1}{f(x_2;a^{-1})},
		\end{align}
		where $f(x_1;a)$ is defined by
		\[
		f(x_1;a) := \frac{1-e^{-2ax_1}}{2\sin(x_1)e^{-ax_1}},\quad
		x_1\in(0,\frac{\pi}{2}].
		\]
		Since $\lim_{x_1\to 0} f(x_1;a)=a$, we let $f(0;a)=a$ so that $f$
		is defined on $[0,\pi/2]$. We claim 
		\[
		f(x_1;a)>f(0;a),\quad{\rm for}\quad x_1\in(0,\frac{\pi}{2}],
		\]
		since one can easily check that for any $a>0$,
		\[
		(1-e^{-2ax_1}) - 2a\sin x_1e^{-ax_1}>0,\quad{\rm for}\quad
		x_1\in(0,\frac{\pi}{2}].
		\]
		Thus, we obtain
		\[
		f(x_1;a)>a,\quad f(x_2;a^{-1})>a^{-1},\quad{\rm for}\quad x_1x_2\neq 0,
		\]
		such that 
		\[
		f(x_1;a)-\frac{1}{f(x_2;a^{-1})}>a - \frac{1}{a^{-1}}=0.
		\]
		Consequently, equation \eqref{eq:Bandf} implies that $F(x_1,x_2)>0$,
		when $(x_1,x_2)\in D_{2}$ and $x_1x_2\neq 0$, which completes the proof since it is obvious that $F(x_1,x_2)=0$ if $x_1x_2= 0$.
	\end{proof}
\end{mylemma}

\begin{mylemma}\label{techlemma2}
	Suppose $a, b>0$, then 
	\[
	\left| \frac{e^{\bi \mu_j (a+\bi b)} - 1}{\mu_j} \right|\leq
	\frac{4+2\sqrt{k_2^2-k_1^2}|a+\bi b|}{|\mu_1+\mu_2|},\quad j=1,2,
	\]
	for all $\xi\in\ol{\mathbb{C}^{-+}\cup\mathbb{C}^{+-}}$, where we take the limit value for the left part when	$\xi=\pm k_1, \pm k_2$.
	\begin{proof} 
		By elementary analysis, we see that for any $c\geq 0$ and $d\in\mathbb{R}$,
    \begin{equation}
      \label{eq:tmp}
		\left| \frac{e^{-c+\bi d} - 1}{-c + \bi d} \right|\leq 2,
    \end{equation}
		where limit is considered when $c^2+ d^2=0$.
		Hence, it yields by (\ref{eq:tmp})
		\begin{align*}
		\left|  \frac{( e^{\bi \mu_j (a+\bi b)} - 1 )
			(\mu_1+\mu_2)}{\mu_j}\right|
		&\leq 2|e^{\bi \mu_j (a+\bi b)} - 1| +
		|\mu_1-\mu_2|\left| \frac{e^{\bi \mu_j(a+\bi b)} - 1}{\mu_j} \right|\\
		&\leq 4 + 2\sqrt{k_2^2-k_1^2}|a+\bi b|,
		\end{align*}
    {where we have used the fact $|\mu_1-\mu_2|\leq \sqrt{k_2^2-k_1^2}$ in the
    last inequality.}
	\end{proof}
\end{mylemma}

To study the property of $A$ in \eqref{eq:def:A}, we first note that $A$ can be regarded as a function of $\mu_j, j=1,2$, or a function of $\xi$. Due to the relations among $\xi$, $\mu_1$, and
$\mu_2$, these notations are all equivalent. In the following, to show the property of $A$, we may use all the three notations: $A(\mu_j)$,  $j=1,2$, or $A(\xi)$, depending on which one is more convenient than the others.

\begin{mylemma}\label{lemA1}
	$A(\xi)$ has a total of four roots on the real axis, namely,
	\[
	\xi = \pm k_1,\quad \xi = \pm k_2.
	\]
	Furthermore, $A(\xi)\ne 0$ for $\xi$ on the imaginary axis.
	\begin{proof}
		We first prove the case when $\xi\in \mathbb{R}$.
		Since $A(\xi)=A(-\xi)$, we consider $\xi>0$ only. 
		We claim that $A$ doesn't have any root for
		$\xi\in[0,k_1)\cup(k_2,\infty)$. Otherwise, suppose $A(\xi_1)=0$ for
		$\xi_1<k_1$  and $A(\xi_2)=0$ for $\xi_2>k_2$. When $\xi\ne k_j$,  $j=1,2$, since $\epsilon_j(\xi)\ne 1$, it holds
		\begin{align}
		\label{eq:rel0}
		\frac{A}{(1-\epsilon_1)(1-\epsilon_2)}=\frac{1+\epsilon_1(\xi_j)}{1-\epsilon_1(\xi_j)}\mu_1(\xi_j) + 
		\frac{1+\epsilon_2(\xi_j)}{1-\epsilon_2(\xi_j)}\mu_2(\xi_j)=0,
		\end{align}
		which can be rewritten as
		\begin{align}
		\label{eq:rel1}
		\frac{(1-|\epsilon_1(\xi_j)|^2) + 2{\rm
				Im}(\epsilon_1(\xi_j))\bi}{|1-\epsilon_1(\xi_j)|^2}\mu_1(\xi_j) + 
		\frac{(1-|\epsilon_2(\xi_j)|^2) + 2{\rm
				Im}(\epsilon_2(\xi_j))\bi}{|1-\epsilon_2(\xi_j)|^2}\mu_2(\xi_j)=0
		\end{align}
		for $j=1,2$. Since $\mu_2(\xi_1)>\mu_1(\xi_1)>0$ and 
		$\mu_1(\xi_2)(-\bi)>\mu_2(\xi_2)(-\bi)>0$, the real part of
		left hand side(l.h.s.) of \eqref{eq:rel1} for $j=1$ is strictly positive and the imaginary
		part of l.h.s. of (\ref{eq:rel1}) for $j=2$ is strictly positive, which 
		is impossible.
		
		Now, we claim that there is no root in $(k_1,k_2)$. Otherwise, suppose
		$A(\xi_*)=0$ for $\xi_*\in(k_1,k_2)$. Then,
		$\mu_1 = \sqrt{\xi_*^2-k_1^2}\bi$ and $\mu_2 = \sqrt{k_2^2-\xi_*^2}$.
		Denote $c=\sqrt{\xi_*^2 - k_1^2}>0$ and $d=\sqrt{k_2^2-\xi_*^2}>0$. Then,
		\[
		0=A(\xi_*) = (1-\epsilon_2)(1+\epsilon_1)c\bi +
		(1-\epsilon_1)(1+\epsilon_2)d = (1-\epsilon_1\epsilon_2)(c\bi+d) +
		(\epsilon_1-\epsilon_2)(c\bi -d).
		\]
		Therefore, $|1-\epsilon_1\epsilon_2|^2 = |\epsilon_1-\epsilon_2|^2$, which
    is equivalent to
		\begin{align}
		\label{eq:rel2}
		(1-|\epsilon_1|^2)(1-|\epsilon_2|^2) + 4{\rm Im}(\epsilon_1){\rm
			Im}(\epsilon_2) = 0.
		\end{align}
		Note that 
		$\epsilon_1 = e^{-2cM_2}e^{-2c\bar{\sigma}_2\bi}$ and $\epsilon_2 = e^{-2d\bar{\sigma}_2}e^{2dM_2\bi}$.
		Equation \eqref{eq:rel2} becomes
		\begin{equation}
		\label{eq:rel3}
		(1 - e^{-4cM_2})(1-e^{-4d\bar{\sigma}_2}) - 4 e^{-2cM_2-2d\bar{\sigma}_2}\sin(2c\bar{\sigma})\sin(2dM_2)=0.
		\end{equation}
		Now, by choosing $a=\frac{M_2}{\bar{\sigma}_2}>0$, $x=2c\bar{\sigma}_2\geq 0$,
		and $y=2dM_2\geq 0$ in Lemma \ref{techlemma1}, we see that
		\[
		(1 - e^{-4cM_2})(1-e^{-4d\bar{\sigma}_2}) - 4
		e^{-2cM_2-2d\bar{\sigma}}\sin(2c\bar{\sigma}_2)\sin(2dM_2)\geq 0,
		\]
		where the equality holds only when $cd=0$, which is a contradiction to the
		choice of $\xi_*$. 
		
		Finally, since $\mu_1,\mu_2>0$ when $\xi$ is on the imaginary axis, 
		the real part of l.h.s of equation ~(\ref{eq:rel1}) is strictly positive, 
		which implies that $A(\xi)\neq 0$. 
	\end{proof}
\end{mylemma}

Another important property of $A$ is given by the following proposition.
\begin{prop} \label{PropA}
	The function $A(\xi)$ is not zero everywhere in $\mathbb{C}^{-+}\cup \mathbb{C}^{+-}$.

	\begin{proof}
		Since $A(-\xi)=A(\xi)$, in the following, we will assume that
		$\xi\in\mathbb{C}^{-+}$, which implies $\mu_j\in \mathbb{C}^{++}$
		and $|\epsilon_j|<1$, for $j=1,2$. Hence, the function
		$f(\mu_1) = \frac{A(\mu_1)}{(1-\epsilon_1)(1-\epsilon_2)}$,
		with $\mu_2 = \sqrt{k_2^2-k_1^2 + \mu_1^2}$, is holomorphic in
		$\mathbb{C}^{++}$. We now show that $f(\mu_1)\neq 0$ for
		$\mu_1\in\mathbb{C}^{++}$.
		
		Lemma \ref{lemA1} indicates that on the boundary of $\mathbb{C}^{++}$, $f(\mu_1)$
		has only two roots, i.e., $0$ and $\sqrt{k_2^2-k_1^2}\bi$. For
		sufficiently small $\varepsilon>0$ and for sufficently large $r>0$,
we define the counter-clockwise oriented closed curve $C_\varepsilon^r$, \cb{as
  shown in Figure~\ref{fig:integralcurve}},
		\begin{figure}[!ht]
			\centering
			\includegraphics[width=0.4\textwidth]{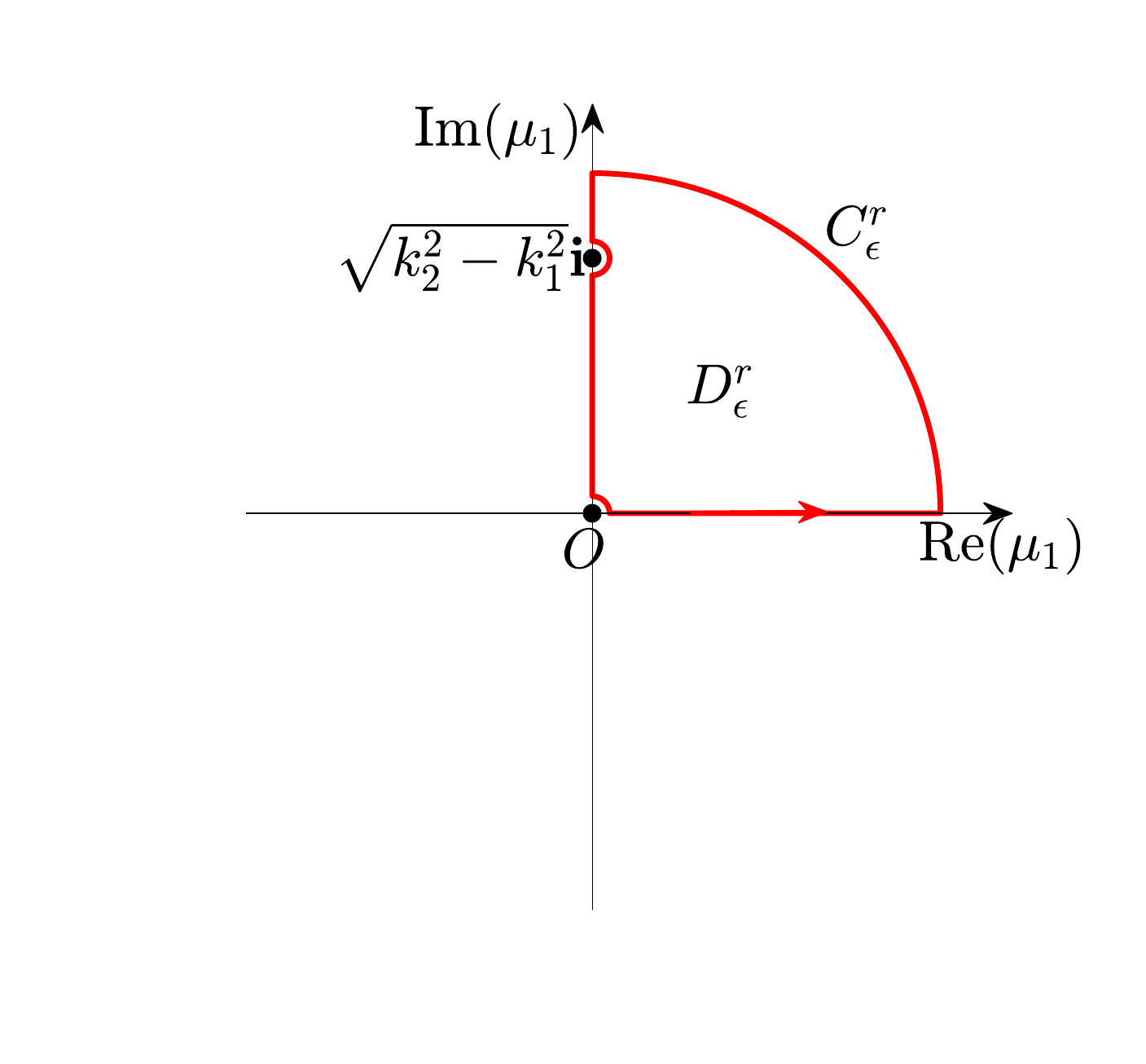}
			\vspace{-1cm}
			\caption{\cb{The integral contour $C_\varepsilon^r$}.}
			\label{fig:integralcurve}
		\end{figure}
    by the boundary of the region $D_\varepsilon^r=D_r\backslash
    (D_1^{\varepsilon}\cup D_2^{\varepsilon})$, and the regions
    $D_1^{\varepsilon}$, $D_2^{\varepsilon}$ and $D_r$ are given by
		\begin{align*}
		\begin{cases}
		D_1^{\varepsilon} = & \{ae^{i\theta}\in\mathbb{C}: 0<a\leq \varepsilon,0\leq \theta\leq \frac{\pi}{4}\},\\
		D_2^{\varepsilon} = & \{\sqrt{k_2^2-k_1^2}\bi + ae^{i\theta}\in\mathbb{C}: 0<a\leq \varepsilon,-\frac{\pi}{2}\leq \theta\leq \frac{\pi}{2}\}, \\
		D_r= &\{ae^{i\theta}\in\mathbb{C}:0\leq a\leq r,0\leq
		\theta\leq \frac{\pi}{4}\}.
		\end{cases}
		\end{align*}
		We show that for sufficiently large $r$, $f(\mu_1)$ must have at most
		finite number of zeros in $D_r$. In fact, when $|\mu_1|=r$ and
		$r\to\infty$, we notice
		\[
		|\mu_1-\mu_2| = |\frac{k_2^2-k_1^2}{\mu_1+\mu_2}|\leq
		\frac{k_2^2-k_1^2}{|\mu_1|}\to 0,\quad |\epsilon_j|= |e^{\bi
			re^{\bi\theta}(M_2+\bi\bar{\sigma})}| =
		e^{-r(M_2\sin\theta+\bar{\sigma}\cos\theta)} \to 0,
		\] 
		which implies
		\begin{align}
		\label{eq:A:infty}
		\lim_{r\to\infty}|f(\mu_1)|= 
		\lim_{r\to\infty}|(1-\epsilon_1\epsilon_2)(\mu_1+\mu_2) +
		(\epsilon_1-\epsilon_2)(\mu_1-\mu_2)|=\lim_{r\to\infty}|2\mu_1|=\infty.
		\end{align}
		Therefore, $f(\mu_1)$ cannot be zero outside $D_r$ for sufficiently
		large $r$. Now, suppose there is a sequence $\{\mu_{1,n}\}_{n=1}^{\infty}$
		such that $f(\mu_{1,n})=0$ and $\lim_{n\to\infty}\mu_{1,n}=\mu_{1,*}\in
		D_r$. Then, $\mu_{1,*}$ must be on the boundary of $D_r$, since otherwise
		the analyticity indicates that $f(\mu_1)\equiv 0$ everywhere inside
		$D_r$, which is a contradiction. As $\mu_{1,*}$ cannot be on the boundary of
		$D_r$ with $|\mu_1|=r$, it must be on the real or imaginary axis. By Lemma \ref{lemA1},
		it leads us to the following two situations: $\mu_{1,*}=0$ or
		$\mu_{1,*}=\sqrt{k_2^2-k_1^2}\bi$. If $\mu_{1,*}=0$, then
		\begin{align*}
      0=\lim_{n\to\infty}{\rm Re} f(\mu_{1,n})
		=\frac{2\bar{\sigma}_2}{|\tilde{M}_2|^2} + \frac{1-e^{-4\sqrt{k_2^2-k_1^2}\bar{\sigma}_2}}{|1-e^{2\bi\sqrt{k_2^2-k_1^2}\tilde{M}_2}|^2}\sqrt{k_2^2-k_1^2}>0,
		\end{align*}
		a contradiction; one similarly shows the impossibility of
    the other case. 
		
		Thus, we can choose sufficiently small $\varepsilon$ and sufficiently large
    $r$ so that all the zeros of $f(\mu_1)$ are contained inside the curve
    $C_\varepsilon^r$. By the argument principle, the total number of zeros
    equals to
		\begin{align}
		\label{eq:root:integral2}
		{\frac{1}{2\pi\bi}\int_{C_\varepsilon^r} \frac{f'(\mu_1)}{f(\mu_1)} d\mu_1, }
		\end{align}
		which evaluates the total change in the argument of $f(\mu_1)$ as $\mu_1$
		travels around $C_\varepsilon^r$.
		
		By (\ref{eq:rel0}) and (\ref{eq:rel1}), we see that
		\begin{equation}
		\label{eq:f:exp}
		f(\mu_1) = \frac{(1-|\epsilon_1|^2) + 2{\rm
				Im}(\epsilon_1)\bi}{1+|\epsilon_1|^2-2{\rm Re}(\epsilon_1)}\mu_1 + \frac{(1-|\epsilon_2|^2) + 2{\rm
				Im}(\epsilon_2)\bi}{1+|\epsilon_2|^2-2{\rm Re}(\epsilon_2)}\mu_2.
		\end{equation}
		We now analyze the change of argument of $f(\mu_1)$ on $C_\varepsilon^r$ part by part.
		\begin{enumerate}
			\item On the real-axis part of $C_\varepsilon^r$, since $\mu_2$ is also
			real, which implies ${\rm Re}(f(\mu_1))>0$, we have $f(\mu_1)$ lies in the right half of the	complex plane. 
			\item When $\mu_1=re^{i\theta}, 0\leq \theta\leq \frac{\pi}{2}$, since as
			$r\rightarrow\infty$, $f(\mu_1)\sim 2\mu_1$. Therefore, $f(\mu_1)$ is
			dominated by a nonzero complex number in
			$\overline{\mathbb{C}^{++}}\backslash\{0\}$ and hence cannot lie in
			$\overline{\mathbb{C}^{--}}$.
			\item When $\mu_1 = y\bi$ for $[\sqrt{k_2^2-k_1^2}+\varepsilon, r]$,
			since $\mu_1^2<k_2^2-k_1^2$, $\mu_2=\sqrt{y^2 - k_2^2+k_1^2}\bi$ such that
			${\rm Im}(f(\mu_1))>0$ so that $f(\mu_1)$ lies in the upper half complex plane.
			\item When $\mu_1\in\partial D_1^{\varepsilon}=\{\varepsilon e^{i\theta}:0\leq
			\theta\leq \frac{\pi}{2}\}$, the boundary of $D_1^\varepsilon$, since
			$\lim_{\varepsilon\to 0^+}{\rm Re} f(\mu_1)={\rm Re}f(0)>0$. we can make
			$\varepsilon$ sufficiently small such that, $f(\mu_1)$ lies in the right
			half plane with a positive real part.
			\item On the boundary of $D_2^\varepsilon$, one similarly derives that $f(\mu_1)$
			lies in the upper half plane since $\lim_{\varepsilon\to 0^+}{\rm
				Im}f(\mu_1)>0$.
			\item On the line segment $\mu_1=y\bi$ for $y\in[\varepsilon,
			\sqrt{k_2^2-k_1^2}-\varepsilon]$, $\mu_2 = \sqrt{k_2^2-k_1^2 + y^2}$, it holds that
			\begin{align}
			\label{eq:f:exp2}
			f(\mu_1) =& \frac{-2e^{-2yM_2}\sin(2y\bar{\sigma}_2)y}{1+e^{-4yM_2}-2e^{-2yM_2}\cos(2y\bar{\sigma}_2)} +\frac{(1-e^{-4\mu_2\bar{\sigma}_2})\mu_2}{1+e^{-4\mu_2\bar{\sigma}_2}-2e^{-2\mu_2\bar{\sigma}_2}\cos(2\mu_2M_2)} \nonumber\\
			&+ \left(  \frac{-2e^{-2\mu_2\bar{\sigma}_2}\sin(2\mu_2M_2)\mu_2}{1+e^{-4\mu_2\bar{\sigma}_2}-2e^{-2\mu_2\bar{\sigma}_2}\cos(2\mu_2 M_2)} +\frac{(1-e^{-4yM_2})y}{1+e^{-4yM_2}-2e^{2yM_2}\cos(2y\bar{\sigma}_2)}\right)\bi,
			\end{align}
			We claim $f(\mu_1)$ cannot lie in $\overline{\mathbb{C}^{--}}$. Otherwise,
			\[
			(1-e^{-4\mu_2\bar{\sigma}_2})\mu_2\cdot (1-e^{-4yM_2})y\leq
			2e^{-2yM_2}|\sin(2y\bar{\sigma}_2)|y\cdot
			2e^{-2\mu_2\bar{\sigma}_2}|\sin(2\mu_2 M_2)|\mu_2,
			\]
			which is equaivalent to 
			\[
			(1-e^{-4\mu_2\bar{\sigma}_2})(1-e^{-4yM_2})-4e^{-2(yM_2+\mu_2\bar{\sigma}_2)}|\sin(2y\bar{\sigma}_2)||\sin(2\mu_2M_2)| \leq 0,
			\]
			but this is impossible since Lemma \ref{techlemma1} indicates that the l.h.s is strictly
			positive when $y>0$.
		\end{enumerate}
		
		To sum up, we have shown that $f(\mu_1)$ on $C_\varepsilon^r$ cannot
		attain the region $\overline{\mathbb{C}^{--}}$ nor origin $O$ certainly.
		Thus, the total change of argument of $f$ as $\mu_1$ travels around
		$C_\varepsilon^r$ must be $0$.
		Consequently, by the argument principle and by (\ref{eq:A:infty}),
		$f(\mu_1)$ has no root in $\mathbb{C}^{++}$, which
		completes the proof since it implies $A(\xi)$ can not be zero in $\mathbb{C}^{-+}$.
	\end{proof}
\end{prop}
As a corollary, we get the following result.
\begin{mycorollary}
	The following eigenvalue problem 
	\begin{align}\label{eq:eigproblem}
	\begin{cases}
	&\frac{1}{\alpha_2}\frac{d}{dx_2}\left(\frac{1}{\alpha_2}\frac{d \phi}{d x_2} \right) + k^2\phi = \xi^2\phi,\quad {\rm for }\quad  x_2\in(-M_2,M_2).\\
	&[\phi] = 0,\quad [\phi'(x_2)] = 0,\quad{\rm on}\quad x_2=0,\\
	&\phi = 0,\quad{\rm on}\quad x_2 = \pm M_2,
	\end{cases}
	\end{align}
	has no eigenvalues $\xi$ in $\mathbb{C}^{-+}\cup\mathbb{C}^{+-}$.
	\begin{proof}
		Suppose there exists an eigenvalue $\xi\in\mathbb{C}^{-+}\cup\mathbb{C}^{+-}$
		with its associated eigenfunction $\phi\neq 0$. For the two-layered medium, $\phi$ can be written in the form,
		\begin{align} 
		\phi = \left\{
		\begin{array}{lc}
		c_1 e^{\bi\mu_1 \tilde{x}_2} + c_2 e^{-\bi\mu_1 \tilde{x}_2}, & x_2>0,\\
		d_1 e^{-\bi\mu_2 \tilde{x}_2} + d_2 e^{\bi\mu_2 \tilde{x}_2}, & x_2<0.
		\end{array}
		\right.
		\end{align}
		The boundary condition and interface conditions in \eqref{eq:eigproblem}
    give rise to
		\[
		\left\{
		\begin{array}{l}
		(1-\epsilon_1)c_1 -  (1-\epsilon_2) d_1=0,\\
		\mu_1(1+\epsilon_1)c_1 + \mu_2(1+\epsilon_2)d_2 = 0. 
		\end{array}
		\right.
		\]
		Since $\phi\neq 0$, the  linear system above must have a nonzero solution, which implies  the determinant 
		\[
		0=(1-\epsilon_1)\mu_2(1+\epsilon_2) + (1-\epsilon_2)(1+\epsilon_1)\mu_1 =
		A(\xi).
		\]
		It indicates that $A(\xi)$ has a root in
		$\mathbb{C}^{-+}\cup\mathbb{C}^{+-}$, which is a contradiction to Proposition 4.1.
	\end{proof}
\end{mycorollary}
This Corollary reveals a stronger result in comparison with {
  \cite[Prop. A.1]{lulusong18}}, where $\bar{\sigma}_2$ should be sufficiently large.

We also need the properties of $A$ at $\mu_j=0$ for $j=1,2$.
\begin{mylemma}\label{lemA3}
	The function $A(\mu_j)$ has a simple zero at $\mu_j=0$, $j=1,2$. In particular, it holds
	\begin{equation}
	\label{eq:est:Ajp}
	\left|A'(\mu_j)|_{\mu_j=0}\right|\geq 2\sqrt{k_2^2-k_1^2}\min(M_2,\bar{\sigma}_2)(1-e^{-2\sqrt{k_2^2-k_1^2}\min(M_2,\bar{\sigma}_2)}), \quad j=1,2.
	\end{equation}
	
	\begin{proof}
		We prove the case $j=1$, {and for the moment we denote by
      $A(0)$ and $A'(0)$ the function $A(\mu_1)$ and $A'(\mu_1)$ evaluated at
      $\mu_1=0$, respectively.} It can be verified that $A(0)=0$ since $\mu_1=0$
    when $\xi = \pm k_1$.
		By direct calculation, at $\mu_1=0$, $\epsilon_1=1$ and
    $\mu_2=\sqrt{k_2^2-k_1^2}>0$, we get
		\[
		A'(0) = (2-2\bi\tilde{M}_2\mu_2) - (2+2\bi\tilde{M}_2\mu_2)\epsilon_2.
		\]
		Since $\bi\tilde{M}_2\mu_2=\sqrt{k_2^2-k_1^2}(-\bar{\sigma}_2+\bi M_2)\in
    \mathbb{C}^{-+}$, we have
    $|2-2\bi\tilde{M}_2\mu_2|>|2+2\bi\tilde{M}_2\mu_2|$, yielding
		\[
		|A'(0)|> |2+2\bi\tilde{M}_2\mu_2|\cdot|1-|\epsilon_2||>2\sqrt{k_2^2-k_1^2}M_2(1-e^{-2\sqrt{k_2^2-k_1^2}\bar{\sigma}_2}).
		\]
    One similarly proves the property for $\mu_2=0$; we omit the details.
	\end{proof}
\end{mylemma}
Combining all the properties above, we obtain a lower bound of $A(\xi)$.
\begin{mylemma}\label{thmA5}
	For any $z\in\ol\mCpp$ with $|z|\lesssim|\tilde{M}_2|$, we have
	\[
	\max\left\{|\mu_2( e^{\bi\mu_1 z} - 1 )|,|\mu_1(e^{\bi\mu_2 z}-1)|,\left|
	\frac{\mu_1\mu_2}{\mu_1+\mu_2} \right|  \right\}\lesssim |A(\xi)|,
	\]
	for all $\xi\in\ol{\mCpm\cup\mCmp}$.
	\begin{proof}
		By Lemma \ref{techlemma2}, we only need to prove the estimate 
		\[
		\left| \frac{\mu_1\mu_2}{\mu_1+\mu_2} \right|\lesssim |A(\xi)|.
		\]
		Otherwise, there must exist a sequence
		$\{\xi_{n}\}_{n=1}^{\infty}\in\ol{\mCpm\cup\mCmp}$ with $\xi_n\to\xi_0$ as
		$n\to\infty$, such that
		\[
		\lim_{n\to\infty}\left| \frac{(\mu_1+\mu_2)A(\xi)}{\mu_1\mu_2} \right| = 0.
		\]
		We discuss two cases:
		\begin{itemize}
    \item[1.] $|\xi_0|<+\infty$, then we claim that $\xi_0$ must be one of the
      four values $\pm k_1, \pm k_2$ since otherwise we get $A(\xi)=0$ for
      $\xi=\xi_0$, which is in contradiction with Proposition \ref{PropA}.
      However, even if $\xi_0\in\{\pm k_1,\pm k_2\}$, we can immediately get
      $A'(\mu_1)|_{\mu_1=0} = 0$ or $A'(\mu_2)|_{\mu_2=0} = 0$, which is in
      contradiction with Lemma \ref{lemA3}.
			\item[2.] $|\xi_0|=+\infty$, then one easily gets that $\eps_j\to 0$ since
			$M_2,\bar{\sigma}_2>0$ and since $\mu_{j,n}=\sqrt{k_j^2-\xi_n^2}\to
			\sqrt{\xi_0^2}\bi\in\mCpp$. Consequently,
			$\lim_{n\to\infty}\frac{(\mu_{1,n}+\mu_{2,n})A(\xi_n)}{\mu_{1,n}\mu_{2,n}}
			= 4$, 
			which is also a contradiction.
		\end{itemize}
	\end{proof}
\end{mylemma}
\subsection{The properties of $f_{x_2,y_2}^{i,j}$,  $i,j=1,2$}
To see $f_{x_2,y_2}^{i,j}$ more clearly, we give a decomposition of
$f_{x_2,y_2}^{i,j}$ for $i,j=1,2$ first. It turns out that
$f_{x_2,y_2}^{i,j}(\xi)$ with $x\in B_{\rm ex}^i$ and $y\in B_{\rm ex}^j$ for
$i,j=1,2$ can be decomposed as follows,
	\[
	f_{x_2,y_2}^{i,j}(\xi) = \sum_{l=1}^2f_{x_2,y_2}^{i,j;l}(\xi)e^{\bi\mu_l\tilde{M}_2},
	\]
	where
	\begin{align}
	\label{eq:def:111}
	f_{x_2,y_2}^{i,i;i}(\xi) =&\left[ 2(\epsilon_i - 1) + \frac{4\mu_{3-i}}{\mu_1+\mu_2} \right]e^{i\mu_{i}(\tilde{M}_2 + \tilde{x}_2^+ + \tilde{y}_2^+)}\nonumber\\
	&-\left[ (\epsilon_{3-i}-1) + \frac{(1+\epsilon_{3-i})\mu_{3-i}}{\mu_i}\right]\Big[e^{i\mu_i(\tilde{M}_2 + \tx_2^+ + \ty_2^+)}+e^{\bi \mu_i(3\tilde{M}_2-\tx_2^+-\ty_2^+)}\nonumber\\
	&- e^{\bi \mu_i(\tilde{M}_2 - \ty_2^++\tx_2^+)} - e^{\bi \mu_i(\tilde{M}_2 + \ty_2^+ - \tx_2^+)}\Big],\\
	\label{eq:def:112}
	f_{x_2,y_2}^{i,i;3-i}(\xi) =& -\frac{4\mu_{3-i}}{\mu_1+\mu_2} e^{\bi \mu_i(\tx_2^+ + \ty_2^+) + \bi\mu_{3-i}\tM_2},\\
	\label{eq:def:211}
	f_{x_2,y_2}^{3-i,i;i}(\xi) =& \frac{\mu_{3-i}-\mu_i}{\mu_1+\mu_2}e^{\bi(\mu_i(\tilde{M}_2+\ty_2^+)+\mu_{3-i}\tx_2^+)} - e^{\bi(\mu_i(\tilde{M}_2-\ty_2^+)+\mu_{3-i}\tx_2^+)},\\
	\label{eq:def:212}
	f_{x_2,y_2}^{3-i,i;3-i}(\xi) =& \left(\epsilon_{i}+ \frac{\mu_{i}-\mu_{3-i}}{\mu_1+\mu_2}\right)e^{\bi(\mu_{i}\ty_2^++\mu_{3-i}(\tilde{M}_2+\tx_2^+))} \nonumber\\
	&+ e^{\bi(\mu_{i} (2\tilde{M}_2-\ty_2^+) + \mu_{3-i}(\tilde{M}_2 - \tx_2^+))} - e^{\bi(\mu_{i} \ty_2^+ + \mu_{3-i}(\tilde{M}_2 - \tx_2^+))},
	\end{align}
  for $i=1,2$.
	Based on the above decompositions, we reveal the relation
between $A(\xi)$ and $f_{x_2,y_2}^{i,j;l}(\xi)$ for $i,j,l=1,2$ in the following lemma.
	\begin{mylemma}\label{lemA7}
	For $\xi\in\ol{\mathbb{C}^{+-}\cup\mathbb{C}^{-+}}$, it holds
	\begin{align*}
	|f_{x_2,y_2}^{i,j;l}(\xi)|\lesssim \frac{\left|A\right| }{|\mu_i|},\quad |\partial_{x_2}f_{x_2,y_2}^{i,j;l}(\xi)|\lesssim\Gamma\left|A\right|,
	\end{align*}
	where $x\in B_{\rm ex}^i$, $y\in B_{\rm ex}^j$ for $i,j,l=1,2$.
	\begin{proof}
		We prove $j=1$ only. According to \eqref{eq:def:111} and Lemma \ref{thmA5}, we see that 
		\begin{align*}
		|f_{x_2,y_2}^{1,1;1}|\lesssim & |\eps_2 - 1| + \frac{|\mu_2|}{|\mu_1+\mu_2|} + \frac{|\mu_2|}{|\mu_1|}\Big|e^{i\mu_1(\tilde{M}_2 + \tx_2 + \ty_2)}+e^{\bi \mu_1(3\tilde{M}_2-\tx_2-\ty_2)}\nonumber\\
		&- e^{\bi \mu_1(\tilde{M}_2 - \ty_2+\tx_2)} - e^{\bi \mu_1(\tilde{M}_2 + \ty_2 - \tx_2)}\Big|\lesssim\frac{|A|}{|\mu_1|}.
		\end{align*}
		Similarly, one obtains that $|\partial_{x_2}
    f_{x_2,y_2}^{1,1;1}|\lesssim|A|$. The estimates for $f_{x_2,y_2}^{1,1;2}$
    can be obtained easily by {Lemma \ref{thmA5}}. According to
    (\ref{eq:def:211}) and {Lemma \ref{thmA5}}, we see that
		\begin{align*}
		|f_{x_2,y_2}^{2,1;1}| \lesssim \frac{\left|\mu_1\right|}{|\mu_1+\mu_2|} + |e^{\bi\mu_1( \tilde{M}_2 + y_2 )} - e^{\bi \mu_1( \tilde{M}_2 - y_2 )}||		\lesssim\frac{|A|}{|\mu_2|},
		\end{align*}
		and (\ref{eq:def:212}) leads to
		\begin{align*}
		|f_{x_2,y_2}^{2,1;2}| \lesssim|\epsilon_1-1| + \frac{|\mu_1|}{|\mu_1+\mu_2|} + |e^{\bi\mu_1(2\tilde{M}_2-\ty_2)} - e^{\bi \mu_1 \ty_2}|
		\lesssim\frac{|A|}{|\mu_2|}.
		\end{align*}
		The estimates for all the other cases can be similarly analyzed.
	\end{proof}
\end{mylemma}
\subsection{Existence of the Green's function for the waveguide problem}\label{subsec:Gw}
With the properties of $A$ and $f_{x_2,y_2}^{i,j}$ at our disposal, we are now
ready to show the existence of Green's function for the waveguide problem
(\ref{eq:green:pmly}). The following lemma is used to show $G_{\rm layer}^{i,j}$
and $G_{\rm res}^{i,j}$ appeared in
\eqref{eq:Green:res:++}-\eqref{eq:Green:layer:-+} are well-defined.
\begin{mylemma}\label{lemgreen}
	For any $x, y\in\mathbb{R}\times((-M_2,0)\cup(0,M_2))$, the integrals  
	\begin{align}
	\label{eq:def:I1ij}
	I_1^{i,j}(x_1,\tilde{x}_2;y_1,\tilde{y}_2) = \int_{-\infty}^{+\infty}\frac{e^{\bi(x_1-y_1)\xi}}{A}f_{x_2,y_2}^{i,j}(\xi)d\xi,\\ 
	I_2^{i,j}(x_1,\tilde{x}_2;y_1,\tilde{y}_2) = \int_{-\infty}^{+\infty}e^{\bi(x_1-y_1)\xi}g_{x_2,y_2}^{i,j}(\xi)d\xi,
	\end{align}
	satisfy the following properties:
	\begin{itemize}
		\item[(1).] They are well-defined as improper integrals.
		\item[(2).] They solve the following Helmholtz equations:
		\[
		\partial_{x_1}^2 I_l^{i,j} + \partial_{\tilde{x}_2}^2 I_l^{i,j} + k_i^2 I_l^{i,j}
		= 0,\quad l=1,2.
		\]
		\item[(3).] We have 
		\[
		I_1^{i,j} = \left(  \int_{+\infty\bi}^{0} + \int_{0}^{+\infty}\right)\frac{e^{\bi|x_1-y_1|\xi}}{A}f_{x_2,y_2}^{i,j}(\xi)d\xi,
		\]
		and
		\[
		I_2^{i,j} = \left( \int_{+\infty\bi}^{0} + \int_{0}^{+\infty}\right)e^{\bi|x_1-y_1|\xi}g_{x_2,y_2}^{i,j}(\xi)d\xi.
		\]
		\item[(4).] They satisfy the radiation condition, for $l=1,2$,
		\[
		\left( \frac{\partial I_l^{i,j}}{\partial |x_1-y_1|} - \bi k_i
		I_l^{i,j}\right) = {\cal O}(|x_1-y_1|^{-1}),\quad{\rm as}\ |x_1-y_1|\to\infty.
		\]
		\item[(5).] They satisfy the finiteness condition
		\[
		I_1^{i,j}={\cal O}(|x_1-y_1|^{-1/2}),\quad I_2^{i,j} = {\cal
      O}(|x_1-y_1|^{-1}),\quad{\rm as}\ |x_1-y_1|\to\infty.
		\]
	\end{itemize}
	\begin{proof}
		(1) and (2). Noticing that $f_{x_2,y_2}^{i,j}$ and $A$ are even functions of $\xi$, we get 
		\[
		I_1^{i,j} = \int_{-\infty}^{\infty} \frac{e^{\bi|x_1-y_1|\xi}}{A}f_{x_2,y_2}^{i,j}(\xi) d\xi.
		\]
		By Lemma 3.7 above, we get that 
		\[
		\left|\frac{e^{\bi|x_1-y_1|\xi}}{A}f_{x_2,y_2}^{i,j}(\xi)\right|\lesssim
		\sum_{l=1}^2\frac{|e^{\bi\mu_l\tilde{M}_2|}}{|\mu_l|}\leq
		\sum_{l=1}^2\frac{e^{-p_l\bar{\sigma}_2 - q_lM_2}}{\sqrt{|k_l^2-\xi^2|}},
		\]
		in which we recall $\mu_l = p_l + \bi q_l$ for $l=1,2$. As
		$\xi\to\infty$, one sees from the above that 
		\[
		\left|\frac{e^{\bi|x_1-y_1|\xi}}{A}f_{x_2,y_2}^{i,j}(\xi)\right| = {\cal O}\left( \frac{ e^{-|\xi|M_2} }{|\xi|} \right).
		\]
		On the other hand, as $|\xi|\to k_i$, it holds that 
		\[
		\left|\frac{e^{\bi|x_1-y_1|\xi}}{A}f_{x_2,y_2}^{i,j}(\xi)\right| = {\cal O}(|k_i-|\xi||^{-1/2}).
		\]
		Consequently, the integral $I_1^{i,j}$ exists as an improper integral for
		$i,j=1,2$. One similarly proves the following identities
		\begin{align*}
		\partial_{x_1}^{m}\partial_{\tilde{x}_2}^{n} I_1^{i,j} = \int_{-\infty}^{+\infty} (\bi\xi)^m(\bi\mu_i)^n\frac{e^{\bi|x_1-y_1|\xi}}{A}f_{x_2,y_2}^{i,j}(\xi) d\xi,
		\end{align*}
		for any $m,n\in\mathbb{N}$, as the r.h.s integral exists as an improper
		integral.
		
		Thus, one gets
		\begin{align*}
		\left( \partial_{x_1}^{2} + \partial_{\tilde{x}_2}^2\right) I_1^{i,j} =& \int_{-\infty}^{+\infty} \left((\bi\xi)^2+(\bi\mu_i)^2\right)\frac{e^{\bi|x_1-y_1|\xi}}{A}f_{x_2,y_2}^{i,j}(\xi) d\xi = -k_i^2 I_1^{i,j}.
		\end{align*}
		
		The case for $I_2^{i,j}$ can be similarly analyzed by using the fact that $\mu_1+\mu_2\neq 0$ for	any $\xi\in\mathbb{C}$.
		
		\noindent(3). On $C_r = \{\xi\in\mathbb{C}:\xi=re^{i\theta}, \frac{\pi}{2}<\theta<\pi\}$,
		since $\lim_{r\to\infty}\mu_l/( -\xi\bi ) = 1$ for $l=1,2$, we have
		\[
		\limsup_{r\to\infty}|r^2e^{\bi\mu_l\tM_2}| =
		\limsup_{r\to\infty}|r^2e^{\xi\tM_2}| =
		\limsup_{r\to\infty}r^2e^{r\cos(\theta)\bar{\sigma}_2 - r\sin(\theta)M_2}
		= \limsup_{r\to\infty}r^2e^{-r\min(\bar{\sigma}_2,M_2)} = 0.
		\]
		Thus, for sufficiently large $r$, we could make $|e^{\bi\mu_l\tM_2}|\lesssim \frac{1}{r^{2}}$, 
		so that
		\[
		\left|\frac{e^{\bi|x_1-y_1|\xi}}{A}f_{x_2,y_2}^{i,j}(\xi)\right|\lesssim\frac{1}{r^2}.
		\]
		Therefore,
		\[
		\lim_{r\to\infty}\int_{C_r}\frac{e^{\bi|x_1-y_1|\xi}}{A}f_{x_2,y_2}^{i,j}(\xi)d\xi = 0.
		\]
		Consequently, by Cauchy's theorem, we get
		\[
		I_1^{i,j} = \left( \int_{+\infty\bi}^0 +\int_{0}^{\infty}\right)\frac{e^{\bi|x_1-y_1|\xi}}{A}f_{x_2,y_2}^{i,j}(\xi)d\xi.
		\]
		On the other hand, 
		\begin{align*}
		\limsup_{r\to\infty}\left|\frac{\xi g_{x_2,y_2}^{i,j}}{e^{\xi(|x_2|+|y_2|)}}\right|
		&\eqsim \limsup_{r\to\infty}\left|\frac{\xi e^{\bi\mu_i\sqrt{\tilde{x}_2^2} + \bi\mu_j\sqrt{\tilde{y}_2^2}}}{e^{\xi(|x_2|+|y_2|)}(\mu_1+\mu_2)}\right| \eqsim \limsup_{r\to\infty}\left|\frac{e^{\xi(\sqrt{\tilde{x}_2^2} + \sqrt{\tilde{y}_2^2})}}{e^{\xi(|x_2|+|y_2|)}}\right|\lesssim 1,
		\end{align*}
		so that, for sufficiently large $r$, we can make
		$|g_{x_2,y_2}^{i,j}|\lesssim \frac{e^{r\cos\theta(|x_2|+|y_2|)}}{r}$,
		for $\xi=re^{\bi\theta}$ for $\theta\in(\pi/2,\pi).$
		Thus, 
		\begin{align*}
		&\left|  \int_{C_r}\left|e^{\bi|x_1-y_1|\xi}g_{x_2,y_2}^{i,j}(\xi)\right|d\xi\right|
		=\int_{0}^{\pi/2}e^{-r\sin\theta(|x_2|+|y_2|)}d\theta\\
		&\leq \int_{0}^{\theta_0}e^{-r\theta/2(|x_2|+|y_2|)}d\theta + \int_{\theta_0}^{\pi/2}e^{-r\sin\theta(|x_2|+|y_2|)}d\theta \lesssim \frac{1}{r(|x_2|+|y_2|)} + e^{-r\sin\theta_0(|x_2|+|y_2|)}\to 0,
		\end{align*}
		as $r\to\infty$. Here, $\theta_0>0$ is a sufficiently small constant such that
		$\sin\theta>\theta/2$ for $\theta\in(0,\theta_0)$.
		Again, Cauchy's theorem indicates that 
		\[
		I_2^{i,j} = \left( \int_{+\infty\bi}^0 +\int_{0}^{\infty}\right)e^{\bi|x_1-y_1|\xi}g_{x_2,y_2}^{i,j}(\xi)d\xi.
		\]
		\noindent(4) First, we observe that on the integral path $\xi:+\infty\bi\to
		0\to +\infty$, the following function
		\[
		h_{i,j}^l(\mu_l):=\frac{\mu_i(\mu_l)f_{x_2,y_2}^{i,j}(\sqrt{k_l^2-\mu_l^2})}{A_l(\mu_l)},
		\]
		in a sufficiently small neighborhood of
		$\mu_l=0$, has a removable singularity at $\mu_l=0$ and hence can be extended
		as a holomorphic function of $\mu_l$ in the neighborhood for $l=1,2$. Thus,
		we decompose
		\begin{align*}
		(\partial_{|x_1-y_1|}- \bi k_i)I_1^{i,j} =& \left( \int_{+\infty\bi}^0 +\int_{0}^{\infty}\right) \bi(\xi-k_i)\frac{e^{\bi|x_1-y_1|\xi}}{A}f_{x_2,y_2}^{i,j}(\xi) d\xi\\
		=&\left( \int_{+\infty\bi}^0 + \int_{0}^{k_1-\epsilon_0} + \int_{k_1+\epsilon_0}^{k_2-\epsilon_0} + \int_{k_2+\epsilon_0}^{+\infty} \right)\bi(\xi-k_i)\frac{e^{\bi|x_1-y_1|\xi}}{A}f_{x_2,y_2}^{i,j}(\xi) d\xi\\
		&-\sum_{l=1}^2\int_{k_l-\epsilon_0}^{k_l+\epsilon_0}e^{\bi|x_1-y_1|\xi}\sqrt{\frac{\xi-k_i}{k_i+\xi}} h_{i,j}^l(\sqrt{k_l^2-\xi^2}) d\xi,
		\end{align*}
		where $\epsilon_0$ is a sufficiently small positive constant. On the
		positive imaginary axis, we have the following estimate
		\begin{align*}
		&\left|\int_{+\infty\bi}^0\bi(\xi-k_i)\frac{e^{\bi|x_1-y_1|\xi}}{A}f_{x_2,y_2}^{i,j}(\xi)
     d\xi\right|=\left|\int_{0}^{+\infty}(t\bi-k_i)\frac{e^{-|x_1-y_1|t}}{A_\xi(t\bi)}f_{x_2,y_2}^{i,j}(t\bi) dt\right|\\
		&\lesssim \int_{0}^{+\infty}(t + k_i)e^{-t|x_1-y_1|}dt\lesssim\frac{1}{|x_1-y_1|},
		\end{align*}
		When $\xi\in (0,k_1-\epsilon_0)\cup (k_1+\epsilon_0,
		k_2-\epsilon_0)\cup(k_2+\epsilon_0, +\infty)$,
		$(\xi-k_i)f_{x_2,y_2}^{i,j}(\xi)/A_\xi(\xi)$ is a smooth function of $\xi$
		so that by integration by parts,
		\begin{align*}
		&\left|\left(  \int_{0}^{k_1-\epsilon_0}+\int_{k_1+\epsilon_0}^{k_2-\epsilon_0} + \int_{k_2+\epsilon_0}^{+\infty}\right)\bi(\xi-k_i)\frac{e^{\bi|x_1-y_1|\xi}}{A}f_{x_2,y_2}^{i,j}(\xi)
		d\xi\right|\\
		\leq&\left|\left(  \frac{(\xi-k_i)f_{x_2,y_2}^{i,j}(\xi)}{A_\xi(\xi)}\frac{e^{\bi|x_1-y_1|\xi}}{|x_1-y_1|}\right)\Big|_{0}^{k_1-\epsilon_0}\Big|_{k_1+\epsilon_0}^{k_2-\epsilon_0}\Big|_{k_2+\epsilon_0}^{+\infty}\right|\\
		&+\frac{1}{|x_1-y_1|}\left(  \int_{0}^{k_1-\epsilon_0}+\int_{k_1+\epsilon_0}^{k_2-\epsilon_0} + \int_{k_2+\epsilon_0}^{+\infty}\right)\left| \left(\frac{(\xi-k_i)f_{x_2,y_2}^{i,j}(\xi)}{A_{\xi}(\xi)}\right)'\right| d\xi \lesssim \frac{1}{|x_1-y_1|}.
		\end{align*}
		Here, the finiteness of the integral on the third part of the contour is straightforward, so we omit the details.
		
		On the other hand, 
		\begin{align*}
		&\left|  \left(\int_{k_1-\epsilon_0}^{k_1+\epsilon_0} +\int_{k_2-\epsilon_0}^{k_2+\epsilon_0} \right)e^{\bi|x_1-y_1|\xi}\sqrt{\frac{\xi-k_i}{k_i+\xi}}\frac{\mu_if_{x_2,y_2}^{i,j}}{A} d\xi\right|\\
		\leq&\sum_{l=1}^2\left| \left.\frac{e^{\bi|x_1-y_1|\xi}}{|x_1-y_1|}\sqrt{\frac{\xi-k_i}{k_i+\xi}}\frac{\mu_if_{x_2,y_2}^{i,j}}{A}\right|_{k_l-\epsilon_0}^{k_l+\epsilon_0} \right| +\frac{1}{|x_1-y_1|}\int_{k_l-\epsilon_0}^{k_l+\epsilon_0}\left|\left(  \sqrt{\frac{\xi-k_i}{k_i+\xi}}\right)'h_{i,j}^l(\sqrt{k_l^2-\xi^2})\right|d\xi\\
		&+\frac{1}{|x_1-y_1|}\int_{k_l-\epsilon_0}^{k_l+\epsilon_0}\left|\sqrt{\frac{\xi-k_i}{k_i+\xi}}{ h_{i,j}^l }'(\sqrt{k_l^2-\xi^2}) \frac{\xi}{\sqrt{k_l^2-\xi^2}}\right|d\xi\\
		\lesssim&\frac{1}{|x_1-y_1|}+\frac{1}{|x_1-y_1|}\sum_{l=1}^2\int_{k_l-\epsilon_0}^{k_l+\epsilon_0}\frac{d\xi}{\sqrt{|\xi-k_l|}}\lesssim\frac{1}{|x_1-y_1|}. 
		\end{align*}
		Consequently, we get
		\[
		(\partial_{|x_1-y_1|}-\bi k_i)I_{1}^{i,j} = {\cal
      O}(|x_1-y_1|^{-1}),\quad{\rm as}\ x_1\to\infty.
		\]
		The radiation condition for $I_2^{i,j}$ is similar and much easier to prove
    since $\mu_1+\mu_2$ is nonzero for
		$\xi\in [0,+\infty]\bi\cup[0,+\infty]$; we omit the details.\\
		\noindent(5). We make use of the method of stationary phase. Considering $I_1^{i,j}$, as shown in part (4), we easily get that 
		\begin{align*}
		\left|\left(\int_{+\infty\bi}^0 + \int_{0}^{k_1-\epsilon_0} + \int_{k_1+\epsilon_0}^{k_2-\epsilon_0} + \int_{k_2+\epsilon_0}^{+\infty} \right)\frac{f_{x_2,y_2}^{i,j}(\xi)}{A_{\xi}(\xi)}e^{\bi\xi|x_1-y_1|}d\xi\right| = {\cal O}(|x_1-y_1|^{-1}),\quad{\rm as}\ x_1\to\infty.
		\end{align*}
		On the neighborhood of $\xi=k_i$, we get
		\begin{align*}
		&\left|\int_{k_i-\epsilon_0}^{k_i+\epsilon_0}\frac{e^{\bi\xi|x_1-y_1|}}{\sqrt{k_i^2-\xi^2}}h_{i,j}^i(\sqrt{k_i^2-\xi^2})d\xi\right|\\
		\leq &\left| \int_{0}^{\theta_{ \epsilon_0,1 }} e^{\bi k_i\cos\theta|x_1-y_1|}h_{i,j}^i(k_i\cos\theta)d\theta \right|+\left|\int_{0}^{\theta_{\epsilon_0,2}} e^{\bi k_i\sec\theta|x_1-y_1|}h_{i,j}^i(k_i\sec\theta\bi) \sec\theta d\theta \right|\\
		=&{\cal O}(|x_1-y_1|^{-1/2})\quad{\rm as}\ |x_1-y_1|\to\infty,
		\end{align*}
    where $\theta_{\epsilon_0,1} = \arccos(1-\epsilon_0/k_i)$, and
    $\theta_{\epsilon_0,2}=\arccos((1+\epsilon_0/k_i)^{-1})$. Here, we have used
    {\cite[Prop. 3, Page 334]{stein93}} in the last inequality as for
    sufficiently small $\epsilon_0$, the integration domains contain only one
    stationary point $\theta=0$.
		
		On the neighborhood of $\xi=k_{3-i}$, we get through integration by parts that
		\begin{align*}
		&\left|\int_{k_{3-i}-\epsilon_0}^{k_{3-i}+\epsilon_0}\frac{e^{\bi\xi|x_1-y_1|}}{\sqrt{k_i^2-\xi^2}}h_{i,j}^i(\sqrt{k_i^2-\xi^2})d\xi\right|\leq {\cal O}(|x_1-y_1|^{-1})\quad{\rm as}\ |x_1-y_1|\to\infty.
		\end{align*}
		Consequently, it implies $I^{i,j}_1 = {\cal O}(|x_1-y_1|^{-1/2})$,
		as $|x_1-y_1|\to\infty$.
		
		For $I_2^{i,j}$, since the integrand itself is smooth, one easily
		obtains that $I^{i,j}_2 = {\cal O}(|x_1-y_1|^{-1})$,
		as $|x_1-y_1|\to\infty$.
	\end{proof}
\end{mylemma}

The existence of Green's function for the waveguide problem \eqref{eq:green:pmly}\cb{--\eqref{eq:som}} is now obtained.
\begin{mytheorem}
	The Green's function $G(x,y)$ defined in \eqref{eq:Green:sol:+} is
  well-defined and solves the problem \eqref{eq:green:pmly}\cb{--\eqref{eq:som}}. Furthermore, $G$
  satisfies the following finiteness property
	\[
	G(x,y) = {\cal O}(|x_1-y_1|^{-1/2}),\quad {\rm as}\ |x_1-y_1|\to\infty.
	\]
	\begin{proof}
		Lemma \ref{lemgreen} (1), (2), and (4) show that $I_l^{i,j}(x_1,\tilde{x}_2;y_1,\tilde{y}_2),
    l=1,2$, are well-defined and satisfy the Helmholtz equation 
		\[
		\partial_{x_1}(\alpha_2\partial_{x_1}I_l^{i,j}) +
		\partial_{x_2}(\alpha_2^{-1}\partial_{x_2}I_l^{i,j}) + \alpha_2k_i^2
		I_l^{i,j} = 0,
		\]
		and the radiation condition
		\[
		\lim_{r\rightarrow \infty}\sqrt{|x_1-y_1|}\left( \frac{\partial I_l^{i,j}}{\partial |x_1-y_1|} -
		\bi k_lI_l^{i,j} \right) = 0.
		\]
		According to \cite{lassas_somersalo_2001}, since
		\[
		\Phi(k_l,(x_1,\tilde{x}_2);(y_1,\tilde{y}_2))=\frac{\bi}{4}H_0^{(1)}\left(k_l\sqrt{(x_1-y_1)^2 + (\tilde{x}_2-\tilde{y}_2)^2}\right),
		\] 
		directly solves 
		\[
		\partial_{x_1}(\alpha_2\partial_{x_1}\Phi) + \partial_{x_2}(\alpha_2^{-1}\partial_{x_2}\Phi) + \alpha_2k_l^2 \Phi = - \delta_y(x), 
		\]
		and satisfies
		\[
		\lim_{|x_1-y_1|}\sqrt{|x_1-y_1|}\left( \frac{\partial \Phi}{\partial |x_1-y_1|} -
		\bi k_l\Phi \right) = 0.
		\]
		The verification that $G$ is well-defined and solves problem
    \eqref{eq:green:pmly}\cb{--\eqref{eq:som}} becomes straightforward.
		
		The finiteness property can be readily proved by considering the finiteness
    of $I_l^{i,j}$ for $i,j,l=1,2$, and that
		\[
		\frac{\bi}{4}H_0^{(1)}(k_l\sqrt{(x_1-y_1)^2 + (\tilde{x}_2-\tilde{y}_2)^2})=
    {\cal O}(|x_1-y_1|^{-1/2}),\quad{\rm as}\ |x_1-y_1|\to\infty.
		\]
	\end{proof}
\end{mytheorem}
\subsection{Existence of the Green's function with UPML}
In this section, we discuss how to construct the Green's function for the uniaxial PML problem \eqref{eq:upml:green} through the waveguide problem \eqref{eq:green:pmly}\cb{--\eqref{eq:som}}. To study the Green's function with rectangular PML truncation, we first analytically extend the domain of $G(x,y)$ for the waveguide problem \eqref{eq:green:pmly}\cb{--\eqref{eq:som}} from $x_1,y_1\in \mathbb{R}$ to $\tx_1,\ty_1\in
\mathbb{C}^{++}\cup\mathbb{C}^{--}$ by the PML transformation
\[
\tilde{x}_1=x_1 + \bi \int_{0}^{x_1}\sigma_1^p(t)dt,\quad\tilde{y}_1=y_1 + \bi \int_{0}^{y_1}\sigma_1^p(t)dt,
\]
where the absorbing function along the $x_1$-axis takes the form 
\[
\sigma_1^p(x_1) = \left\{
\begin{array}{ll}
\sigma_1(x_1) & |x_1|<M_1;\\
\sigma_1(x_1-2nM_1) &(2n-1)M_1<x_1<(2n+1)M_1,\quad n\in \mathbb{Z}\backslash\{0\}.
\end{array}
\right.
\]
One issue is that the real path used in (\ref{eq:Green:res:++}-\ref{eq:Green:layer:-+}) is
not usable to perform the extension since $e^{\bi(\tilde{x}_1-\tilde{y}_1)\xi}$
blows up in one of the two cases $\xi\to\pm \infty$. To resolve this, we make
use of Lemma \ref{lemgreen} by changing the real path to
\[
{\rm EXT}: +\infty\bi \to 0 \to +\infty,
\]
so that we can define, for instance,
\begin{equation}
\label{eq:Greennew:res:++}
G_{\rm res}^{1,1}(\tilde{x};\tilde{y}) = -\Phi(k_1,(\tilde{x}_1, 2\tilde{M}_2-\tilde{x}_2); (\tilde{y}_1,\tilde{y}_2)) + \frac{\bi}{4\pi}\int_{\rm EXT}\frac{e^{\bi\xi(\tilde{x}_1-\tilde{y}_1)^+}}{A}f_{x_2,y_2}^{1,1}(\xi)d\xi,
\end{equation}
where we recall $a^+=\sqrt{a^2}$ is {defined in} the branch with a nonnegative real part. One similarly defines the other terms $G_{\rm res}^{i,j}$ and $G_{\rm layer}^{i,j}$ for $i,j=1,2$.

Consequently, we can make an analytic extenstion of $G(x,y)$ by defining
\begin{align}
\label{eq:Green:sol:ext}
\tilde{G}(x,y)=& G_{\rm layer}^{i,j}(\tilde{x},\tilde{y})+ G_{\rm res}^{i,j}(\tilde{x},\tilde{y}),
\end{align} 
for $x\in \Omega^i, y\in \Omega^j$. By following a similar argument as in Lemma \ref{lemgreen}, we can show that $\tilde{G}$ is well-defined and satisfies the modified Helmholtz equation.
\begin{mytheorem}
	$\tilde{G}$ solves the following problem
	\begin{align}\label{eq:green:ypml:ext}
	\begin{cases}
	&\frac{\partial}{\partial x_1}\left( \frac{\alpha_2}{\alpha_1^p}\frac{\partial \tilde{G}}{\partial x_1}
	\right)+ \frac{\partial}{\partial x_2}\left( \frac{\alpha_1^p}{\alpha_2}\frac{\partial
		\tilde{G}}{\partial x_2} \right) + \alpha_1^p\alpha_2 k^2\tilde{G} = -\delta_y(x),\quad x,y\in\mathbb{R}\times(-M_2,M_2),\\
	&[\tilde{G}] = 0,\quad [\partial_{x_2} \tilde{G}] = 0,\quad{\rm on}\quad x_2=0,\\
	&\tilde{G} = 0,\quad{\rm on}\quad x_2=\pm M_2,
	\end{cases}
	\end{align}
	where $\alpha_1^p(x_1)=1+\bi\sigma_1^p(x_1)$.  
\end{mytheorem}
To construct the Green's function for equation \eqref{eq:upml:green}, we define an infinite series based on
$\tilde{G}$
\begin{equation}
\label{eq:inf:series:G}
G_{\rm PML}(x,y) = \sum_{n=-\infty}^{\infty}\left[ -\tilde{G}(x'+ne_1,y) + \tilde{G}(x+ne_1,y)  \right], \mbox{ for } x,y\in B_{\rm ex},
\end{equation}
where
$x' = (2M_1 - x_1,x_2),$ and $e_1 = (4M_1,0)$.
For $n\in\mathbb{Z}$ and $x_2y_2\geq 0$, define
\[
a_{2n}^{x_1,y_1} = (4n\tM_1+\tx_1-\ty_1)^+\ {\rm and}\ 
a_{2n+1}^{x_1,y_1}= ((4n+2)\tM_1-\tx_1-\ty_1)^+,
\]
and
\[
b_1^{x_2,y_2} = (\tilde{x}_2 - \tilde{y}_2)^+,   b_2^{x_2,y_2} = (\tilde{x}_2 +
\tilde{y}_2)^+,\ {\rm and}\ b_3^{x_2,y_2} = 2\tilde{M}_2 - b_2^{x_2,y_2}.
\] {By properly rearranging the terms in (\ref{eq:inf:series:G}),
    }
we obtain that for $i=1,2$,  when $x,y\in B_{ex}^i$, it holds
\begin{align}
\label{eq:infseries:G:1}
G_{\rm PML}(x,y) =& G_{\rm layer}^{i,i}(\tx,\ty) + \frac{\bi}{4\pi}\int_{\rm EXT}e^{\bi\xi a_0^{x_1,y_1}}\frac{f_{x_2,y_2}^{i,i}(\xi)}{A} d\xi\nonumber\\
&+\frac{\bi}{4\pi}\sum_{n=-\infty,n\neq 0}^{\infty}(-1)^n\int_{\rm EXT} e^{\bi\xi a_{n}^{x_1,y_1}}\left(\frac{f_{x_2,y_2}^{i,i}(\xi)}{A} +  g_{x_2,y_2}^{i,i}(\xi)\right)d\xi \nonumber\\
&+\frac{\bi}{4\pi}\sum_{n=-\infty,n\neq 0}^{\infty}(-1)^n\sum_{j=1}^2H_0^{(1)}(k_i\sqrt{( a_{n}^{x_1,y_1} )^2 + ( b_j^{x_2,y_2} )^2})\nonumber\\
&-\frac{\bi}{4\pi}\sum_{n=-\infty}^{\infty}(-1)^nH_0^{(1)}(k_i\sqrt{( a_{n}^{x_1,y_1} )^2 + ( b_3^{x_2,y_2} )^2}),
\end{align}
and when $x\in B_{\rm ex}^i$ and $y\in B_{\rm ex}^{3-i}$ (or vice versa),
\begin{align}
\label{eq:infseries:G:2}
G_{\rm PML}(x,y) =&G_{\rm layer}^{i,3-i}(\tx,\ty) +\frac{\bi}{2\pi}\int_{\rm EXT}e^{\bi \xi a_n^{x_1,y_1}}\frac{f_{x_2,y_2}^{i,3-i}(\xi)}{A}d\xi\nonumber\\
&+\frac{\bi}{2\pi}\sum_{n=-\infty,n\neq 0}^{\infty}(-1)^n\int_{\rm EXT} e^{\bi\xi a_{n}^{x_1,y_1}}\left(  \frac{f_{x_2,y_2}^{i,3-i}(\xi)}{A}+g_{x_2,y_2}^{i,3-i}(\xi)\right)d\xi.
\end{align}
We now show {that the two series in (\ref{eq:infseries:G:1}) and
  (\ref{eq:infseries:G:2}) are absolutely convergent so that the rearrangement
  of terms in (\ref{eq:inf:series:G}) is reasonable,} and that $G_{\rm PML}$ is
the Green's function that satisfies \eqref{eq:upml:green}. To this purpose, we
need to estimate the terms in (\ref{eq:infseries:G:1}) and
(\ref{eq:infseries:G:2}). Before this, we first extend the domain from
$\ol{\mCmp\cup\mCpm}$ to a larger one.
\begin{mylemma}\label{lemA10}
	There exists a constant $\delta\in(0,1)$, such that $A(\xi)\neq 0$ for
	any $\xi\in E_{\delta} = \{\xi\in\ol{\mCpp}: {\rm Re}(\xi)\leq
	\delta k_1,{\rm Im}(\xi)\leq \delta k_1\}$. Furthermore, 
	\[
	|\mu_1+\mu_2|\lesssim|A|,
	\]
	for any $\xi\in E_{\delta}$.
	\begin{proof}
		We first prove the existence of $E_{\delta}$. Suppose otherwise there
		exist a sequence of $\{\delta_n\}_{n=1}^{\infty}$ with $\delta_n>0$ and
		$\lim_{n\to \infty}\delta_n=0$. A sequence of
		$\{\xi_n\}$ with $\xi_n\in\ol{\mCpp}$ and $\max({\rm Re}(\xi),{\rm
			Im}(\xi))\leq \delta_nk_1$, such that
		$\lim_{n\to\infty}A_\xi(\xi_n) = 0$.
		As $\lim_{n\to\infty}\xi_n = 0$, we directly get $A(0)=0$ which is in
		contradiction with Lemma \ref{lemA1}. Consequently, there must exist a box
		$E_{\delta}$ with $\delta>0$ such that $A\neq 0$ for any $\xi\in
		E_{\delta}$.
		
		We now prove the estimate. Suppose there exist a sequence of
		$\{\xi_n\}_{n=1}^{\infty}$ with $\lim_{n\to\infty}\xi_n=\xi_0\in
		E_{\delta}$ such that
		\[
		\lim_{n\to\infty} \left| \frac{A(\xi_n)}{\mu_{1,n}+\mu_{2,n}} \right|
		= 0,
		\]
		where $\mu_{l,n}=\sqrt{k_l^2-\xi_n^2}$. Since $\mu_{1,n}+\mu_{2,n}\geq
		\sqrt{k_2^2-k_1^2}>0$, we have to enforce $A(\xi_0)=0$, which is
		impossible due to the choice of $E_{\delta}$.
	\end{proof}
\end{mylemma}

\begin{mylemma}\label{lemA11}
	For all $n\in\mathbb{Z}\backslash\{0\}$, $l,m\in\{0,1,2\}$ and $x\in
	B_{\rm ex}^i$ and $y\in B_{\rm ex}^j$, $i,j=1,2$, it holds
	\[
	\left| \xi^l\mu_i^me^{\bi\xi a_n^{x_1,y_1}}\left(\frac{f_{x_2,y_2}^{i,j}(\xi)}{A} +
	g_{x_2,y_2}^{i,j}(\xi)\right) \right|\lesssim\frac{(\xi_1^2+\xi_2^2)^{l/2}(k_2^2+\xi_1^2+\xi_2^2)^{m/2} }{\sqrt{k_2^2-k_1^2}} e^{-2|n|M_1\xi_2 - 2|n|\bar{\sigma}_1\xi_1},
	\]
	for any $\xi=\xi_1+\bi\xi_2\in E_{\delta}$ and 
	\[
	\left|\xi^l\mu_i^m e^{\bi\xi a_n^{x_1,y_1}}\left(\frac{f_{x_2,y_2}^{i,j}(\xi)}{A} +
	g_{x_2,y_2}^{i,j}(\xi)\right) \right|\lesssim\frac{(\xi_1^2+\xi_2^2)^{l/2}(k_2^2+\xi_1^2+\xi_2^2)^{m/2}}{|\mu_i|} e^{-2|n|M_1\xi_2 - 2|n|\bar{\sigma}_1\xi_1},
	\]
	for any $\xi=\xi_1+\bi \xi_2\in \partial \mCpp$, where $\partial \mCpp$
  consists of the positive real and imaginary axis.
	\begin{proof}
		For $\xi=\xi_1+\bi \xi_2\in \partial \mCpp$, it holds 
		\[
		\left|\frac{f_{x_2,y_2}^{i,j}(\xi)}{A} + g_{x_2,y_2}^{i,j}(\xi)\right|\lesssim
		\sum_{l=1}^{2}\frac{1}{|\mu_i|}\left|e^{\bi\mu_l\tilde{M}_2}\right|\lesssim
		\frac{1}{|\mu_i|} + \frac{1}{|\mu_1+\mu_2|}\lesssim\frac{1}{|\mu_i|}.
		\]
		In $E_{\delta}$, since $f_{x_2,y_2}^{i,j}$ has no singularities, we have
		by Lemma \ref{lemA10} that
		\[
		\left|\frac{f_{x_2,y_2}^{i,j}(\xi)}{A}+g_{x_2,y_2}^{i,j}(\xi)\right|\lesssim\frac{1}{\sqrt{k_2^2-k_1^2}}.
		\]
		For any $\xi=\xi_1+\bi\xi_2\in E_{\delta}\cup \partial \mCpp$, it holds
		$\left|e^{\bi\xi a_n^{x_1,y_1}}\right|\lesssim e^{-2|n|M_1\xi_2 - 2|n|\bar{\sigma}_1\xi_1}$,
		and 
		\[
		|\xi^l\mu_i^m|\leq (\xi_1^2+\xi_2^2)^{l/2}|\mu_i|^{m/2}\leq (\xi_1^2+\xi_2^2)^{l/2}(k_2^2+\xi_1^2+\xi_2^2)^{m/2}. 
		\]
		Consequently, the estimates follow from the above inequalities.
	\end{proof}
\end{mylemma}
The following lemma shows the contribution from all the other terms except $n=0$ in the infinite series $G_{\rm PML}$ is exponentially small. 
\begin{mylemma}\label{lemA12}
	For all $n\in\mathbb{Z}\backslash\{0\}$ and $l,m\in\{0,1,2\}$,
	\[
	\left|\int_{\rm EXT}\xi^l\mu_i^me^{\bi \xi a_n^{x_1,y_1}}\left(\frac{f_{x_2,y_2}^{i,j}(\xi)}{A} + g_{x_2,y_2}^{i,j}(\xi)\right)d\xi\right|\lesssim
	\left(  e^{-2|n|M_1\delta k_1}+ e^{-2|n|\bar{\sigma}_1\delta k_1}\right).
	\]
	\begin{proof}
		We define the following path:
		\[
		P_{\delta}: \xi\in +\infty\bi\to \delta k_1\bi \to \delta k_1\bi + \delta k_1\to
		\delta k_1 \to \infty.
		\]
		As $A\neq 0$ in $E_{\delta}$, we get by Cauchy's theorem that
		\begin{align*}
		&\int_{\rm EXT}\xi^l\mu_i^me^{\bi \xi a_n^{x_1,y_1}}\left(\frac{f_{x_2,y_2}^{i,j}(\xi)}{A} + g_{x_2,y_2}^{i,j}(\xi)\right)d\xi
		= \int_{P_{\delta}}\xi^l\mu_i^me^{\bi \xi a_n^{x_1,y_1}}\left(\frac{f_{x_2,y_2}^{i,j}(\xi)}{A} + g_{x_2,y_2}^{i,j}(\xi)\right)d\xi.
		\end{align*}
		By Lemma \ref{lemA11}, we get the following estimates
		\begin{align*}
		&\left|\int_{P_{\delta}}\xi^l\mu_i^me^{\bi \xi a_n^{x_1,y_1}}\left(\frac{f_{x_2,y_2}^{i,j}(\xi)}{A} + g_{x_2,y_2}^{i,j}(\xi)\right)d\xi\right|\\
		\lesssim& \int_{\delta k_1}^{+\infty} \xi_2^l(k_2^2+\xi_2^2)^{m/2}e^{-2|n|M_1\xi_2}d\xi_2 +\int_{0}^{\delta k_1}e^{-2|n|M_1k_1\delta }e^{-2|n|\bar{\sigma}_1\xi_1}d\xi_1 \\
		&+ \int_{0}^{\delta k_1}e^{-2|n|\bar{\sigma}_1\delta k_1}e^{-2|n|M_1k_1\xi_2}d\xi_2+\int_{\delta k_1}^{+\infty} \frac{\xi_1^l(k_2^2+\xi_1^2)^{m/2}e^{-2|n|\bar{\sigma}_1\xi_1}}{|\sqrt{k_i^2-\xi_1^2}|}d\xi\\
		\lesssim& e^{-2|n|M_1\delta k_1} + e^{-2|n|M_1k_1\delta } + e^{-2|n|\bar{\sigma}_1\delta k_1} + e^{-2|n|\bar{\sigma}_1\delta k_1}.
		\end{align*}
	\end{proof}
\end{mylemma}
The following lemma studies properties of $a_{n}^{x_1,y_1}$ and $b_j^{x_2,y_2}$.
\begin{mylemma}\label{lemab}
	Suppose $x,y\in B_{\rm ex}^i$ for $i=1,2$. For any $n\in\mathbb{Z}\backslash\{0\}$ and $j=1,2,3$:
		\begin{align}
		\label{eq:im:anbj:1}
		{\rm Im}\left( \sqrt{( a_n^{x_1,y_1} )^2 + ( b_j^{x_2,y_2} )^2} \right)\geq \frac{(2|n|-2)^2\bar{\sigma}_1}{\sqrt{(2|n|+2)^2+4(M_2/M_1)^2}}. 
		\end{align}
	\begin{proof}
	At first, it is easy to see that
		\begin{align*}
		&{\rm Im}(a_{n}^{x_1,y_1}) \in [( 2|n|-2 )\bar{\sigma}_1,( 2|n|+2 )\bar{\sigma}_1],\quad
		{\rm Re}(a_{n}^{x_1,y_1}) \in [( 2|n|-2 )M_1,( 2|n|+2 )M_1],\\
		&{\rm Re}(b_j^{x_2,y_2}) \in [0, 2M_2],\quad {\rm Im}(b_j^{x_2,y_2}) \in [0, 2\bar{\sigma}_2].
		\end{align*}
		This and {\cite[Lemma 6.1]{chenzheng}} immediately give (\ref{eq:im:anbj:1}).
	\end{proof}
\end{mylemma}

We now recall an important lemma from {\color{red}\cite{cheliu05}}.
\begin{mylemma}\label{lemhankel}
	For any $\nu\in\mathbb{R}$, $z\in \mathbb{C}^{++}$, and $\Theta\in\mathbb{R}$
	such that $0<\Theta\leq |z|$, we have 
	\begin{equation}
	\label{eq:est:hankel}
	|H_\nu^{(1)}(z)|\leq e^{-{\rm Im}(z)\left( 1-\frac{\Theta^2}{|z|^2} \right)^{1/2}}|H_{\nu}^{(1)}(\Theta)|.
	\end{equation}
\end{mylemma}
One application of Lemma \ref{lemhankel} is given by the following.
\begin{mycorollary}\label{coro1}
	Suppose $x,y\in B_{\rm ex}^i$. There exist a positive integer $N$, such that for all $n\geq N$,
	\[
	H_\nu^{(1)}\left(k_i\sqrt{(a_n^{x_1,y_1})^2 + (b_j^{x_2,y_2})^2}\right)\leq
	e^{-n\bar{\sigma}_1}|H_{\nu}^{(1)}(\bar{\sigma}_1)|,
	\]
	for $j=1,2,3$ and for any $\nu\in\mathbb{R}$.
	\begin{proof}
		{Eq.~(\ref{eq:im:anbj:1})}  in Lemma \ref{lemab} indicates that 
		\[
		\liminf_{n\to\infty}\frac{{\rm Im}(\sqrt{(a_n^{x_1,y_1})^2 +
				(b_j^{x_2,y_2})^2})}{2|n|\bar{\sigma}_1}\geq 1,
		\]
		so that for sufficiently large $n$, we have
		\[
		\left|\sqrt{(a_n^{x_1,y_1})^2 + (b_j^{x_2,y_2})^2}\right|\geq {\rm
			Im}(\sqrt{(a_n^{x_1,y_1})^2 + (b_j^{x_2,y_2})^2})\geq
		|n|\bar{\sigma}_1\geq \bar{\sigma}_1,
		\]
		The estimate immediately follows from Lemma \ref{lemhankel}.
	\end{proof}
\end{mycorollary}

Combining all the results above, we now show $G_{\rm PML}$ is well-defined.

\begin{mytheorem}\label{THMGPML}
	It holds that:
	\begin{itemize}
		\item[(1).] The series in $G_{\rm PML}$ defined in (\ref{eq:infseries:G:1}) and
		(\ref{eq:infseries:G:2}) is absolutely convergent for any $x\in B_{\rm ex}^{i},y\in
		B_{\rm ex}^{j}$ with $x\neq y$ for $i,j=1,2$.
		\item[(2).] Suppose $y\in B_{\rm ex}^i$ for $i=1,2$. Then $G_{\rm PML}(x;y)\in
		H^2(B_{\rm ex}\backslash\overline{B(y,\varepsilon)})$ for any $\varepsilon>0$, where $B(y,\varepsilon)$ denotes a disk centered at $y$ with radius $\varepsilon$. \cb{Moreover, 
		$G_{\rm PML}(x;y)-\Phi(k_i,\tx;\ty)\in W^{2,\infty}(B_{\rm ex}^i)$}.
		\item[(3).] $G_{\rm PML}(x;y)$ solves the {truncated PML problem \eqref{eq:upml:green}.}
	\end{itemize}
	\begin{proof}
			(1). Lemma \ref{lemA12} with $l=m=0$ and Corollary \ref{coro1} with $\nu =
      0$ directly imply the
			series in (\ref{eq:infseries:G:1}) and (\ref{eq:infseries:G:2}) are
			absolutely convergent, which explains the validity of rearranging terms in
			(\ref{eq:inf:series:G}) to arrive at (\ref{eq:infseries:G:1}) and (\ref{eq:infseries:G:2}).
			
			(2). We can make use of the facts that $\tilde{x}_j\in W^{2,\infty}(B_{\rm
        ex})$, Lemma \ref{lemA12} with $0\leq l,m\leq 2$ and Corollary
      \ref{coro1} with $\nu=1,2$ to see the results.
			
			(3). The reason that $G_{\rm PML}$ {satisfies \eqref{eq:upml:green}}
			based on the fact that $\tilde{G}$ satisfies (\ref{eq:green:ypml:ext}) and
			the differentiation of $G_{\rm PML}$ can be exchanged with the summation in
			(\ref{eq:inf:series:G}). The interface condition is satisfied by construction.
			
			We now verify the zero boundary condition {in \eqref{eq:upml:green}}. On $x_2=\pm
			M_2$, $G_{\rm PML}=0$ since $\tilde{G}(x,\pm M_2)=0$. On $x_1=M_1$, we
			get 
			\begin{align*}
			G_{\rm PML}(x;y) &= \sum_{n=-\infty}^{\infty}\tilde{G}(x+ne_1,y) - \sum_{n=-\infty}^{\infty}\tilde{G}(x'+ne_1,y)\\
			&=\sum_{n=-\infty}^{\infty}\tilde{G}(((4n+1)M_1,x_2);y) - \sum_{n=-\infty}^{\infty}\tilde{G}(((4n+1)M_1,x_2);y)=0;
			\end{align*}
			one similarly verifies that $G_{\rm PML}(x;y)=0$ on $x_1=-M_1$.
	\end{proof}
	
\end{mytheorem}
\subsection{Wellposedness of the UPML problem}
We are ready to analyze the well-posedness of the layered scattering problem \eqref{eq:upml} with UPML.
\begin{proof}[Proof of Theorem \ref{wellposethm}]
  It is easy to see that $a(\tu,v)$ in (\ref{eq:varform:a}) satisfies the Garding inequality and thus is
  a Fredholm operator of index zero \cite[Thm 2.34]{McLean2000}. Therefore, to prove the existence, we only
  need to show the uniqueness. \cb{It suffices to show that the following problem has only zero solution: Find  $w\in H^{1}_0(B_{\rm ex})$ such that
  \begin{equation}
    \label{eq:vf:u:pf}
    a(w,v)=0,\forall v\in H_0^1(B_{\rm ex}).
  \end{equation}
  Since the coefficient ${\bf A}$ is Lipschitz, the regularity theory of elliptic equations implies that $w\in H^2(B_{\rm ex})\cap H_0^1(B_{\rm ex})$. Clearly, $w$ satisfies the following equation 
 \begin{align}\label{eq:w}
\nabla\cdot({\bf A}\nabla w) + \alpha_1\alpha_2 k^2 w= 0\quad\text{ in } B_{\rm ex}^j, \;j=1,2.
 \end{align}
 We claim that
  $w(y)=0$ for any $y\in B_{\rm ex}^j, j=1$ or $2$. }Let $\epsilon$ be so small that $B(y,\epsilon)\subset B_{\rm ex}^j$. Since $G_{\rm PML}(\cdot;y)$ solves \eqref{eq:green:pmly}, its restriction on  $B_{\rm ex}^{j,\epsilon}=B_{\rm ex}^j\setminus \overline{B(y,\epsilon)}$ belongs to $H^2(B_{\rm ex}^{j,\epsilon})$. Clearly,
\begin{align}\label{eq:GPML}
\nabla\cdot({\bf A}\nabla G_{\rm PML}(\cdot;y)) + \alpha_1\alpha_2 k^2 G_{\rm PML}(\cdot;y)= 0\quad\text{ in } B_{\rm ex}^{j,\epsilon}.
\end{align}
Let ${\bm \nu }^c={\bf A}^T{\bm \nu}$ and ${\bm \nu}$ denotes the outer unit normal vector to $\partial B(y,\epsilon)$.
It follows from the second Green's identity and \eqref{eq:w}--\eqref{eq:GPML} that
\begin{align}
  0 = \int_{\partial B(y,\epsilon)}\partial_{{\bm \nu}^c}wG_{\rm PML}(x;y)ds(x)-\int_{\partial B(y,\epsilon)}w\partial_{{\bm \nu}^c} G_{\rm PML}(\cdot;y)ds(x),
\end{align}
\cb{where $\partial_{{\bm \nu}^c}=\nabla\cdot {\bm \nu}^c$.} By Theorem
\ref{THMGPML} (2), for $y\in B_{\rm ex}^j$, as $\epsilon \to 0$,
    \[
      \int_{\partial B(y,\epsilon)}\partial_{{\bm \nu}^c}w[G_{\rm
        PML}(x;y)ds(x) - \Phi(k_j,\tilde{x};\tilde{y})] ds(x) \to 0.
    \]
    and 
    \[
      \int_{\partial B(y,\epsilon)}w[\partial_{{\bm \nu}^c} G_{\rm PML}(\cdot;y)-\partial_{{\bm \nu}^c} \Phi(k_j,\tilde{x};\tilde{y})] ds(x) \to 0.
    \]
    On the other hand, for sufficiently small $\epsilon>0$, $w$ solves the
    PML-transformed Helmholtz equation with wavenumber $k_j$ while
    $\Phi(k_j,\tx;\ty)$ is the associate PML-transformed free-space Green's
    function for medium $k_j$ \cite{lassas_somersalo_2001}, we see from the
    Green's representation formula \cite[Prop. 2]{Lu2018Perfectly} that
    \[
      w(y)=\int_{\partial B(y,\epsilon)}\left[ \partial_{{\bm \nu}^c}
        w\Phi(k_j,\tilde{x};\tilde{y}) - w\partial_{{\bm \nu}^c}
        \Phi(k_j,\tilde{x};\tilde{y})\right] ds(x).
    \]
    Consequently,
    \[
      w(y)=\lim_{\epsilon\to 0} \int_{\partial B(y,\epsilon)}\partial_{{\bm
          \nu}^c}wG_{\rm PML}(x;y)ds(x)-\int_{\partial
        B(y,\epsilon)}w\partial_{{\bm \nu}^c} G_{\rm PML}(\cdot;y)ds(x) = 0,
    \]
    for all $y\in B^j_{\rm ex}, j=1,2$. The continuity then implies that
    $w\equiv 0$.
\end{proof}

  \section{Convergence Analysis}

  Throughout this section, we will suppose the previous assumptions
  (\ref{eq:assump1}-\ref{eq:assump3}) hold and will show how rapidly the PML
  solution $\tu$ converges to the true solution $u$ in the physical domain
  $B_{\rm in}$ as $\bar{\sigma}_j$ or $d_j$ increase for $j=1,2$. For this
  purpose, we have to reconsider the properties of $A$ and $f_{x_2,y_2}^{i,j}$
  with $i,j=1,2$ by restricting to the special case where the target point $x\in
  B_{\rm in}$ and the source point $y\in D$. Unlike the previous estimates in
  section 3, much more delicate analysis must be made to estimate the residual
  terms in (\ref{eq:infseries:G:1}) and (\ref{eq:infseries:G:2}) since now
  $k_j$, $L_j, d_j$, and $\sigma_j$ for $j=1,2$ may vary; \cb{we emphasize that
  $\kappa=k_2/k_1>1$ is fixed in this section, and the generic constant $C$
  defined in notations $\lesssim$, $\gtrsim$, and $\eqsim$ are now indepedent of
  $k_j, L_j, d_j$ and $\sigma_j$ for $j=1,2$}.

\subsection{Properties of $A$ and $f_{x_2,y_2}^{i,j}$}

The following Lemma \ref{thmA41} is related to Lemma \ref{thmA5}. \cb{However, the two lemmas
  differ significantly from each other in several aspects}: in Lemma~\ref{thmA41}, (1) $x$ and $y$ are restricted in $B_{\rm in}$ and $D$, respectively; (2)
we look for a lower bound of $A$ that holds uniformly as the three parameters
$\xi$, $\bar{\sigma}_2$ and $d_2$ (or $M_2$) vary; (3) $\xi$ is defined in a
slightly smaller region; (4) we take $z=2\tilde{M}_2$ here but $|z|\lesssim
|\tilde{M}_2|$ in Lemma \ref{thmA5}.
\begin{mylemma}\label{thmA41}
	It holds
	\[
	\max\left\{\left|(\epsilon_1-1)\mu_2\right|,
	\left|(\epsilon_2-1)\mu_1\right|,\left|\frac{\mu_1\mu_2}{\mu_1+\mu_2}\right|\right\}\lesssim
	|A(\xi)|,
	\]
	for all $\xi\in\ol{\mathbb{C}^{+-}\cup\mathbb{C}^{-+}}\backslash
  \cup_{j=1}^2\{\xi\in\mathbb{C}\backslash\mathbb{R}:
  |\mu_j|<\varepsilon_0k_1\}$, $\bar{\sigma}_2\gtrsim k_1^{-1}$, and $M_2\gtrsim
  k_1^{-1}$, where the positive constant $\varepsilon_0\eqsim 1$.
	\begin{proof}
		By the definitions of $\mu_j$, $\epsilon_j$ for $j=1,2$, and $A$ in
		(\ref{eq:def:A}), we note that
		\begin{align*}
		\mu_j(\xi;k_j) &= k_1\mu_j(\frac{\xi}{k_1};\frac{k_j}{k_1}),\quad \epsilon_j(\xi;k_j,\bar{\sigma}_2,M_2) = \epsilon_j(\frac{\xi}{k_1};\frac{k_j}{k_1},k_1\bar{\sigma}_2,k_1M_2),\\
		A(\xi;k_1,k_2,\bar{\sigma}_2,M_2) &= k_1A(\frac{\xi}{k_1};1,\kappa, k_1\bar{\sigma}_2,k_1M_2),
		\end{align*}
		where we recall that $\kappa = k_2/k_1>1$ is a fixed constant and $A(\xi;k_1,k_2,\bar{\sigma}_2,M_2)$ is used to emphasize that $A$
		depends on the model parameters $k_1$, $k_2$, $\bar{\sigma}_2$ and $M_2$, etc.. Thus, we only
		consider the case when $k_1=1$ and $k_2=\kappa$ in the following and denote by $A(\xi;\bar{\sigma}_2,M_{2})=A(\xi;1,\kappa,\bar{\sigma}_2,M_2)$.

		Suppose there exist three sequences $\{\xi_n\}_{n=1}^{\infty}$,
    $\{\bar{\sigma}_{2,n}\}_{n=1}^{\infty}$, and $\{d_{2,n}\}_{n=1}^{\infty}$
    with $\xi_n\to\xi_0$, $\bar{\sigma}_{2,n}\to S\gtrsim 1$, and $M_{2,n}\to
    M\gtrsim 1$, such that
		\begin{equation}
		\label{eq:A:lim}
		\lim_{n\to\infty}\left|
		\frac{A(\xi_n;\bar{\sigma}_{2,n},M_{2,n})(\mu_1+\mu_2)}{\mu_1\mu_2} \right| = 0.
		\end{equation}
		Suppose $\max(S,M)=\infty$ as by similar arguments in Lemma \ref{thmA5},
    case $\max(S,M)<+\infty$ can be proved to be impossible. We now distinguish
    two cases:

			(I). $|\xi_0|<\infty$ but $\xi_0\notin\{\pm
			k_1,\pm k_2\}$. If $\xi_0\in(k_2,\infty)\cup (-\infty, -k_2)$, then
			${\rm Re}(\mu_1(\xi_0)) = {\rm Re}(\mu_2(\xi_0)) = 0$,
			so that $|\epsilon_{j,n}|=e^{-2M_{2,n}{\rm Im}(\mu_j(\xi_n))}=e^{-2 M\sqrt{\xi_0^2-k_j^2}}<1$. Then, the proof in Lemma \ref{lemA1} indicates that
			\begin{align*}
			\liminf_{n\to\infty}\left|\frac{A_n}{(1-\epsilon_{1,n})(1-\epsilon_{2,n})}\right|
			&\geq \liminf_{n\to\infty}\sum_{j=1}^2
			\frac{(1-|\epsilon_{j,n}|^2)}{1+|\epsilon_{j,n}|^2}{\rm Im}(\mu_{j,n}) \geq\sum_{j=1}^2
			\frac{(1-e^{-4M\sqrt{\xi_0^2-k_j^2}})}{1+e^{-4M\sqrt{\xi_0^2-k_j^2}}}\sqrt{\xi_0^2-k_j^2},
			\end{align*}
			which implies $A(\xi_0;S, M)>0$. The case when $\xi_0\in
      (-k_1,0)\cup(0,k_1)\cup(-\infty\bi,\infty\bi)$ can be analyzed similarly.
      Now suppose $\xi_0\in\mathbb{C}^{+-}\cup\mathbb{C}^{-+}$ so that we must
      have $\epsilon_{j,n}\to 0$ as $n\to\infty$ for both $j=1,2$. Then,
      $\lim_{n\to\infty}|A_n|=|\mu_1(\xi_0)+\mu_2(\xi_0)|\geq
      \sqrt{k_2^2-k_1^2}>0$.

			(II). $\xi_0\in\{\pm k_1,\pm k_2\}$ so that
			$\xi_{n}\in\mathbb{R}$ (for sufficiently large $n$). Suppose
			first $\xi_0=\pm k_1$ such that $\mu_1(\xi_n)\to 0$ as $n\to\infty$. Then,
			$\lim_{n\to\infty}|\epsilon_{2,n}| = e^{-\sqrt{k_2^2-k_1^2}S}<1$.
			We must have $\limsup_{n\to\infty}{\rm Re}(\epsilon_{1,n}) = 1$,
			since otherwise
			\[
			\liminf_{n\to\infty} |A_n| =
			\liminf_{n\to\infty}|(1+\epsilon_{2,n})(1-\epsilon_{1,n})\mu_2|\geq
			(1-e^{-\sqrt{k_2^2-k_1^2}S})(1-\limsup_{n\to\infty}{\rm
				Re}( \epsilon_{1,n} ))\sqrt{k_2^2-k_1^2}>0,
			\]
			which is impossible by (\ref{eq:A:lim}). Thus, there exists a subsequence
      $\{n_k\}_{k=1}^{\infty}$ with $n_k\to\infty$ as $k\to\infty$ such that
      $\lim_{k\to\infty}\epsilon_{1,n_k}=1$, indicating that
      $\lim_{k\to\infty}\Log(\epsilon_{1,n_k}) = 0$ where the logarithm is
      defined to be with its imaginary part in $(-\pi,\pi]$. Let $\mu_{1,n_k} =
      p_{1,n_k} + \bi q_{1,n_k}$, with $p_{1,n_k}q_{1,n_k}=0$ for all
      $k\in\mathbb{Z}^+$ since $\xi_{n_k}\in\mathbb{R}$. Thus, we have to
      consider two cases: Firstly, $p_{1,n_k}\geq 0$ but $q_{1,n_k}=0$ for all
      $k\in\mathbb{Z}^+$ (or there is a subsequence with such a property). Since
			\[
        \Log(\epsilon_{1,n_k}) = -2p_{1,n_k}\bar{\sigma}_{2,n_k}+ \Log(e^{2\bi
          p_{1,n_k}M_{2,n_k}}),
			\]
			we have
			\begin{align*}
			\lim_{k\to\infty}{\rm Re}\left( \frac{1-\epsilon_{1,n_k}}{\mu_{1,n_k}} \right) &= \lim_{k\to\infty}{\rm Re}\left( \frac{ -\Log(\epsilon_{1,n_k}) }{\mu_{1,n_k}} \right)=2S>0.
			\end{align*}
			Therefore,
			\begin{align*}
			&\liminf_{k\to\infty} {\rm Re}\left(\frac{A_{n_k}(\mu_{1,n_k}+\mu_{2,n_k})}{\mu_{1,n_k}\mu_{2,n_k}(1+\epsilon_{2,n_k})}\right)
          \geq \liminf_{k\to\infty}\frac{2(1-|\epsilon_{2,n_k}|^2)}{|1+\epsilon_{2,n_k}|^2} + 2\sqrt{k_2^2-k_1^2}S>0,
			\end{align*}
			which is in contradiction with (\ref{eq:A:lim}).
			Secondly, $q_{1,n_k}\geq 0$ but $p_{1,n_k}=0$ for all $k\in\mathbb{Z}^+$
			(or there is a subsequence with such a property).
			Then, we have
			\[
			\Log(\epsilon_{1,n_k}) = -2q_{1,n_k}M_{2,n_k}  + \Log(e^{-2\bi q_{1,n_k}\bar{\sigma}_{2,n_k}}),
			\]
			so that
			\begin{align*}
			\lim_{k\to\infty}{\rm Im}\left( \frac{1-\epsilon_{1,n_k}}{\mu_{1,n_k}} \right) &= \lim_{k\to\infty}{\rm Im}\left( \frac{ -\Log(\epsilon_{1,n_k}) }{\mu_{1,n_k}} \right)=-2M<0.
			\end{align*}
			If $M=\infty$, then
			\begin{align*}
			&\limsup_{k\to\infty} {\rm Im}\left(\frac{A_{n_k}(\mu_{1,n_k}+\mu_{2,n_k})}{\mu_{1,n_k}\mu_{2,n_k}(1+\epsilon_{2,n_k})}\right) 
			\leq \limsup_{k\to\infty}\frac{-4{\rm Im}( \epsilon_{2,n_k} )}{|1+\epsilon_{2,n_k}|^2} -2\sqrt{k_2^2-k_1^2}M=-\infty,
			\end{align*}
			which contradicts (\ref{eq:A:lim}).
			Otherwise we must have $S=\infty$ so that
      $\lim_{k\to\infty}|\epsilon_{2,n_k}|=0$ and 
			\begin{align*}
			&\limsup_{k\to\infty} {\rm Im}\left(\frac{A_{n_k}(\mu_{1,n_k}+\mu_{2,n_k})}{\mu_{1,n_k}\mu_{2,n_k}(1+\epsilon_{2,n_k})}\right) \leq  -2\sqrt{k_2^2-k_1^2}M<0, 
			\end{align*}
			which again contradicts  (\ref{eq:A:lim}).
			The case when $\xi_0=\pm k_2$ can be analyzed similarly.
		
		To sum up, we claim that (\ref{eq:A:lim}) cannot occur. Therefore 
		\[
		\left| \frac{\mu_1\mu_2}{\mu_1+\mu_2} \right|\lesssim |A|,
		\]
		for all $\xi\in\ol{\mathbb{C}^{+-}\cup\mathbb{C}^{-+}}\backslash
		\cup_{j=1}^2\{\xi\in\mathbb{C}\backslash\mathbb{R}: |\mu_j|<\varepsilon_0k_1\}$,
		$\bar{\sigma}_2\gtrsim k_1^{-1}$, and $M_2\gtrsim k_1^{-1}$.
		
		For the other two inequalities, we follow the same procedure by
		considering two cases. For brevity, we prove the estimate for
		$|(\epsilon_1 - 1)\mu_2|$ in case (II) only.
		Consider $\xi_0=\pm k_1$ in the following since the case $\xi_0=\pm k_2$ is similar. As before, we still distinguish two cases for the
		subsequence $\{n_k\}_{k=1}^{\infty}$ with
		$\lim_{k\to\infty}\epsilon_{1,n_k}=1$: Firstly, $p_{1,n_k}\geq 0$ but
		$q_{1,n_k}=0$. We have
		\begin{align*}
		\liminf_{k\to\infty}{\rm Re}\left( \frac{\mu_{1,n_k}}{1-\epsilon_{1,n_k}} \right)
		&= \liminf_{k\to\infty}{\rm Re}\left( \frac{\mu_{1,n_k}} { -\Log(\epsilon_{1,n_k}) }\right)\\
		&=\liminf_{k\to\infty}{\rm Re}\left( \frac{p_{1,n_k}}{ 2p_{1,n_k}\bar{\sigma}_{2,n_k} - \Log(e^{2\bi p_{1,n_k}M_{2,n_k}}) } \right) \geq 0,
		\end{align*}
		so that 
		\begin{align*}
      &\liminf_{k\to\infty} {\rm Re}\left(\frac{A_{n_k}}{(1-\epsilon_{1,n_k})\mu_{2,n_k}(1-\epsilon_{2,n_k})}\right)
        \geq \liminf_{k\to\infty}\frac{1-|\epsilon_{2,n_k}|^2}{|1+\epsilon_{2,n_k}|^2} >0,
		\end{align*}
		which is in contradiction with (\ref{eq:A:lim}).
		Secondly, consider $p_{1,n_k}=0$ but $q_{1,n_k}\geq 0$. Suppose first
		$M=\infty$ or 
		\[
		\limsup_{k\to\infty}\left|\frac{\Log(e^{-2\bi q_{1,n_k}\bar{\sigma}_{2,n_k}})}{q_{1,n_k}}\right| = \infty,
		\]
		then since
		\begin{align*}
		&\liminf_{k\to\infty}\left| \frac{\mu_{1,n_k}}{1-\epsilon_{1,n_k}} \right| =\liminf_{k\to\infty}\left| \frac{q_{1,n_k}\bi}{2q_{1,n_k}M_{2,n_k} - \Log(e^{-2\bi q_{1,n_k}\bar{\sigma}_{2,n_k}})} \right|= 0, 
		\end{align*}
		so that we have
		\begin{align*}
		&\limsup_{k\to\infty} {\rm Re}\left(\frac{A_{n_k}}{(1-\epsilon_{1,n_k})\mu_{2,n_k}(1-\epsilon_{2,n_k})}\right) \geq \limsup_{k\to\infty}\frac{1-|\epsilon_{2,n_k}|^2}{|1+\epsilon_{2,n_k}|^2} >0,
		\end{align*}
		which contradicts (\ref{eq:A:lim}).
		Otherwise we have $S=\infty$, $M<\infty$, and
		\[
		\limsup_{k\to\infty}\left|\frac{\Log(e^{-2\bi q_{1,n_k}\bar{\sigma}_{2,n_k}})}{q_{1,n_k}}\right| < \infty,
		\]
		so that $\lim_{k\to\infty}\epsilon_{2,n_k}=0$. One easily gets
		\begin{align*}
      \liminf_{k\to\infty}{\rm Im}\left( \frac{\mu_{1,n_k}}{1-\epsilon_{1,n_k}} \right)
      &= \liminf_{k\to\infty}{\rm Im}\left( \frac{\mu_{1,n_k}} { -\Log(\epsilon_{1,n_k}) }\right)\\
      &=\liminf_{k\to\infty}\frac{2M_{2,n_k}}{( 2M_{2,n_k} )^2 +  \left|\Log(e^{-2\bi q_{1,n_k}\bar{\sigma}_{2,n_k}})/q_{1,n_k}\right|^2}\\ 
      &\geq \frac{2M}{(2M)^2 + \limsup_{k\to\infty}\left|\Log(e^{-2\bi q_{1,n_k}\bar{\sigma}_{2,n_k}})/q_{1,n_k}\right|^2}>0.
		\end{align*}
		Thus,
		\begin{align*}
		&\liminf_{k\to\infty} {\rm Im}\left(\frac{A_{n_k}}{(1-\epsilon_{1,n_k})\mu_{2,n_k}(1-\epsilon_{2,n_k})}\right) =\liminf_{k\to\infty}{\rm Im}\left(1 + \frac{2\mu_{1,n_k}}{\sqrt{k_2^2-k_1^2}( 1-\epsilon_{1,n_k} )}  \right)>0,
		\end{align*}
		which is also in contradiction with (\ref{eq:A:lim}).
	\end{proof}
\end{mylemma}
The next lemma can be taken as a refined version of lemma \ref{lemA7} in the
case when  $x\in B_{\rm in}$ and  $y\in D$.
\begin{mylemma}\label{lemA42}
  For $x\in B_{\rm in}$ and $y\in D$, $\xi\in
  \ol{\mathbb{C}^{+-}\cup\mathbb{C}^{-+}}\backslash\cup_{l=1}^2\{\xi\in\mathbb{C}\backslash\mathbb{R}:|\mu_l|\leq
  \varepsilon_0k_1\}$ with $\varepsilon_0\eqsim 1$, the following property holds
  for $f_{x_2,y_2}^{i,j,l}(\xi)$, $i,j,l=1,2$,
	\begin{align*}
    |f_{x_2,y_2}^{i,j;l}(\xi)|\lesssim \frac{\left|e^{\bi\mu_l(d_2+\bi\bar{\sigma}_2) }A\right| }{|\mu_i|},\quad
    |\partial_{x_2}f_{x_2,y_2}^{i,j;l}(\xi)|\lesssim \left|e^{\bi\mu_l(d_2+\bi\bar{\sigma}_2) }A\right|.
	\end{align*}
	\begin{proof}
		We prove $j=1$ only; case $j=2$ can be analyzed similarly. According to
    \eqref{eq:def:111}, {Lemma \ref{techlemma2}, and Lemma 
      \ref{thmA41}}, we see that
		\begin{align*}
      |f_{x_2,y_2}^{1,1;1}|\lesssim & \left|(1-\epsilon_2)e^{\bi\mu_1(d_2+\bar{\sigma}_2\bi)}\right|+\left|\frac{\mu_2e^{\bi\mu_1(d_2+\bar{\sigma}_2\bi)}}{\mu_1+\mu_2}\right| + \frac{|\mu_2|}{|\mu_1|}\left| e^{\bi\mu_1(\tilde{M}_2 + x_2+y_2)} - e^{\bi \mu_1(\tilde{M}_2 - y_2 + x_2)}\right|\\
                                    &+ \frac{|\mu_2|}{|\mu_1|}\left|e^{\bi \mu_1(3\tilde{M}_2 - y_2- x_2 )}   - e^{\bi \mu_1(3\tilde{M}_2 + y_2-x_2)}\right| + \frac{|\mu_2|}{|\mu_1|}\left|e^{\bi \mu_1(3\tilde{M}_2 + y_2 - x_2)}   - e^{\bi \mu_1(\tilde{M}_2 + y_2-x_2)}\right|\\
      \lesssim&\left|e^{\bi\mu_1(d_2+\bar{\sigma}_2\bi)}\right|\left(  \frac{|A|}{|\mu_1|} +\frac{|\mu_2|}{|\mu_1|}|e^{2\bi\mu_1 y_2}-1| + \frac{|\mu_2|}{|\mu_1|}|\epsilon_1-1|\right) \lesssim \frac{\left|e^{\bi\mu_1(d_2+\bar{\sigma}_2\bi)}A\right|}{|\mu_1|}.
		\end{align*}
		Similarly, one obtains that $|\partial_{x_2}
    f_{x_2,y_2}^{1,1;1}|\lesssim\left|e^{\bi\mu_1(d_2+\bar{\sigma}_2\bi)}A\right|$.
    The estimates for $f_{x_2,y_2}^{1,1;2}$ can be obtained readily by
    {Lemma \ref{thmA41}}. According to (\ref{eq:def:211}) and
    {Lemma \ref{thmA41}}, we see that
		\begin{align*}
      |f_{x_2,y_2}^{2,1;1}| \lesssim& \frac{\left|\mu_1e^{\bi\mu_1(d_2+\bi\bar{\sigma}_2)}\right|}{|\mu_1+\mu_2|} + |e^{\bi\mu_1\tilde{M}_2 + y_2} - e^{\bi \mu_1\tilde{M}_2 - y_2}|\\
      \lesssim&\frac{\left|Ae^{\bi\mu_1(d_2+\bi\bar{\sigma}_2)}\right|}{|\mu_2|} + \left| e^{\bi\mu_1(d_2+\bi\bar{\sigma}_2)} \right||e^{2\bi\mu_1y_2}-1| \lesssim\frac{\left|e^{\bi\mu_1(d_2+\bi\bar{\sigma}_2)}A\right|}{|\mu_2|}.
		\end{align*}
		Similarly, (\ref{eq:def:212}) leads to
		\begin{align*}
      |f_{x_2,y_2}^{2,1;2}| \lesssim&|\epsilon_1-1|\left|e^{\bi\mu_2(d_2+\bi\bar{\sigma}_2)}\right| + \frac{|\mu_1|}{|\mu_1+\mu_2|} \left|e^{\bi\mu_2(d_2+\bi\bar{\sigma}_2)}\right| + \left|e^{\bi\mu_2(d_2+\bi\bar{\sigma}_2)}\right|\\
                                      &\left(  |e^{\bi\mu_1(2\tilde{M}_2-y_2)} - e^{\bi \mu_1 (2\tilde{M_2}+y_2)}| + |\epsilon_1-1||e^{\bi\mu_1y_2}|\right) \lesssim\frac{\left|e^{\bi\mu_2(d_2+\bi\bar{\sigma}_2)}A\right|}{|\mu_2|}.
		\end{align*}
		The estimates for $\partial_{x_2} f_{x_2,y_2}^{2,1;l}, l=1,2$, can be
		similarly analyzed.
	\end{proof}
\end{mylemma}

Let 
\[
B_{\delta_0}= \{\xi=\xi_1+\bi\xi_2:
0\leq\xi_{1}\leq\sqrt{2}k_1/2,0\leq\xi_2\leq\frac{\delta_0}{M_2}\lesssim
\delta_0 k_1\}.
\]
where $\delta_0\leq \frac{\sqrt{2}k_1}{4}\bar{\sigma}_2$.
\begin{mylemma}\label{LemImmu}
	For any $\xi\in B_{\delta_0}$ and $x_2\in[0,L_2/2]$, we have that
	\begin{equation}
	{\rm Im}(\mu_j(\tilde{M}_2-x_2))\geq \max\left\{\frac{\sqrt{2}}{2}k_1\bar{\sigma}_2 - \delta_0,\delta_0\right\},\quad {\rm Im}(\mu_jx_2)\geq -\delta_0.
	\end{equation}
	for $j=1,2$.
	\begin{proof}
		Let $\xi = \xi_1 + \bi \xi_2\in B_{\delta_0}$ and $\mu_j = p_j -\bi q_j$
		with $ \xi_1,\xi_2,p_j,q_j\geq 0$. Then, one gets
		\[
		(p_{j}^2 + \xi_{1}^2)(1-\frac{\xi_{2}^2}{p_{j}^2}) = (p_j^2 - \xi_2^2)(1 +
		\frac{\xi_{1}^2}{p_{j}^2})= k_j^2,\quad p_jq_j = \xi_1\xi_2.
		\]
		Thus, $q_j = \frac{\xi_1\xi_2}{p_j}\leq \frac{\delta_0\xi_1}{p_jM_2}$ and $p_j\geq \sqrt{k_j^2-\xi_1^2}\geq \xi_1$,
		so that
		\[
		{\rm Im}(\mu_j(\tilde{M}_2-x_2)) = p_j\bar{\sigma}_2 - q_j(M_2-x_2)\geq
		p_j\bar{\sigma}_2 - \frac{\delta_0\xi_1}{p_j}\geq \sqrt{k_j^2-\xi_1^2}\bar{\sigma}_2 - \frac{\delta_0\xi_1}{\sqrt{k_j^2-\xi_1^2}}.
		\]
		Since $\xi_1\leq \frac{\sqrt{2}}{2}k_1$ and since $k_j\geq k_1$, we get
		${\rm Im}(\mu_j(\tilde{M}_2-x_2))\geq \frac{\sqrt{2}}{2}k_1\bar{\sigma}_2 - \delta_0 \geq \delta_0$. One similarly obtains 
		${\rm Im}(\mu_j x_2) = -q_jx_2\geq -\frac{\delta_0\xi_1x_2}{\sqrt{k_j^2-\xi_1^2}M_2}\geq -\delta_0$.
		
	\end{proof}
\end{mylemma}
The following lemma is concerned with the uniformly lower bound of $A$ in $B_{\delta_0}$.
\begin{mylemma}\label{lemABdelta0}
	There exists a constant $\delta_0\eqsim 1$, such that $A(\xi;\bar{\sigma}_{2},M_2)\neq 0$ in
	$B_{\delta_0}$ for all $\bar{\sigma}_2\gtrsim k_1^{-1}$ and 
	$M_2\gtrsim k_1^{-1}$. Furthermore, 
	\[
	|\mu_1+\mu_2|\lesssim|A(\xi;\bar{\sigma}_2,M_2)|,
	\] 
	for all $\xi\in B_{\delta_0}$, $\bar{\sigma}_2\gtrsim k_1^{-1}$ and $M_2\gtrsim k_1^{-1}$.
	
	\begin{proof}
		As in the proof of Lemma~\ref{thmA41}, we do the rescaling
		by assuming $k_1=1$ and $k_2=\kappa$ in the following.
		We first prove there exists a non-empty region $B_{\delta_0}$ such that $A\neq 0$ for $\xi\in B_{\delta_0}$.
		Suppose otherwise there exist a sequence of $\{\xi_n,
		\delta_n,\bar{\sigma}_{2,n},M_{2,n}\}$ with $\lim_{n\to\infty}\xi_n=\xi_0\in
		B_{\delta_0}$, $\lim_{n\to\infty}\delta_n=0$,
		$\lim_{n\to\infty}\bar{\sigma}_{2,n}=S\gtrsim 1$ and
		$\lim_{n\to\infty}M_{2,n}=M\gtrsim 1$, such that $A(\xi_n;\bar{\sigma}_{2,n},M_{2,n})=0$,
		for all $n\in\mathbb{N}$. As a result, $ \lim_{n\to\infty}{\rm Im}(\xi_n) = 0$, which implies
		$\xi_0\in[0,k_1/\sqrt{2}]$. 

    We cliam that $\max(S,M)=\infty$, since otherwise, $A(\xi_0;S,M) = 0$, which
    is in contradiction with Lemma \ref{lemA1}. Let $\xi_n =
    \xi_{1,n}+\bi\xi_{2,n}$, $\mu_{j,n}=\sqrt{k_j^2-\xi_{n}^2}=p_{j,n}-\bi
    q_{j,n}$ with $p_{j,n},q_{j,n}\geq 0$ for $j=1,2$. Then, by Lemma
    \ref{LemImmu},
		\begin{align*}
		\liminf_{n\to\infty}{\rm Im}(\mu_{j,n}\tilde{M}_2)\geq \liminf_{n\to\infty}\frac{\sqrt{2}}{2}k_1\bar{\sigma}_{2,n}-\delta_{n}\geq \frac{\sqrt{2}}{2}k_1S,
		\end{align*}
		so that if $S=\infty$,  $\epsilon_{j,n}\to 0$ and
		$\lim_{n\to\infty} A(\xi_n;\bar{\sigma}_{2,n},M_{2,n}) = \sqrt{k_1^2-\xi_0^2} + \sqrt{k_2^2-\xi_0^2} > 0$,
		which is a contradiction.
		Now let us assume $M=\infty$ and $S<\infty$. Since $\lim_{n\to\infty}q_{j,n}=0$, we get
		\begin{align*}
      &\liminf_{n\to\infty}\left|  \frac{A(\xi_n;\bar{\sigma}_{2,n},M_{2,n})}{(1-\epsilon_{1,n})(1-\epsilon_{2,n})}\right|\\
      &\geq\liminf_{n\to\infty}\sum_{i=1}^2 \frac{(1-|\epsilon_{i,n}|^2)}{1+|\epsilon_{i,n}|^2}p_{i,n} - \limsup_{n\to\infty}\sum_{i=1}^2 \frac{2|{\rm Im}(\epsilon_{i,n})|}{1+|\epsilon_{i,n}|^2}q_{i,n} =\sum_{i=1}^2 \frac{1-e^{-\sqrt{2}k_1S}}{1+e^{-\sqrt{2}k_1S}}\sqrt{k_i^2-\xi_0^2}>0,
		\end{align*}
		which is also a contradiction. Consequently, there must exist a constant
		$\delta_0\eqsim 1$ with the desired property.
		
		Now we prove that $|A|$ has a uniformly nonzero lower bound in
    $B_{\delta_0}$. Otherwise, we can find a sequence $\{\xi_n, \bar{\sigma}_{2,n},
    M_{2,n}\}_{n=1}^{\infty}$ with $\xi_n\to\xi_0\in B_{\delta_0}$,
    $\bar{\sigma}_{2,n}\to S\gtrsim 1$, and $M_{2,n}\to M\gtrsim 1$ such that
		\[
      \lim_{n\to\infty}\frac{|A(\xi_n;\bar{\sigma}_{2,n},M_{2,n})|}{|\mu_{1,n}+\mu_{2,n}|}
      = 0.
		\]
		Suppose first $\max(S,M)<\infty$, then we must have $A(\xi_0;S,M) = 0$,
		which is impossible according to the choice of $\delta_0$.
		Now consider the case $\max(S,M)=\infty$. It holds
		\[
		{\rm Im}(\mu_{j,n}\tilde{M}_{2,n}) \geq
		\max\left\{\frac{\sqrt{2}}{2}k_1\bar{\sigma}_{2,n} - \delta_0,\delta_0  \right\}.
		\]
		If $S=\infty$, then $\epsilon_{j,n}\to 0$ so that
		$\lim_{n\to\infty} A(\xi_n;\bar{\sigma}_{2,n},M_{2,n}) = \sqrt{k_1^2-\xi_0^2} + \sqrt{k_2^2-\xi_0^2} > 0$,
		which is a contradiction.
		If $M=\infty$ and $S<\infty$,  then
		$\lim_{n\to\infty}q_{j,n}=0$ by the definition of $B_{\delta_0}$, which implies
		\begin{align*}
		&\liminf_{n\to\infty}\left|  \frac{A(\xi_n;\bar{\sigma}_{2,n},M_{2,n})}{(1-\epsilon_{1,n})(1-\epsilon_{2,n})}\right| \geq\liminf_{n\to\infty}\sum_{i=1}^2 \frac{(1-|\epsilon_{i,n}|^2)}{1+|\epsilon_{i,n}|^2}p_{i,n} - \limsup_{n\to\infty}\sum_{i=1}^2 \frac{2|{\rm Im}(\epsilon_{i,n})|}{1+|\epsilon_{i,n}|^2}q_{i,n}\\
		&=\sum_{i=1}^2 \frac{(1-e^{-2\delta_0})}{1+e^{-2\delta_0}}\sqrt{k_i^2-\xi_0^2}>0,
		\end{align*}
		a contradiction.
	\end{proof}
\end{mylemma}

\subsection{Property of $G_{\rm PML}$}
To study the difference between $G_{\rm PML}$ and $G_{\rm layer}$ as the absorption
of PML increases, we need to estimate term-by-term for the series in
\eqref{eq:infseries:G:1} and \eqref{eq:infseries:G:2}. In particular, we show
that all the contributions from terms in the series other than the zeroth order
term are exponentially small. Since $\bar{\sigma}_j$ and $d_j$ for $j=1,2$ vary,
we here introduce various integral paths and notations to estimate the
terms in (\ref{eq:infseries:G:1}) and (\ref{eq:infseries:G:2}).
\begin{figure}[!ht]
  \centering
  (a)\includegraphics[width=0.3\textwidth]{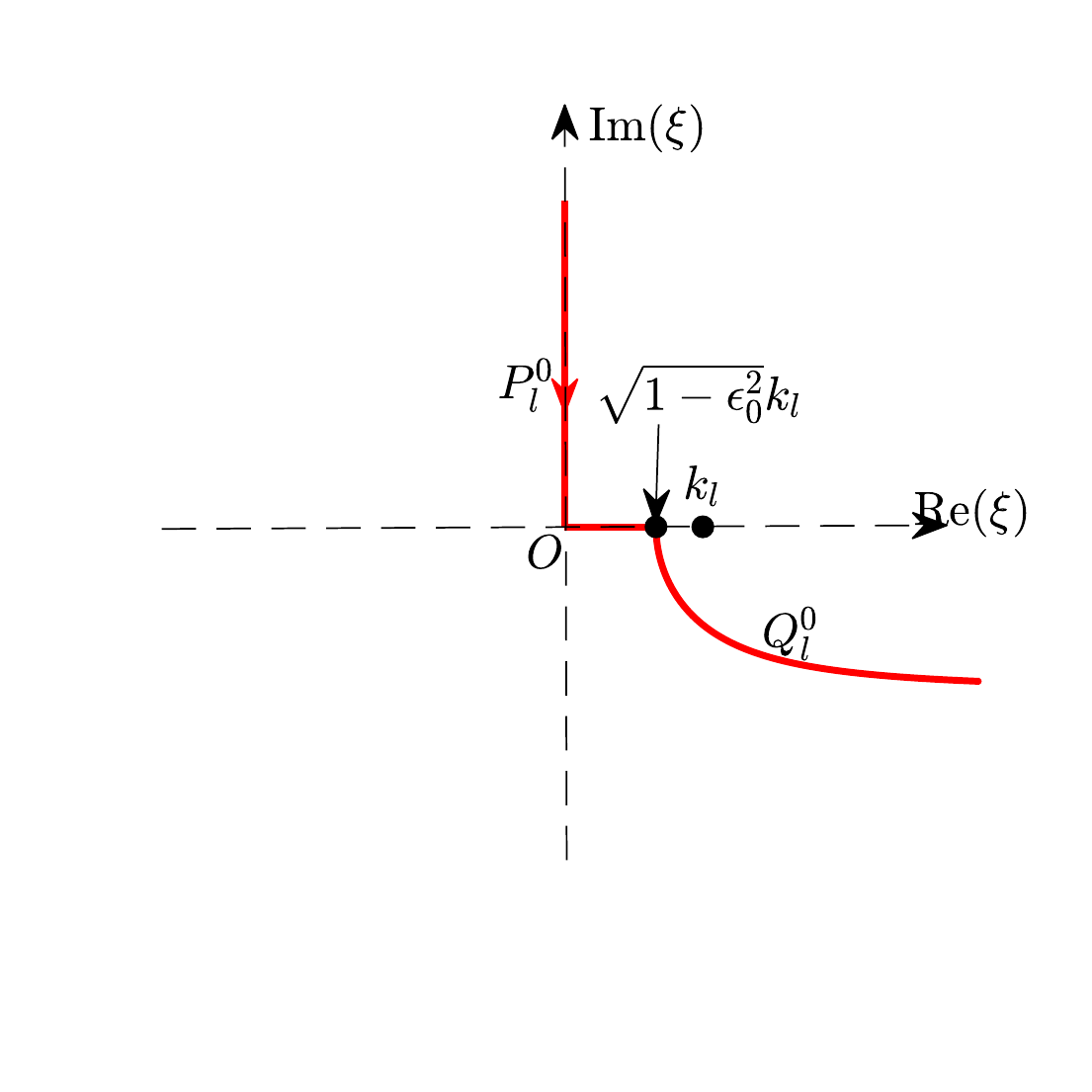}
  (b)\includegraphics[width=0.3\textwidth]{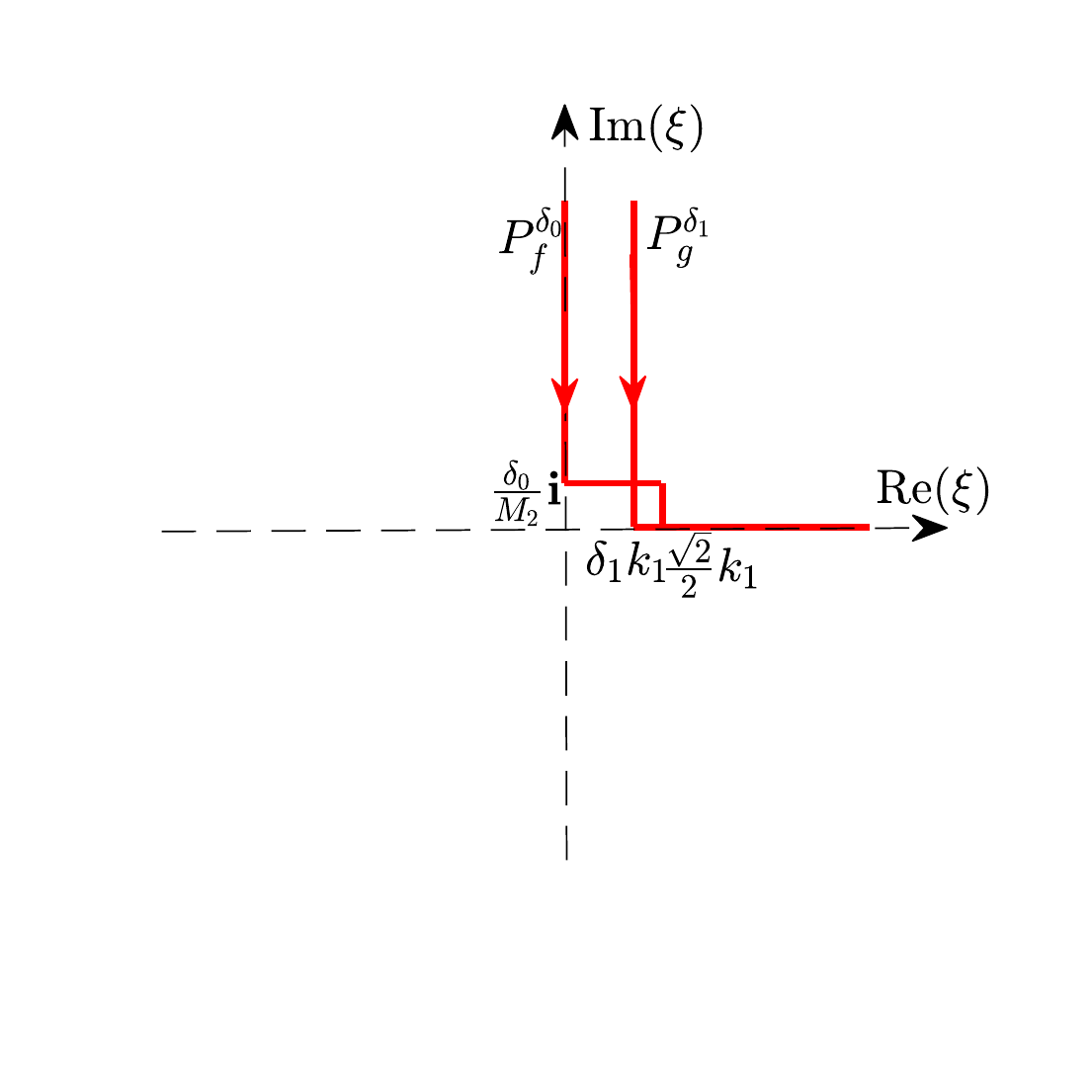}
  \vspace{-0.4cm}
  \caption{(a): path $P_l^0$. (b): paths $P_f^{\delta0}$ and
    $P_g^{\delta_1}$.}
  \label{fig:integralpaths}
\end{figure}

\begin{itemize}
	\item[{\bf (a)}:] Let $P_l^0$ be the following path
	\begin{equation}
	\label{eq:path0}
	P^0_l=\{\xi: +\infty\bi\to 0 \to \sqrt{1-\varepsilon_0^2}k_l\}\cup Q_l^0,
	\end{equation}
	where 
	\[
	Q_l^0: \{\xi=\sqrt{k_l^2-\mu_l^2}: \mu_l\in \varepsilon_0k_l\to \varepsilon_0k_l+\infty\bi\},
	\]
	for $l=1,2$, and $\varepsilon_0\eqsim 1$ satisfies the following inequalities, 
	\begin{equation}
	\label{eq:cond:eps0}
	\frac{L_2}{2} - \frac{\varepsilon_0L_1}{2\sqrt{1-\varepsilon_0^2}}\geq
	\frac{L_2}{4},\quad \varepsilon_0\leq \min\left\{\sqrt{\frac{k_2^2-k_1^2}{k_2^2+k_1^2}},\frac{2\sqrt{k_2^2-k_1^2}}{k_1}\right\}.
	\end{equation}
Figure~\ref{fig:integralpaths}(a) shows the path $P_l^0$.
	\item[{\bf (b)}:] Let $P_f^{\delta_0}$ be the following path,
	\begin{align}
	P_f^{\delta_0}&=\{\xi: +\infty\bi \to \frac{\delta_0}{M_2}\bi \to \frac{\delta_0}{M_2}\bi+\frac{\sqrt{2}}{2}k_1 \to \frac{\sqrt{2}}{2}k_1 \to +\infty\},
	\end{align}
	where $\delta_0\eqsim 1$ was introduced in {Lemma~\ref{lemABdelta0}}.  
	Let $P_g^{\delta_1}$ be the following path,
	\[
	P_{g}^{\delta_1}: \delta_1k_1 + \infty\bi \to \delta_1k_1\to \infty,
	\]
	where $\delta_1\eqsim 1$ satisfies
	\begin{equation}
	\label{eq:cond:del1}
	\frac{L_1}{2}-R - \frac{\delta_1}{\sqrt{1-\delta_1^2}}\left(\frac{L_2}{2}+R\right)\geq 0.
	\end{equation}
Figure~\ref{fig:integralpaths}(b) shows paths $P_f^{\delta_0}$ and $P_g^{\delta_1}$.
	\item[{\bf (c)}:] For $l=1,2$, let $P_l^{\delta_2}$ be the following path
	\[
	P_{l}^{\delta_2}: \delta_2k_l + \infty\bi \to \delta_2k_l\to \infty,
	\]
	where $\delta_2\eqsim 1$ satisfies the following two inequalities,
	\begin{align}
	\label{eq:cond:del2:a}
	\frac{L_1}{2}-R - \frac{\delta_2L_2}{\sqrt{1-\delta_2^2}}\geq 0,\quad d_1 - \frac{2\delta_2d_2}{\sqrt{1-\delta_2^2}}\geq 0.
	\end{align}
	They in fact imply
	\begin{align}
	\label{eq:cond:del2}
	2M_1 -\frac{L_1}{2}-R - 2\frac{\delta_2M_2}{\sqrt{1-\delta_2^2}}&\geq d_1.
	\end{align}
  Path $P_l^{\delta_2}$ is similar to $P_g^{\delta_1}$.
	\item[{\bf (d)}:] For $l=1,2$, let $P_l^1$ be the following path
	\begin{equation}
	\label{eq:path1}
	P^1_l=\{\xi: +\infty\bi\to 0 \to \sqrt{1-\varepsilon_1^2}k_l\}\cup Q_l^1,
	\end{equation}
	where 
	\[
	Q_l^1: \{\xi=\sqrt{k_l^2-\mu_l^2}: \mu_l\in \varepsilon_1k_l\to \varepsilon_1k_l+\infty\bi\},
	\]
	and $\varepsilon_1\eqsim 1$ satisfies the following inequalities
	\begin{equation}
	\label{eq:cond:eps1}
	\frac{L_2}{2}-R - (\frac{L_1}{2}+R)\frac{\varepsilon_1}{\sqrt{1-\varepsilon_1^2}}\geq 0.
	\end{equation}
  Path $P_l^1$ is similar to $P_l^0$.
\end{itemize}
The following lemmas are concerned with the exponential convergence for terms in $G_{\rm res}(x,y)$.
\begin{mylemma}\label{lemA45}
	We have that 
	\begin{align}
	\label{eq:int:0fij:est}
    \left|\int_{\rm EXT}\frac{e^{\bi\xi a_0^{x_1,y_1}}f_{x_2,y_2}^{i,j}(\xi)}{A}d\xi\right|\lesssim & e^{-\min(\sqrt{2},2\varepsilon_0)k_1\bar{\sigma}_2},\\
	\label{eq:int:grad0fij:est}
    \left|\nabla_x\int_{\rm EXT}\frac{e^{\bi\xi a_0^{x_1,y_1}}f_{x_2,y_2}^{i,j}(\xi)}{A}d\xi\right| \lesssim& k_2e^{-\min(\sqrt{2},2\varepsilon_0)k_1\bar{\sigma}_2}.
	\end{align}
	\begin{proof}
		For $\xi=\xi_1 - \bi \xi_2\in
		Q_l^0$ with $\xi_1,\xi_2\geq 0$, we obtain
		$(\xi_1^2 + p_l^2)\left( 1-\frac{q_l^2}{\xi_1^2} \right) = k_l^2,
		\xi_1\xi_2 = p_lq_l$,
		so that
		$\xi_1 \geq \sqrt{k_l^2 - p_l^2} = \sqrt{1-\varepsilon_0^2}k_l,\quad \xi_2 =
		\frac{p_lq_l}{\xi_1}\leq \frac{\varepsilon_0q_l}{\sqrt{1-\varepsilon_0^2}}$.
		We directly verify that $\min(|\mu_1|,|\mu_2|)\geq \varepsilon_0 k_1$,
		so that Lemmas \ref{thmA41} and \ref{lemA42} are applicable.
		
		Thus,
		\begin{align*}
		\left|\frac{e^{\bi\xi a_0^{x_1,y_1}}f_{x_2,y_2}^{i,j,l}(\xi)e^{\bi\mu_l\tilde{M}_2}}{A}\right|
		&\lesssim \frac{ 1 }{|\sqrt{k_i^2-\xi^2}|} e^{-q_lM_2-q_ld_2-2\varepsilon_0k_l\bar{\sigma}_2+ \xi_2L_1/2}\\
		&\lesssim \frac{ 1 }{|\sqrt{k_i^2-\xi^2}|} e^{-2q_ld_2-2\varepsilon_0k_l\bar{\sigma}_2 - q_l(L_2-\varepsilon_0L_1/\sqrt{1-\varepsilon_0^2})/2}\\
		&\lesssim \frac{ 1 }{|\sqrt{k_i^2-\xi^2}|} e^{-2q_ld_2-2\varepsilon_0k_l\bar{\sigma}_2 - q_lL_2/4},
		\end{align*}
		where we have used the condition of the choice of $\varepsilon_0$.
		By Cauchy's theorem and by Lemma~\ref{lemA42}, it is easy to derive that 
		\begin{align*}
		&\left|  \int_{\rm EXT} \frac{e^{\bi\xi a_0^{x_1,y_1}}f_{x_2,y_2}^{i,j}(\xi)}{A}d\xi\right|
		=  \Bigg|\sum_{l=1}^2\int_{P_l^0}\frac{e^{\bi\xi a_0^{x_1,y_1}}f_{x_2,y_2}^{i,j;l}(\xi)}{A}e^{\bi\mu_l\tilde{M}_2}d\xi\Bigg| \\
		\lesssim & \sum_{l=1}^2 \frac{e^{-\sqrt{2}k_l\bar{\sigma}_2}}{k_i\bar{\sigma}_2}  + \sqrt{\frac{k_2}{k_1}}e^{-2\varepsilon_0 k_l\bar{\sigma}_2}+\frac{e^{-2\varepsilon_0k_l\bar{\sigma}_2}}{\varepsilon_0k_1\sqrt{1-\varepsilon_0^2}k_l(2d_2+L_2/4)}\left(\varepsilon_0k_l + \frac{1}{2d_2+L_2/4}  \right).
		\end{align*}
		One similarly gets the estimate for the gradient of the integral; we omit the details here.
	\end{proof}
\end{mylemma}
\begin{mylemma}
	Let $n$ be a nonzero integer. Then,
	\begin{align}
	\label{eq:fij:est}
	\left|e^{i\xi a_{n}^{x_1,y_1}}\frac{f_{x_2,y_2}^{i,j}(\xi)}{A}\right|\lesssim &  \frac{k_2}{k_1|\mu_1+\mu_2|}e^{-( 2|n|M_1-L_1/2-R )\xi_2 - 2|n|\bar{\sigma}_1\xi_1 -\frac{\sqrt{2}}{2}k_1\bar{\sigma}_2+\delta_0}, i,j=1,2
	\end{align}
	for all $\xi\in B_{\delta_0}$. It also holds
	\begin{equation}
	\label{eq:gij:est}
	\left| e^{\bi \xi a_n^{x_1,y_1}}g_{x_2,y_2}^{i,j}(\xi) \right|\leq \frac{2}{|\mu_1+\mu_2|} e^{-( (2|n|-2)M_1 + 2d_2)\xi_2 - 2|n|\bar{\sigma}_1\xi_1}, i,j=1,2,
	\end{equation}
  for all $\xi\in L_{\delta_1}$, where $L_{\delta_1}=\{\xi=\xi_1+\bi\xi_2: \xi_1=\delta_1k_1,
	\xi_2\geq 0\}$ with $\delta_1$ defined in (\ref{eq:cond:del1}).

	\begin{proof}
		We here prove case $i=j=1$ only and the others can be estimated similarly. We
		prove (\ref{eq:fij:est}) first. Since for any $\xi\in B_{\delta_0}$, Lemma \ref{LemImmu} 
		indicates that $|e^{\bi \mu_l |x_2|}|<e^{\delta_0}, |e^{\bi \mu_l \tilde{M}_2}|\leq e^{-\frac{\sqrt{2}}{2}k_1\bar{\sigma}_2+\delta_0}$.
    Furthermore, one easily gets that
		\[
		\left| \frac{\mu_2}{\mu_1} \right| \leq \frac{\sqrt{k_2^2+\xi_1^2 +
				\xi_2^2}}{\sqrt{k_1^2-\xi_1^2}}\lesssim \frac{k_2}{k_1}.
		\]
		Thus, one verifies that 
		\[
		|f_{x_2,y_2}^{1,1}(\xi)|\leq \left( 8\left|\frac{\mu_2}{\mu_1}\right|
		+20\right)e^{-\sqrt{k_1^2-\xi_1^2}\bar{\sigma}_2+\delta_0}\lesssim
		\frac{k_2}{k_1}e^{-\frac{\sqrt{2}}{2}k_1\bar{\sigma}_2+\delta_0},
		\]
		Since for $n>0$,
		${\rm Re}(a_n^{x_1,y_1})\geq 2|n|M_1 - L_1/2-R,\quad {\rm Im}(a_n^{x_1,y_1}) = 2|n|\bar{\sigma}_1$,
		we have
		\begin{align*}
		\left|e^{i\xi a_{n}^{x_1,y_1}}\frac{f_{x_2,y_2}^{1,1}(\xi)}{A}\right|\lesssim&  \frac{k_2}{k_1|\mu_1+\mu_2|}e^{-( 2|n|M_1-L_1/2-R )\xi_2 - 2|n|\bar{\sigma}_1\xi_1 -\frac{\sqrt{2}}{2}k_1\bar{\sigma}_2+\delta_0}.
		\end{align*}
		
		Now suppose $\xi\in L_{\delta_1}$, then one similarly obtains
		$|g_{x_2,y_2}^{1,1}(\xi)|\leq \frac{2}{|\mu_1+\mu_2|}e^{(L_2/2+R)q_1}$,
		so that 
		\begin{align*}
      \left|e^{i\xi a_{n}^{x_1,y_1}}g_{x_2,y_2}^{1,1}(\xi)\right|
      \leq&\frac{2}{|\mu_1+\mu_2|} e^{-( 2|n|M_1-L_1/2-R )\xi_2 - 2|n|\bar{\sigma}_1\xi_1 +(L_2/2+R)\delta_1\xi_2/\sqrt{1-\delta_1^2}}\\ 
		\leq&\frac{2}{|\mu_1+\mu_2|} e^{-( (2|n|-2)M_1 + 2d_2)\xi_2 - 2|n|\bar{\sigma}_1\xi_1}.
		\end{align*}
	\end{proof}
\end{mylemma}
The next lemma follows naturally from the lemma above.
\begin{mylemma}\label{lemA47}
	For any nonzero integer $n$,
	\begin{align}
	\label{eq:int:nfij:est}
	\left|\int_{\rm EXT}e^{\bi\xi a_n^{x_1,y_1}}\frac{f_{x_2,y_2}^{i,j}(\xi)}{A}d\xi\right|&\lesssim e^{-\sqrt{2}\bar{\sigma}_2k_1/2-\gamma_0|n|},\\
	\label{eq:int:gradnfij:est}
	\left|\nabla_{x}\int_{\rm EXT}e^{\bi\xi a_n^{x_1,y_1}}\frac{f_{x_2,y_2}^{i,j}(\xi)}{A}d\xi\right|&\lesssim k_2e^{-\sqrt{2}\bar{\sigma}_2k_1/2-\gamma_0|n|},
	\end{align}
  where
  $\gamma_0=\min((2d_1+L_1/2-R)\delta_0/M_2,k_1\bar{\sigma}_1/\sqrt{2},k_1\bar{\sigma}_2/\sqrt{2})\eqsim
  1$.
	\begin{proof}
		We prove (\ref{eq:int:nfij:est}) first.     Then, one verifes by Cauchy's theorem that 
		\begin{align*}
		\int_{\rm EXT}e^{\bi\xi a_n^{x_1,y_1}}\frac{f_{x_2,y_2}^{i,j}(\xi)}{A}d\xi&=\int_{P_f^{\delta_0}} e^{\bi\xi a_n^{x_1,y_1}}\frac{f_{x_2,y_2}^{i,j}(\xi)}{A}d\xi\\
		&=\left(  \int_{+\infty\bi}^{\frac{\delta_0}{M_2}\bi} + \int_{\frac{\delta_0}{M_2}\bi}^{\frac{\delta_0}{M_2}\bi+\frac{k_1}{\sqrt{2}}} + \int_{\frac{\delta_0}{M_2}\bi + \frac{k_1}{\sqrt{2}}}^{\frac{k_1}{\sqrt{2}}} + \int_{\frac{k_1}{\sqrt{2}}}^{\infty}\right)e^{\bi\xi a_n^{x_1,y_1}}\frac{f_{x_2,y_2}^{i,j}(\xi)}{A}d\xi\\  
		&=:I_1 + I_2 + I_3 + I_4 ,
		\end{align*}
		where we have 
		\begin{align*}
		|I_1|\lesssim& \sum_{l=1}^2\int_{\frac{\delta_0}{M_2}}^{\infty}\frac{1}{\sqrt{k_i^2+\xi_2^2}}e^{-\xi_2 (2|n|M_1 - L_1/2-R) - 2\sqrt{k_l^2+\xi_2^2}\bar{\sigma}_2}d\xi_2 \lesssim\sum_{l=1}^2 \frac{e^{-2\bar{\sigma}_2 k_l-\delta_0/M_2(2|n|M_1-L_1/2-R)}}{k_i(2|n|M_1-L_1/2-R)},\\
		|I_2|\lesssim& \frac{k_1}{\sqrt{(k_2^2-k_1^2)}}e^{-(2|n|M_1-L_1/2-R)\delta_0/M_2 - k_1\bar{\sigma}_2/\sqrt{2}+\delta_0},\\
		|I_3|\lesssim& \frac{\delta_0}{M_2\sqrt{(k_2^2-k_1^2)}}e^{-\sqrt{2}|n|k_1\bar{\sigma}_1 - k_1\bar{\sigma}_2/\sqrt{2}+\delta_0},\\
		|I_4|\lesssim& \sum_{l=1}^2\int_{\frac{k_1}{\sqrt{2}}}^{\infty} \frac{|e^{\bi\sqrt{k_l^2-\xi^2}(M_2+d_2+2\bi \bar{\sigma}_2)}|}{\sqrt{|k_i^2-\xi^2|}}e^{-2|n|\bar{\sigma}_2\xi}d\xi\lesssim \frac{e^{-\sqrt{2}|n|\bar{\sigma}_2k_1}}{\sqrt{k_1}}\left( \sqrt{k_2} + \frac{1}{(M_2+d_2)\sqrt{k_2}} \right).
		\end{align*}
    Combining the above four estimates, we get (\ref{eq:int:nfij:est})
    immediately. One can similarly estimate the gradients $\nabla I_j$ for
    $j=1,2,3,4$ to get the estimate (\ref{eq:int:gradnfij:est}); we omit the details.
	\end{proof}
\end{mylemma}
Our next two lemmas are concerned with the properties for the non-zeroth order terms in $G_{\rm layer}$.
\begin{mylemma}\label{lemA48}
	For any nonzero integer $n$, it holds
	\begin{align}
	\label{eq:int:ngij:est}
    \left|\int_{\rm EXT}e^{\bi\xi a_n^{x_1,y_1}}g_{x_2,y_2}^{i,j}(\xi)d\xi\right|&\lesssim e^{-2|n|\delta_1 k_1\min(\bar{\sigma}_1,\bar{\sigma}_2)}\\
    \label{eq:int:gradngij:est}
	\left|\nabla_{x}\int_{\rm EXT}e^{\bi\xi a_n^{x_1,y_1}}g_{x_2,y_2}^{i,j}(\xi)d\xi\right|&\lesssim k_2e^{-2|n|\delta_1 k_1\min(\bar{\sigma}_1,\bar{\sigma}_2)}.
	\end{align}
	\begin{proof}
		We obtain by Cauchy's theorem that 
		\begin{align*}
		\int_{\rm EXT}e^{\bi\xi a_n^{x_1,y_1}}g_{x_2,y_2}^{i,j}(\xi)d\xi &=\int_{\rm P_g} e^{\bi\xi a_n^{x_1,y_1}}g_{x_2,y_2}^{i,j}(\xi)d\xi = \int_{\delta_1k_1+\infty\bi}^{\delta_1k_1} + \int_{\delta_1k_1}^{\infty}e^{\bi\xi a_n^{x_1,y_1}}g_{x_2,y_2}^{i,j}(\xi)d\xi=: I_1 + I_2.
		\end{align*}
		It holds
		\begin{align*}
		|I_1|\lesssim&\int_0^{\infty} \frac{e^{-( (2|n|-2)M_1 + 2d_2)\xi_2 - 2|n|\bar{\sigma}_1\delta_1k_1}d\xi_2}{\sqrt{k_2^2-k_1^2}}\lesssim\frac{e^{-2|n|\delta_1 k_1\bar{\sigma}_1}}{\sqrt{k_2^2-k_1^2}((|n|-1)M_1+d_2)} ,\\
		|I_2|\lesssim&\int_{\delta_1 k_1}^{\infty}\frac{e^{-2|n|\bar{\sigma}_2\xi}d\xi}{|\mu_1+\mu_2|} \lesssim\frac{e^{-2|n|\bar{\sigma}_2\delta_1k_1}}{\sqrt{k_2^2-k_1^2}|n|\bar{\sigma}_2}, 
		\end{align*}
    so that the inequality (\ref{eq:int:ngij:est}) follows immediately. One
    similarly gets estimates of $\nabla I_j$ for $j=1,2$ to arrive at
    (\ref{eq:int:gradngij:est}); we omit the details.
	\end{proof}
\end{mylemma}

\begin{mylemma}\label{lemA49}
	Suppose $x\in B_{\rm in}^i$ and $y\in B_{\rm in}^i\cap D$, $i=1,2$. Then for any nonzero integer $n$,
	we have
	\begin{align}
	\label{eq:hankel:est:1}
	\left|H_0^{(1)}(k_l\sqrt{(a_n^{x_1,y_1})^2 + (b_j^{x_2,y_2})^2})\right|&\lesssim e^{-2|n|\delta_2k_1\bar{\sigma}_1},\\
	\label{eq:gradhankel:est:1}
	\left|\nabla_{x}H_0^{(1)}(k_l\sqrt{(a_n^{x_1,y_1})^2 + (b_j^{x_2,y_2})^2})\right|&\lesssim k_2e^{-2|n|\delta_2k_1\bar{\sigma}_1},
	\end{align}
	for $j=1,2,3$ and for $l=1,2$. In particular, we have
	\begin{align}
	\label{eq:hankel:est:2}
	\left|H_0^{(1)}(k_l\sqrt{(a_0^{x_1,y_1})^2 + (b_3^{x_2,y_2})^2})\right|\lesssim & e^{-\min(\sqrt{2},2\varepsilon_1)\bar{\sigma}_2k_i},\\
	\label{eq:gradhankel:est:2}
	\left|\nabla_{x}H_0^{(1)}(k_l\sqrt{(a_0^{x_1,y_1})^2 + (b_3^{x_2,y_2})^2})\right|\lesssim &k_2e^{-\min(\sqrt{2},2\varepsilon_1)\bar{\sigma}_2k_i}.
	\end{align}
	\begin{proof}
		When $x\in B_{\rm in}^i$ and $y\in B_{\rm in}^i\cap D$, we get for $n\in\mathbb{Z}\backslash\{0\}$ that
		\begin{align*}
		&{\rm Im}(a_{n}^{x_1,y_1}) =2|n|\bar{\sigma}_1,\quad
		{\rm Re}(a_{n}^{x_1,y_1}) \in [2|n|M_1-L_1/2-R,2|n|M_1+L_1/2+R],
		\end{align*}
		and  
		\begin{align*}
		&{\rm Im}(a_{0}^{x_1,y_1}) =0,\quad
		{\rm Re}(a_{0}^{x_1,y_1}) \in [0,L_1/2+R], \quad{\rm Re}(b_j^{x_2,y_2}) \in [0, L_2/2+R], \quad {\rm Im}(b_j^{x_2,y_2}) =0,\\
		&{\rm Re}(b_3^{x_2,y_2}) \in [2M_2-L_2/2-R, 2M_2], \quad {\rm Im}(b_3^{x_2,y_2}) =2\bar{\sigma}_2.
		\end{align*}
		With the above estimates, {\cite[Lemma 6.1]{chenzheng} and Lemma
      \ref{lemhankel}} can directly give rise to estimates of Hankel functions
    that are analogous to (\ref{eq:hankel:est:1}) and (\ref{eq:hankel:est:2}).
    We here give an alternative proof which benifits estimating gradients of
    Hankel functions.
		
		For $n\in\mathbb{Z}\backslash\{0\}$ and $j=1,2,3$, by Cauchy's theorem, we
    see that
		\begin{align*}
		H_0^{(1)}(k_l\sqrt{(a_n^{x_1,y_1})^2 + (b_j^{x_2,y_2})^2}) &=
		\frac{\bi}{4\pi}\int_{P_l^{\delta_2}} \frac{1}{\mu_l}e^{\bi
			a_n^{x_1,y_1}\xi + \bi\mu_l b_j^{x_2,y_2}}d\xi\\
		&= \frac{\bi}{4\pi}\left( \int_{\delta_2k_l+\infty\bi}^{\delta_2k_l}+\int_{\delta_2k_l}^{+\infty}  \right)\frac{1}{\mu_l}e^{\bi a_n^{x_1,y_1}\xi +
			\bi b_j^{x_2,y_2}\mu_l } d\xi =: I_1 + I_2.
		\end{align*}
		Thus, we have
		\begin{align*}
		|I_1|&\lesssim \int_{0}^{+\infty}
		\frac{e^{-2|n|\delta_2\bar{\sigma}_1k_l}}{\sqrt{1-\delta_2^2}k_l}
		e^{-[(2|n|M_1-L_1/2-R)  - 2\delta_2M_2/\sqrt{1-\delta_2^2}]\xi_2}d\xi_2\\
		&\leq \int_{0}^{+\infty}
		\frac{e^{-2|n|\delta_2\bar{\sigma}_1k_l}}{\sqrt{1-\delta_2^2}k_l}
		e^{-((2|n|-2)M_1 + d_1)\xi_2}d\xi_2 \leq \frac{e^{-2|n|\delta_2\bar{\sigma}_1k_l}}{\sqrt{1-\delta_2^2}k_l((2|n|-2)M_1+d_1)},
		\end{align*}
		and
		\begin{align*}
		|I_2|&\lesssim \int_{\delta_2k_l}^{k_l} \frac{e^{-2|n|\bar{\sigma}_1\xi}}{\sqrt{k_l^2-\xi^2}}d\xi + \int_{k_l}^{+\infty} \frac{e^{-2|n|\bar{\sigma}_1\xi}}{\sqrt{\xi^2-k_l^2}}d\xi \lesssim \sqrt{1-\delta_2}e^{-2|n|\delta_2\bar{\sigma}_1k_l} + e^{-2|n|\delta_2\bar{\sigma}_1k_l}(1 + \frac{1}{|n|\bar{\sigma}_1k_l}),
		\end{align*}
    so that the estimate (\ref{eq:hankel:est:1}) follows immediately. One
    similarly estimates $\nabla I_j$ for $j=1,2$, yielding the other estimate 
    (\ref{eq:gradhankel:est:1}); we omit the details.
	
		On the other hand, by Cauchy's theorem, we get
		\begin{align*}
		&H_0^{(1)}(k_l\sqrt{(a_0^{x_1,y_1})^2 + (b_3^{x_2,y_2})^2}) =
		\frac{\bi}{4\pi}\int_{P_l^{\varepsilon_1}} \frac{1}{\mu_l} e^{\bi a_0^{x_1,y_1}\xi + \bi b_3^{x_2,y_2}\mu_i}d\xi\\
		=&\frac{\bi}{4\pi}\left( \int_{+\infty\bi }^0 + \int_0^{\sqrt{1-\varepsilon_1^2}k_l} \right) \frac{1}{\mu_l} e^{\bi a_0^{x_1,y_1}\xi + \bi b_3^{x_2,y_2}\mu_l}d\xi +\frac{\bi}{4\pi}\int_{\varepsilon_1k_l}^{\varepsilon_1k_l+\infty\bi}\frac{1}{\xi} e^{\bi a_0^{x_1,y_1}\xi + \bi b_3^{x_2,y_2}\mu_l}d\mu_l,\\
		=:& J_1+J_2+J_3.
		\end{align*}
		We have the following estimates
		\begin{align*}
		|J_1|&\lesssim \frac{1}{k_l}\int_{0}^{+\infty}e^{-2\bar{\sigma}_2\sqrt{k_l^2+\xi_2^2}}d\xi_2 \lesssim \frac{e^{-\sqrt{2}\bar{\sigma}_2 k_l}}{k_l\bar{\sigma}_2},\\
		|J_2|&\lesssim \frac{1}{\varepsilon_1k_l}\int_{0}^{\sqrt{1-\varepsilon_1^2}k_l} e^{-2\bar{\sigma}_2 \sqrt{k_l^2-\xi^2}}d\xi\lesssim \frac{e^{-2 \varepsilon_1\bar{\sigma}_2k_l}}{\varepsilon_1},\\
		|J_3|&\lesssim \frac{1}{\sqrt{1-\varepsilon_1^2}k_l}\int_{0}^{+\infty} e^{-[(2M_2-L_2/2-R) - (L_1/2+R)\varepsilon_1 /\sqrt{1-\varepsilon_1^2}] t- 2\bar{\sigma}_2 \varepsilon_1k_l }dt \\
		&\lesssim \frac{1}{\sqrt{1-\varepsilon_1^2}k_l}\int_0^{+\infty} e^{-2d_2t - 2\bar{\sigma}_2\varepsilon_1k_l} dt \lesssim \frac{1}{\sqrt{1-\varepsilon_1^2}k_ld_2}e^{- 2\bar{\sigma}_2\varepsilon_1k_l},
		\end{align*}
    so that the estimate (\ref{eq:hankel:est:2}) follows immediately. Using the
    same arguments to estimate $\nabla J_j$ for $j=1,2,3$, yields the other 
    estimates (\ref{eq:gradhankel:est:2}); we omit the details.
	\end{proof}
\end{mylemma}
\subsection{Convergence analysis of the source problem}
We now suppose $f\in L^2(D)$. According to Theorem 3.4, we easily get that the
UPML problem (\ref{eq:upml}) possesses the following unique solution:
\begin{equation}\label{Greenrep1}
\tilde{u}(x) = \int_{D} G_{\rm PML}(x,y)f(y)dy,\quad x\in B_{\rm in}.
\end{equation}
On the other hand, as we mentioned in the introduction, the original problem (\ref{eq:org:helm}-\ref{eq:rad:cond})
possesses the following unique solution:
\begin{equation}\label{Greenrep2}
u(x) = \int_{D} G_{k_1,k_2}(x,y)f(y)dy,\quad x\in B_{\rm in}.
\end{equation}
We are now ready to analyze the error between $G_{\rm PML}$ and $G_{\rm
  layer}^{k_1,k_2}$ for any
$x\in B_{\rm in}$ and $y\in D$. Combining the estimates in Lemmas \ref{lemA45},
\ref{lemA47}, \ref{lemA48} and \ref{lemA49}, and making use of
the definitions of $G_{\rm PML}$ in (\ref{eq:infseries:G:1}) and
(\ref{eq:infseries:G:2}), we get the following pointwise convergence result.
\begin{mytheorem}\label{Greenestimate}
	For any $x\in B_{\rm in}$ and $y\in B_{\rm in}\cap D$, we have
	\begin{align*}
	|G_{\rm PML}(x,y) - G_{k_1,k_2}(x,y)|\lesssim
	&e^{-\gamma k_1\min(\bar{\sigma}_1,\bar{\sigma}_2)},\\
	|\nabla_{x}\left(  G_{\rm PML}(x,y) - G_{k_1,k_2}(x,y)\right)|\lesssim
	&k_2e^{-\gamma k_1\min(\bar{\sigma}_1,\bar{\sigma}_2)},
	\end{align*}
	where $\gamma=\min(\sqrt{2}/2,
  2\varepsilon_0,2\varepsilon_1,2\delta_1,2\delta_2)\eqsim 1$.
	\begin{proof}
		By {Lemmas \ref{lemA45}, \ref{lemA47}, \ref{lemA48} and \ref{lemA49}}, we
    get, for any $x\in B_{\rm in}^i$ and any $y\in B_{\rm in}^j\cap D$ for
    $i,j=1,2$,
		\begin{align*}
		|G_{\rm PML}(x,y) - G_{\rm layer}^{i,j}(x,y)|\lesssim& e^{-\min(\sqrt{2},2\varepsilon_0)k_1\bar{\sigma}_2} + e^{-\min(\sqrt{2},2\varepsilon_1)k_1\bar{\sigma}_2} \\
		&+ \sum_{n=1}^{+\infty} \left( e^{-\sqrt{2}k_1\bar{\sigma}_2/2 - \gamma_0|n|} + e^{-2|n|\delta_1 k_1\min(\bar{\sigma}_1,\bar{\sigma}_2)} + e^{-2|n|\delta_2k_1\bar{\sigma}_1} \right)\\
		\lesssim& e^{-\gamma k_1\min(\bar{\sigma}_1,\bar{\sigma}_2)}.
		\end{align*}
		The estimates for the gradient of the difference follow the same argument.
	\end{proof}
\end{mytheorem}
Combining the estimate in Theorem \ref{Greenestimate} and the Green's
representations \eqref{Greenrep1} and \eqref{Greenrep2}, we immediately get the
result in Theorem \ref{convergthm}.
\section{Conclusion}
In this paper, we have shown that \cb{ the layered medium scattering problem
  with UPML truncation always possesses a unique solution with no constraints
  about wavenumber or PML absorbing strength}. As the PML absorbing strength
increases, the solution of the truncated problem converges to the solution of
the original scattering problem exponentially fast. The proof is based on the
construction of the Green's function for the layered medium with UPML
truncation. In particular, we show that the Green's function always exists
within the UPML, regardless of wavenumber and absorbing strength of the PML. Our
future work includes investigating the well-posedness for scattering problem
with obstacles, and the extension to multi-layered medium, as well as the
analysis to Maxwell's equations.

\noindent
{\bf Acknowledgements:}
\cb{WL would like to thank Prof. Ya Yan Lu of City University of
  Hong Kong for the motivation of this work. }

\bibliographystyle{plain}
\bibliography{reference}

\begin{thebibliography}{10}

\bibitem{baowu05}
G.~Bao and H.~Wu.
\newblock Convergence analysis of the perfectly matched layer problems for
  time-harmonic maxwell's equations.
\newblock {\em SIAM Journal on Numerical Analysis}, 43(5):2121--2143, 2005.

\bibitem{baochenwu05}
Gang Bao, Zhiming Chen, and Haijun Wu.
\newblock Adaptive finite-element method for diffraction gratings.
\newblock {\em Journal of the Optical Society of America. A, Optics, image
  science, and vision}, 22:1106--14, 07 2005.

\bibitem{Bao2018HuYin}
Gang Bao, Guanghui Hu, and Tao Yin.
\newblock Time-harmonic acoustic scattering from locally perturbed half-planes.
\newblock {\em SIAM Journal on Applied Mathematics}, 78:2672--2691, 10 2018.

\bibitem{Gang2011Imaging}
Gang Bao and Junshan Lin.
\newblock Imaging of local surface displacement on an infinite ground plane:
  The multiple frequency case.
\newblock {\em SIAM Journal on Applied Mathematics}, 71(5):1733--1752, 2011.

\bibitem{BERENGER1994185}
Jean-Pierre Berenger.
\newblock A perfectly matched layer for the absorption of electromagnetic
  waves.
\newblock {\em Journal of Computational Physics}, 114(2):185 -- 200, 1994.

\bibitem{brampas12}
H.~James Bramble and Joseph Pasciak.
\newblock Analysis of a {C}artesian {PML} approximation to the three
  dimensional electromagnetic wave scattering problem.
\newblock {\em International Journal of Numerical Analysis and Modeling}, 9, 05
  2012.

\bibitem{brampas10}
H.~James Bramble, Joseph Pasciak, and Dimitar Trenev.
\newblock Analysis of a finite {PML} approximation to the three dimensional
  elastic wave scattering problem.
\newblock {\em Math. Comput.}, 79:2079--2101, 10 2010.

\bibitem{Bruno2016Windowed}
Oscar Bruno, Mark Lyon, Carlos Perezarancibia, and Catalin Turc.
\newblock Windowed green function method for layered-media scattering.
\newblock {\em Siam Journal on Applied Mathematics}, 76(5):1871--1898, 2016.

\bibitem{caiyu}
W.~Cai and T.~J. Yu.
\newblock {Fast Calculations of Dyadic Green's Functions for Electromagnetic
  Scattering in a Multilayered Medium}.
\newblock {\em J. Comput. Phys.}, 165(1):1--21, 2000.

\bibitem{cai}
Wei Cai.
\newblock {\em Computational Methods for Electromagnetic Phenomena}.
\newblock Cambridge University Press, New York, NY, 2013.

\bibitem{ChenWu03}
Z.~Chen and H.~Wu.
\newblock An adaptive finite element method with perfectly matched absorbing
  layers for the wave scattering by periodic structures.
\newblock {\em SIAM Journal on Numerical Analysis}, 41(3):799--826, 2003.

\bibitem{chenzheng}
Z.~Chen and W.~Zheng.
\newblock Convergence of the uniaxial perfectly matched layer method for
  time-harmonic scattering problems in two-layered media.
\newblock {\em SIAM Journal on Numerical Analysis}, 48(6):2158--2185, 2010.

\bibitem{Chen2013}
Zhiming Chen, Tao Cui, and Linbo Zhang.
\newblock An adaptive anisotropic perfectly matched layer method for 3-d time
  harmonic electromagnetic scattering problems.
\newblock {\em Numerische Mathematik}, 125(4):639--677, 12 2013.

\bibitem{cheliu05}
Zhiming Chen and Xuezhe Liu.
\newblock An adaptive perfectly matched layer technique for time-harmonic
  scattering problems.
\newblock {\em SIAM J. Numer. Anal.}, 43(2):645--671, 2005.

\bibitem{chenxi15}
Zhiming Chen, Xueshuang Xiang, and Xiaohui Zhang.
\newblock Convergence of the {PML} method for elastic wave scattering problems.
\newblock {\em Mathematics of Computation}, 85:2687–2714, 09 2016.

\bibitem{Chen2017PML}
Zhiming Chen and Weiying Zheng.
\newblock {PML} method for electromagnetic scattering problem in a two-layer
  medium.
\newblock {\em SIAM Journal on Numerical Analysis}, 55(4):2050--2084, 2017.

\bibitem{cjm97}
W.~Chew, J.~Jin, and E.~Michielssen.
\newblock Complex coordinate stretching as a generalized absorbing boundary
  condition.
\newblock {\em Microwave Opt. Technol. Lett.}, 15:363--369, 1997.

\bibitem{Chew90}
W.~C. Chew.
\newblock {\em Waves and Fields in Inhomogeneous Media}.
\newblock 1990.

\bibitem{CollMonk}
F.~Collino and P.~Monk.
\newblock The perfectly matched layer in curvilinear coordinates.
\newblock {\em SIAM Journal on Scientific Computing}, 19(6):2061--2090, 1998.

\bibitem{CPHC98}
Pierre-Marie Cutzach and Christophe Hazard.
\newblock Existence, uniqueness and analyticity properties for electromagnetic
  scattering in a two‐layered medium.
\newblock {\em Mathematical Methods in the Applied Sciences}, 21:433 -- 461, 03
  1998.

\bibitem{Hohage2003Solving}
Thorsten Hohage, Frank Schmidt, and Zschiedrich Lin.
\newblock Solving time-harmonic scattering problems based on the pole
  condition: Theory.
\newblock {\em SIAM Journal on Mathematical Analysis}, 35(3):547--560, 2003.

\bibitem{Jiang2016Quantitative}
S.~Jiang and J.~Lai.
\newblock Quantitative study of the effect of cladding thickness on modal
  confinement loss in photonic waveguides.
\newblock {\em Optics Express}, 24(22):24872--24882, 2016.

\bibitem{koh}
Il-Suek Koh and Jong-Gwan Yook.
\newblock {Exact Closed-Form Expression of a Sommerfeld Integral for the
  Impedance Plane Problem}.
\newblock {\em IEEE Trans. Antennas Propag.}, 54(9):2568--2576, 2006.

\bibitem{KG1980}
G.~Kristensson.
\newblock A uniqueness theorem for helmholtz’ equation: Penetrable media with
  an infinite interface.
\newblock {\em SIAM Journal on Mathematical Analysis}, 11(6):1104--1117, 1980.

\bibitem{LAI2018359}
Jun Lai, Leslie Greengard, and Michael O'Neil.
\newblock A new hybrid integral representation for frequency domain scattering
  in layered media.
\newblock {\em Applied and Computational Harmonic Analysis}, 45(2):359 -- 378,
  2018.

\bibitem{Lassas1998On}
Matti Lassas and Erkki Somersalo.
\newblock On the existence and convergence of the solution of {PML} equations.
\newblock {\em Computing}, 60(3):229--241, 1998.

\bibitem{lassas_somersalo_2001}
Matti Lassas and Erkki Somersalo.
\newblock Analysis of the {PML} equations in general convex geometry.
\newblock {\em Proceedings of the Royal Society of Edinburgh: Section A
  Mathematics}, 131(5):1183–1207, 2001.

\bibitem{LiWu2019}
Y.~Li and H.~Wu.
\newblock Fem and cip-fem for helmholtz equation with high wave number and pml
  truncation.
\newblock {\em SIAM J. Numer. Anal.}, page to appear, 2019.

\bibitem{LuHu2019}
Wangtao Lu and Guanghui Hu.
\newblock Time-harmonic acoustic scattering from a nonlocally perturbed
  trapezoidal surface.
\newblock {\em SIAM Journal on Scientific Computing}, 41(3):B522--B544, 2019.

\bibitem{LuLu2012}
Wangtao Lu and Ya~Yan Lu.
\newblock Waveguide mode solver based on neumann-to-dirichlet operators and
  boundary integral equations.
\newblock {\em Journal of Computational Physics}, 231(4):1360--1371, 2012.

\bibitem{Lu2018Perfectly}
Wangtao Lu, Ya~Yan Lu, and Jianliang Qian.
\newblock Perfectly matched layer boundary integral equation method for wave
  scattering in a layered medium.
\newblock {\em SIAM Journal on Applied Mathematics}, 78(1):246--265, 2018.

\bibitem{lulusong18}
Wangtao Lu, Ya~Yan Lu, and Dawei Song.
\newblock A numerical mode matching method for wave scattering in a layered
  medium with a stratified inhomogeneity.
\newblock {\em SIAM Journal on Scientific Computing}, 41(2):B274--B294, 2019.

\bibitem{Odeh63}
F~M~Odeh.
\newblock Uniqueness theorems for the helmholtz equation in domains with
  infinite boundary.
\newblock {\em J. Math. Mech. 12, 857-868}, 12, 01 1963.

\bibitem{McLean2000}
William McLean.
\newblock {\em Strongly elliptic systems and boundary integral equations}.
\newblock Cambridge University Press, Cambridge, 2000.

\bibitem{NovoHecht06}
Lukas Novotny and Bert Hecht.
\newblock {\em Principles of Nano-Optics}.
\newblock Cambridge University Press, 2006.

\bibitem{okhmatovski-2014}
V.~I. Okhmatovski and A.~C. Cangellaris.
\newblock {Evaluation of layered media Greens functions via rational function
  fitting}.
\newblock {\em IEEE Microw. Wireles Comp. Lett.}, 14:22--24, 2004.

\bibitem{oneil-imped}
Michael O'Neil, Leslie Greengard, and Andras Pataki.
\newblock On the efficient representation of the half-space impedance {G}reen's
  function for the {H}elmholtz equation.
\newblock {\em Wave Motion}, 51:1--13, 2014.

\bibitem{paulus_2000}
M.~Paulus, P.~Gay-Balmaz, and O.~J.~F. Martin.
\newblock Accurate and efficient computation of the green’s tensor for
  stratified media.
\newblock {\em Physical Review E}, 62:5797--5807, 2000.

\bibitem{Perez-Arancibia2014}
Carlos P\'{e}rez-Arancibia and Oscar~P. Bruno.
\newblock {High-order integral equation methods for problems of scattering by
  bumps and cavities on half-planes}.
\newblock {\em Journal of the Optical Society of America A}, 31(8):1738, 2014.

\bibitem{sommerfeld}
A.~Sommerfeld.
\newblock Uber die ausbreitung der wellen in der drahtlosen telegraphie.
\newblock {\em Ann. {P}hys. {L}eipzig}, 28:665--737, 1909.

\bibitem{SongLu2015}
Dawei Song and Ya~Yan Lu.
\newblock Analyzing leaky waveguide modes by pseudospectral modal method.
\newblock {\em IEEE Photonics Technology Letters}, 27(9):955--958, 2015.

\bibitem{stein93}
E.~Stein.
\newblock {\em Harmonic analysis: real-variable methods, orthogonality, and
  oscillatory integrals}, volume 43 of Princeton Mathematical Series.
\newblock Princeton University Press, Princeton, NJ, 1993. With the assistance
  of Timothy S. Murphy, Monographs in Harmonic Analysis, III.

\bibitem{ZhouWu2018}
W.~Zhou and H.~Wu.
\newblock An adaptive finite element method for the diffraction grating problem
  with pml and few-mode dtn truncations.
\newblock {\em J. Sci. Comput.}, 76(3):1813--1838, 2018.

\end{thebibliography}
\end{document}